\documentclass[11pt]{amsart}
\usepackage{amsmath, amsthm, amscd, amssymb, enumerate, verbatim, graphicx}
\usepackage[all]{xy}

\theoremstyle{theorem}

\newtheorem{thm}{Theorem}[section]
\newtheorem{prop}{Proposition}[section]
\newtheorem{cor}{Corollary}[section]
\newtheorem{lem}{Lemma}[section]

\newtheorem*{thm_A}{Theorem A}
\newtheorem*{thm_B}{Theorem B}

\newtheorem*{thm_1}{Theorem I}
\newtheorem*{thm_2}{Theorem II}
\newtheorem*{thm_3}{Theorem III}
\newtheorem*{thm_4}{Theorem IV}
\newtheorem*{thm_5}{Theorem V}
\newtheorem*{thm_6}{Theorem VI}

\theoremstyle{definition}

\newtheorem{defn}{Definition}[section]
\newtheorem{example}{Example}[section]
\newtheorem{remark}{Remark}[section]
\newtheorem{setting}{Setting}[section]
\newtheorem{compl}{Compliment}[section] 
\newtheorem{convention}{Convention}[section] 
\newtheorem{fact}{Fact}[section]
\newtheorem*{notation}{Notation}

\newtheorem*{setting_3.2+}{Setting~\ref{setting_handle_decomp_dual-skeleton}$^+$}
\newtheorem*{setting_3.2++}{Setting~\ref{setting_handle_decomp_dual-skeleton}$^{++}$}
\newtheorem*{setting_3.3+}{Setting~\ref{setting_handle_even_cpt}$^+$}
\newtheorem*{setting_3.6+}{Setting~\ref{setting_triang_even_cpt}$^+$}
\newtheorem*{setting_5.2+}{Setting~\ref{setting_cpt_even}$^+$}
\newtheorem*{setting_5.3+}{Setting~\ref{setting_even_cpt_triang}$^+$}
\newtheorem*{setting_5.4+}{Setting~\ref{setting_even_cpt}$^+$}
\newtheorem*{setting_5.6+}{Setting~\ref{setting_even_open_handle}$^+$}
\newtheorem*{setting_5.7+}{Setting~\ref{setting_even_covering}$^+$}
\newtheorem*{setting_5.8+}{Setting~\ref{setting_inf_sum}$^+$}
\newtheorem*{setting_5.8++}{Setting~\ref{setting_inf_sum}$^{++}$}

\def \btab {\begin{tabular}}
\def \etab {\end{tabular}} 

\def \bary {\begin{array}}
\def \eary {\end{array}} 

\def \benum {\begin{enumerate}}
\def \eenum {\end{enumerate}} 

\def \bit {\begin{itemize}}
\def \eit {\end{itemize}} 

\def \itemi {\item[(i)\,]}
\def \itemii {\item[(ii)]}
\def \itemiii {\item[(iii)]}
\def \itemiv {\item[(iv)]}
\def \itemv {\item[(v)\,]}
\def \itemI {\item[(i)\ ]}
\def \itemII {\item[(ii)\,]}

\def \itema {\item[(a)]} 
\def \itemb {\item[(b)]} 
\def \itemc {\item[(c)]} 
\def \itemd {\item[(d)]}

\def \hsf {\hspace*{2mm}}
\def \hsh {\hspace*{5mm}}
\def \hsp {\hspace*{10mm}}
\def \hspp {\hspace*{20mm}}

\def \hspppp {\hspace*{40mm}}

\def \e {\varepsilon} 
\def \cal {\mathcal}
\def \phi {\varphi}

\def \ds {\displaystyle}
\newcommand{\id}{\mathrm{id}}

\def \Llra {\Longleftrightarrow}
\def \Lra {\Longrightarrow}
\def \llra {\longleftrightarrow}
\def \lra {\longrightarrow}
\def \LLRA {\ $\Llra$ \ } 
\def \LRA {\ $\Lra$ \ }

\renewcommand{\Phi}{\varPhi}
\renewcommand{\Psi}{\varPsi}
\renewcommand{\rho}{\varrho}

\newcommand{\IR}{\mathbb R}

\newcommand{\IZ}{\mathbb Z}

\newcommand{\cG}{\mathcal G}

\def \FSubset
{
\bary[t]{c}
\Subset \\[-1.8mm]
\mbox{\tiny $St {\cal F}$} 
\eary
}

\begin{document}
\baselineskip 6mm 

\title[]{The uniform perfectness of diffeomorphism groups \\ of open manifolds}

\author[Kazuhiko Fukui]{Kazuhiko Fukui} 
\address{228-168  Nakamachi, Iwakura, Sakyo-ku, Kyoto 606-0025, Japan}
\email{fukui@cc.kyoto-su.ac.jp}

\author[Tomasz Rybicki]{Tomasz Rybicki} 
\address{Faculty of Applied Mathematics, AGH University of Science and Technology, al. A. Mickiewicza 30, 30-059 Krak\'ow, Poland}
\email{tomasz.rybicki@agh.edu.pl}

\author[Tatsuhiko Yagasaki]{Tatsuhiko Yagasaki}
\address{Graduate School of Science and Technology, Kyoto Institute of Technology, Kyoto, 606-8585, Japan}
\email{yagasaki@kit.ac.jp}


\maketitle

\begin{abstract}{}
In this paper we study the uniform perfectness, boundedness and uniform simplicity of diffeomorphism groups of 
compact manifolds with boundary and open manifolds and 
obtain some upper bounds of their diameters with respect to commutator length, those with support in balls and conjugation-generated norm. 
\end{abstract}

\thispagestyle{empty}

\tableofcontents

\section{Introduction and statement of results}

The algebraic properties of diffeomorphism groups have been studied by Herman \cite{H1}, Thurston \cite{Th}, Mather \cite{Ma}, Epstein \cite{Ep} and others. These fundamental results are about the perfectness and simplicity of groups of diffeomorphisms on closed manifolds (compact, without boundary) or groups of diffeomorphisms with compact support on open manifolds (non-compact, without boundary). 
For a $\sigma$-compact (separable metrizable) $C^\infty$ $n$-manifold $M$ without boundary, let ${\rm Diff}^r(M)_0$ denote the group of all $C^r$ diffeomorphisms of $M$ which are isotopic to the identity and 
${\rm Diff}^r_c(M)_0$ denote the subgroup of all elements of ${\rm Diff}^r(M)$ which are isotopic to the identity by compactly supported isotopies.  
It is well-known  that for $1 \le r \le \infty$,\, $r \neq n +1$ 
the group ${\rm Diff}^r_c(M)_0$ is perfect, i.e., it coincides with its commutator subgroup, and is simple provided $M$ is connected.

McDuff \cite{Mc} studied the lattice of normal subgroups of the diffeomorphism group of an open manifold $M$ which is  the interior of a compact manifold with nonempty boundary and showed that ${\rm Diff}^r(M)_0$ 
is a perfect group,  
using the fact due to Ling \cite{L} and Schweitzer, cf.\,\cite{S}, that certain quotient groups of the group by normal subgroups are simple groups. 

For an element $g$ of a group $G$ its commutator length $cl_G (g)$  
is defined as the minimum number of commutators whose product is equal to $g$ for $g \in [G,G]$ 
and as $cl_G (g) = \infty$ for $g  \in G - [G,G]$. 
The commutator length diameter $cld\,G$ of $G$ is defined by $cld\,G = \sup_{g \in G} cl_G(g)$. 
A group $G$ is said to be {\it uniformly perfect} if $cld\,G < \infty$, that is, 
any element of $G$ can be written as a product of a bounded number of commutators. 

Recall that a group $G$ is (i) bounded if any conjugation-invariant norm on $G$ is bounded and 
(ii) uniformly simple (\cite{Tsu09}) if  there is a positive integer $k$ such that 
for all $f,g\in G$ with $g\neq e$, $f$ can be written as a product of at most $k$ conjugates of $g$ or $g^{-1}$. 
Note that every uniformly simple group is simple and bounded and 
that any bounded perfect group is uniformly perfect. 

The problem of the uniform perfectness of diffeomorphism groups has been studied by Burago, Ivanov and Polterovich \cite{BIP}, Tsuboi \cite{Tsuboi2}, \cite{Tsuboi3} and others. They showed 
that the diffeomorphism groups are uniformly perfect for all closed manifolds except for two and four dimensional cases. 

\begin{thm_A}{\rm (\cite{BIP}, \cite{Tsuboi2, Tsuboi3})} Suppose $M$ is a closed $n$-manifold and $1 \le r \le \infty$,\, $r \neq n +1$. 
\benum 
\item[$(1)$] In the case $n =2m+1$ $(m \geq 0)$ : \hsh $cld\,{\rm Diff}^r(M)_0 \leq 4$. 
\item[$(2)$] In the case $n=2m$ $(m \geq 1)$ : 
\bit 
\itemI $cld\,{\rm Diff}^r(M)_0 \leq 3$ if $M$ admits a handle decomposition without $m$-handles.  
\itemII $cld\,{\rm Diff}^r(M)_0 \leq 4k +11$ if $m \geq 3$ and 
$M$ has a $C^\infty$ triangulation with at most $k$ $m$-simplices. 
\eit 
\eenum 
\end{thm_A}

In \cite{Tsuboi2, Tsuboi3} the estimates of the commutator length diameters in the above theorem were given as 
$cld \, {\rm Diff}^r(M) \leq 5$ in (1) and $cld \, {\rm Diff}^r(M) \leq 4$ in (2)(i). 
However, in his lecture around 2011, Tsuboi announced that 
these estimates are improved to those in the above theorem. 

The uniform simplicity of diffeomorphism groups has been studied by Tsuboi \cite{Tsu09, Tsuboi3}, 
based on the estimates on the commutator length supported in finite disjoint union of balls $(clb^f)$ and the conjugation-generated norm $\nu$ (cf.~Sections 2.3, 6.1). 

\begin{thm_B}{\rm (\cite{Tsu09, Tsuboi3})} 
Suppose $M$ is a closed connected $n$-manifold $(n \geq 1)$ and $1 \le r \le \infty$,\, $r \neq n +1$. \\ 
${\rm Diff}^r(M)_0$ is uniformly simple in the following cases : 
\benum 
\item[$(1)$] $n \neq 2,4$. 
\item[$(2)$] $n = 2m$ and $M$ admits a handle decomposition without $m$-handles.  
\eenum 
\end{thm_B}

The problem of the uniform perfectness and boundedness of groups of equivariant diffeomorphisms, 
leaf-preserving diffeomorphisms or diffeomorphisms which preserve some submanifolds, etc, 
has been also studied in \cite{A-F1},  \cite{A-F2},  \cite{F2},  \cite{R2}, \cite{LMR}, etc.

In this paper we extend the strategy in \cite{BIP}, \cite{Tsuboi2, Tsu09, Tsuboi3} and 
study the uniform perfectness, boundedness and uniform simplicity of the diffeomorphism groups of compact manifolds with boundary and open manifolds. 
Our main results are listed below. 

\begin{thm_1}[Theorems~\ref{thm_cpt_bdry_odd}, ~\ref{thm_open_odd}]
Suppose $n = 2m+1$ $(m \geq 0)$ and $1 \leq r \leq \infty$, $r \neq n+1$. 
\benum 
\item[$(1)$] If $M$ is a compact $n$-manifold possibly with boundary, then \ \ $cld\,{\rm Diff}^r(M, \partial)_0 \leq 4$.  

\item[$(2)$] If $M$ is an open $n$-manifold, then \ \ $cld\,{\rm Diff}^r(M)_0 \leq 8$ \ and \ $cld\,{\rm Diff}_c^r(M)_0 \leq 4$. 
\eenum 
\end{thm_1}

Here, ${\rm Diff}^r(M, \partial)_0$ is the group of $C^r$ diffeomorphisms which are isotopic to $\id_M$ rel a neighborhood of $\partial M$. 
See Section 2 for the precise definition. 

\begin{thm_2}[Corollaries~\ref{cor_cpt_bdry_even}, ~\ref{cor_closed}, Propositions~\ref{prop_2m_open_no-m-h}, ~\ref{prop_2m_open_finite-m-h}] \mbox{} \\ 
Suppose $n = 2m$ $(m \geq 1)$ and $1 \leq r \leq \infty$, $r \neq n+1$. 
\begin{enumerate}
\item[$(1)$] 
Suppose $M$ is a compact $n$-manifold possibly with boundary and $m \geq 3$. \\ 
\hsp $cld\,{\rm Diff}^r(M, \partial)_0 \leq 2k+7$ \\
\hsp \hsh if $M$ has a $C^\infty$ triangulation ${\mathcal T}$ whcih includes at most $k$ $m$-simplices not in $\partial M$. 

\item[$(2)$] Suppose $M$ is an $n$-manifold without boundary. 
\benum 
\item[{\rm (i)}\ ] $cld\,{\rm Diff}^r(M)_0 \leq 6$ \ and \ $cld\,{\rm Diff}_c^r(M)_0 \leq 3$ \\
\hsh if $M$ admits a handle decomposition without $m$-handles. 

\item[{\rm (ii)}\,] $cld\,{\rm Diff}^r(M)_0 \leq 2k +10$ \ and \ $cld\,{\rm Diff}_c^r(M)_0 \leq 2k + 7$ \\
\hsh if \btab[t]{@{}l}
$m \geq 3$ and 
$M$ admits a handle decomposition ${\cal H}$ such that 
${\cal H}$ includes 
at most \\[1.5mm] 
$k$ handles of index $m$ and 
each closed $m$-cell of the core complex $P_{\cal H}$ for ${\cal H}$ has \\[1.5mm]
the strong displacement property for the $m$-skeleton of $P_{\cal H}$. 
\etab 
\eenum 
\end{enumerate}
\end{thm_2} 

See Section 2.2 for definitions of the displacement property.  
The estimates in Theorem II\,(2)(ii) reduce to rough estimates 
$3k+8$ and $3k+5$ respectively, 
when each closed $m$-cell of $P_{\cal H}$ has 
the strong displacement property for only itself in $M$ (Proposition~\ref{prop_2m_open_finite-m-h}). 
These results induce fine estimates even in the case of closed manifolds (cf.~Theorem A\,(2)(ii)). 
For example, if $M$ is the product of two $m$-spheres $(m \geq 3)$, then it has the natural product handle decomposition ${\cal H}$ and
both of two closed $m$-cells of $P_{\cal H}$ have the strong displacement property for itself in $M$. 
This implies $cld\,{\rm Diff}^r(M)_0 \leq 11$ (Example~\ref{exp_even_covering}).  

When an open $2m$-manifold $M$ has infinitely many $m$-handles, 
if these $m$-handles admit a grouping into finitely many classes with appropriate displacement property, 
then we can induce some estimate on $cld\,{\rm Diff}^r(M)_0$ 
(Proposition~\ref{prop_exhausting seq}, Theorem~\ref{thm_even_open_handle}). 
In particular, we obtain some estimates on the commutator length diameter 
for covering spaces of closed $2m$-manifolds and infinite sums of finitely many compact $2m$-manifolds. 

\begin{thm_3}[Corollary~\ref{cor_even_covering}] 
Suppose $\pi : \widetilde{M} \to M$ is a $C^\infty$ covering space over 
a closed $2m$-manifold $M$ $(m \geq 3)$ and $1 \leq r \leq \infty$, $r \neq 2m+1$. 
If $M$ has a $C^\infty$ triangulation with at most $k$ $m$-simplices, then \\
\hspp \hsf $cld\,{\rm Diff}^r(\widetilde{M})_0 \leq 4k+14$ \ \ and \ \ $cld\,{\rm Diff}_c^r(\widetilde{M})_0 \leq 2k+7$. 
\end{thm_3} 

We also obtain a finer estimate of $cld\,{\rm Diff}^r(\widetilde{M})_0$ based on a handle decomposition of $M$ with 
some displacement property (Theorem~\ref{thm_even_covering}).  

\begin{thm_4}[Corollary~\ref{cor_inf_sum}, Example~\ref{exp_inf_conn_sum}] 
Suppose $M$ is any infinite connected sum 
$\bary[c]{@{}c@{\,}}
\mbox{\scriptsize $\infty$} \\[-1mm] 
\# \\[-1.5mm] 
\mbox{\scriptsize $i\!=\!1$}
\eary
 N$ of a closed $2m$-manifold $N$ $(m \geq 3)$ and 
$1 \leq r \leq \infty$, $r \neq 2m+1$. 
Then, $cld\,{\rm Diff}^r(M)_0 < \infty$ \ and \ $cld\,{\rm Diff}_c^r(M)_0 < \infty$. 
\end{thm_4}

In Proposition~\ref{prop_inf_sum} we treat a more general class of infinite sums of finitely many compact $2m$-manifolds
and provide with some explicit estimates on $cld\,{\rm Diff}^r(M)_0 $ and $cld\,{\rm Diff}_c^r(M)_0$. 

The following is an extension of the results of McDuff \cite{Mc} and the second author \cite{R2}.

\begin{thm_5}[Proposition~\ref{prop_product_end}] 
Suppose $M$ is the interior of a compact $n$-manifold $W$ with boundary and $1 \le r \le \infty$, $r \neq n +1$. 
Then $cld\, {\rm Diff}^r(M)_0 \leq \max \{cld\,{\rm Diff}^r(W, \partial)_0, 2\} + 2$. 
In particular, ${\rm Diff}^r(M)_0$ is uniformly perfect for $n \neq 2,4$. 
\end{thm_5}

For the boundedness and uniform simplicity of diffeomorphism groups we obtain the following results. 

\begin{thm_6}[Corollary~\ref{cor_bounded}] 
\mbox{} Suppose $1 \leq r \leq \infty$, $r \neq n+1$.
\benum  
\item[{\rm [1]}]  Suppose $M$ is a compact connected $n$-manifold possibly with boundary and $n \neq 2, 4$. 
Then, \hsh ${\rm Diff}^r(M, \partial)_0$ is uniformly simple. 

\item[{\rm [2]}] Suppose $M$ is a connected open $n$-manifold. Then, \\ 
\hsh ${\rm Diff}^r(M)_0$ is bounded \ and \ ${\rm Diff}^r_c(M)_0$ is uniformly simple \ \ 
in the following cases : 
\bit 
\item[{\rm (1)}] $n = 2m+1$ $(m \geq 0)$
\item[{\rm (2)}] $n = 2m$ $(m \geq 1)$ and $M$ satisfies one of the following conditions: 
\bit 
\itemI $M$ has a handle decomposition without $m$-handles. 
\eit 
\item[]  for $m \geq 3$

\bit 
\itemII $M$ has a handle decomposition ${\cal H}$ with only finitely many $m$-handles  
and for which the closure of each open $m$-cell of $P_{\cal H}$ is strongly displaceable from itself in $M$. 
\itemiii $M$ is a covering space over a closed $2m$-manifold. 
\itemiv $M$ is an infinite sum of finitely many compact $2m$-manifolds $($as in Setting~{\rm \ref{setting_inf_sum_nu}}$)$.
\eit 
\eit 

\vskip 2mm 

\item[{\rm [3]}] Suppose $M$ is the interior of a compact $n$-manifold $W$ with boundary. Then, \\ 
\hsh ${\rm Diff}^r(M)_0$ is bounded, if ${\rm Diff}^r(W, \partial)_0$ is bounded.  
\eenum 
\end{thm_6}

Extending the arguments in \cite{Tsu09}, 
we introduce the commutator length supported in discrete union of balls ($clb^d$) 
as a non-compact version of $clb^f$, and discuss the relation between the conjugation-generated norm $\nu$ and $clb^f$, $clb^d$ (Lemma~\ref{lem_nu_leq_4clb}).  
Since the norms $cl$, $clb^f$ and $clb^d$ are closely related, we provide with the estimates on these three commutator lengthes
simultaneously in all steps of Sections 2\,$\sim$\,5. The estimates on $\nu$ are induced from those on $clb^f$ and $clb^d$ (Section 6.2).
Then, Theorem VI follows from the boundedness of $\nu$. 

In this paper we have tried to find finer estimates on the commutator lengthes $cl$, $clb^f$ and $clb^d$, 
though this causes complicated expressions in some statements. 
This is  because we are interested in the lower bounds of these norms and also in 
the relations between the precise values of these norms and the topology of the manifolds. 

This paper is organized as follows. 
In Sections 2.1, 2.2 we fix basic notations used in this paper and recall generalities on extended conjugation-invariant norms on groups. 
In Section 2.3 we introduce the commutator length supported in discrete union of balls ($clb^d$) as a noncompact version of $clb^f$. 
In Section 2.4 we separate the notions of absorption and displacement properties for compact subsets in a manifold from the related arguments. 
We clarify their basic behavior and roles in the basic factorization lemma on diffeomorphisms in \cite{BIP, Tsuboi2, Tsu09} and 
extend the factorization lemma to the case of manifolds with boundary.
Section 2.5 includes some observation on factorization of isotopies on noncompact manifolds

Section 3 is devoted to some basic arguments on handle decompositions and triangulations of manifolds, 
which are necessary to extend the strategy in \cite{Tsuboi2, Tsu09, Tsuboi3} to the cases of compact manifolds with boundary and open manifolds. 
See an introduction in Section 3 for an explanation on this issue. 
In Section 3.1 we treat $C^\infty$ handle decompositions. 
Our device here is that, given a handle decomposition ${\cal H}$ and its dual ${\cal H}^\ast$ of an open $n$-manifold $M$, 
instead of a single compact $n$-submanifold $N$ of $M$ 
we use a pair $N_1 \subset N_2$ of compact $n$-submanifolds of $M$ such that $N_1$ is ${\cal H}$-saturated and 
$N_2$ is ${\cal H}^\ast$-saturated. 
Section 3.2 includes some arguments on $C^\infty$ triangulations of compact manifolds with boundary. 
In the bounded case $C^\infty$ triangulations are important since they naturally induce flows between some subcomplexes and their duals which preserve the boundary of manifolds.
Here we use the cylinder structures between complimentary full subcomplexes rather than the induced handle decompositions, 
since the description of the former is more direct and simpler. 

Section 4 includes main results in the odd-dimensional case. 
In Section 4.1 we recall the basic procedure in \cite{BIP, Tsuboi2}
for removing crossing points on the track of an isotopy and a factorization of an isotopy.  
We treat the compact case in Section 4.2 and the open case in Section 4.3,  
based on the basic procedure and the preliminary lemmas for the bounded/open cases in Section 3. 

Section 5 includes main results in the even-dimensional case. 
In Section 5.1 we recall the basic procedure in \cite{Tsuboi3}
for the Whitney trick on the track of an isotopy and a factorization of an isotopy, 
together with some improvements for finer estimates of $cl$ and $clb^f$ in our setting. 
In Section 5.2 we discuss the compact case based on the basic procedure  
and the preliminary lemmas for the bounded/open cases in Section 3. 
Section 5.3 is devoted to the open case. 
Section 5.3.1 includes some general results on grouping of $m$-cells in $2m$-manifolds. 
In Section 5.3.2 we treat three classes: (i) open $2m$-manifolds with only finitely many $m$-handles, 
(ii) covering spaces of closed $2m$-manifolds and
(iii) infinite sum of finitely many compact $2m$-manifolds. 
The results on grouping and the case (i) follow from the compact results in Section 5.2 and some basic lemmas in Section 3, 
while the cases (ii), (iii) follow from the results on grouping directly. 

In the final section 6 we obtain the estimates on the conjugation-generated norm $\nu$ in the bounded/open cases.   
First we discuss the relation between $\nu$ and $clb^f$, $clb^d$ as an extension of the arguments in \cite{Tsu09}. 
Then, the estimates on $\nu$ follow from those on $clb^f$ and $clb^d$ in Sections 2 $\sim$ 5. 
The conclusions on boundedness and uniform simplicity of diffeomorphism groups in the bounded/open cases 
follow immediately from the boundedness of $\nu$ according to the basic facts on conjugation-generated norms. 

\section{Preliminaries}

\subsection{Notations} \mbox{} 

For a subset $A$ of a topological space $X$ the symbols ${\rm Int}_XA$ and $Cl_XA$ denote the topological interior and closure of $A$ in $X$. 
For subsets $A$, $B$ of $X$ 
the symbol $A \Subset B$ means $A \subset {\rm Int}_XB$ and the symbol $A_B$ denotes the subset $A - {\rm Int}_X B$. 
The symbols ${\cal O}(X)$, ${\cal F}(X)$ and ${\cal K}(X)$ denote 
the collections of open subsets, closed subsets and compact subsets of $X$ respectively. 
The symbol ${\cal C}(X)$ denotes the collection of connected components of $X$. 

For a family ${\cal F} = \{ F_\lambda \}_{\lambda \in \Lambda}$ of subsets of a topological space $X$, let $|{\cal F}| := \bigcup_{\lambda \in \Lambda} F_\lambda$ 
and $St(A, {\cal F}) := A \cup \big(\bigcup \{ F_\lambda \mid \lambda \in \Lambda, A \cap F_\lambda \neq \emptyset \}\big)$ for $A \subset X$. 
We say that ${\cal F}$ is discrete in $X$ if any point of $X$ has a neighborhood $U$ in $X$ which intersects at most one $F_\lambda$ (i.e., 
$\#\{ \lambda \in \Lambda \mid U \cap F_\lambda \neq \emptyset \}\leq 1$).
When ${\cal F}$ is discrete in $X$, the following holds : 
(i) The family $\{ Cl_X F_\lambda \}_{\lambda \in \Lambda}$ is also discrete in $X$.  
(ii) $|{\cal F}| \in {\cal F}(X)$ if $F_\lambda \in {\cal F}(X)$ $(\lambda \in \Lambda)$. 
(iii) If $X$ is a paracompact Hausdorff space, then there exists a discrete family ${\cal G} = \{ G_\lambda \}_{\lambda \in \Lambda}$ in $X$ 
such that $F_\lambda \Subset G_\lambda$ $(\lambda \in \Lambda)$. 

In this paper an $n$-manifold means  
a $\sigma$-compact (separable metrizable) $C^\infty$ manifold of dimension $n$. 
A closed/open manifold means a compact/non-compact manifold without boundary.  

Suppose $M$ is a $C^\infty$ $n$-manifold possibly with boundary. 
By ${\mathcal S}{\mathcal M}(M)$ we denote the collection of $C^\infty$ $n$-submanifolds of ${\rm Int}\,M$ which are closed in $M$. 
Let ${\mathcal S}{\mathcal M}_c(M) := {\mathcal S}{\mathcal M}(M) \cap {\cal K}(M)$. 
We regard $\emptyset$ as an element of ${\mathcal S}{\mathcal M}_c(M)$.
If $L, N \in {\mathcal S}{\mathcal M}(M)$ and $L \Subset N$, then 
$N_L = N -  {\rm Int}_M L \in {\mathcal S}{\mathcal M}(M)$.
A subcollection ${\cal R}$ of ${\mathcal S}{\mathcal M}_c(M)$ is said to be cofinal if for any $L \in {\mathcal S}{\mathcal M}_c(M)$ there exists $N \in{\cal R}$ with $L \subset N$. 

An exhausting sequence in $M$ is a sequence $M_i \in {\mathcal S}{\mathcal M}_c(M)$ $(i \geq 1)$ such that 
$M_i \Subset M_{i+1}$ $(i \geq 1)$ and $M = \bigcup_{i \geq 1} M_i$. 
Such a sequence always exists, since $M$ is $\sigma$-compact. 
For an exhausting sequence $\{ M_i \}_{i\geq 1}$ 
we use the following notations : $M_{i} := \emptyset$ for $i \leq 0$ (otherwise specified) and  
$M_{i,j} := (M_j)_{M_i} \in {\mathcal S \mathcal M}_c(M)$ $(-\infty < i < j < \infty)$. 

For the manifold $M$, the symbol $\widetilde{M}$ denotes the $C^\infty$ $n$-manifold $\widetilde{M} := M \cup (\partial M \times [0,1))$ and 
for any $O \in {\cal O}(M)$ the symbol $\widetilde{O} \equiv O^\sim$ denotes 
$\widetilde{O} := O \cup (\partial O \times [0,1)) \in {\cal O}(\widetilde{M})$. 

For the simplicity of notation we always use the symbol $I$ to denote the interval $[0,1]$. 
A $C^r$ isotopy on $M$ is a $C^r$ diffeomorphism $F : M \times I \to M \times I$ which preserves the $I$-factor, that is, 
it takes the form $F(x,t) = (F_t(x), t)$ $((x,t) \in M \times I)$. 
For a subset $A$ of $M$ the track of $A$ under $F$ is the subset $\cup_{t \in I} F_t(A) = \pi_M F(A \times I)$, where $\pi_M : M \times I \to M$ is the projection onto $M$. 
The support of an isotopy $F$ is defined by ${\rm supp}\,F := Cl_M \big(\bigcup_{t \in I} {\rm supp}\,F_t\big)$. 
Note that an isotopy $F$ on $M$ has compact support if and only if there exists $K \in {\cal K}(M)$ 
such that $F = \id$ on $(M-K) \times I$. 

The symbols ${\rm Diff}^r(M)$ and ${\rm Diff}^r_c(M)$ $(1 \leq r \leq \infty)$ denote the group of all $C^r$ diffeomorphisms of $M$ and 
its subgroup consisting of those with compact support. Similarly, 
the symbols ${\rm Isot}^r(M)$ and ${\rm Isot}^r_c(M)$ denote the group of all $C^r$ isotopies of $M$ and 
its subgroup consisting of those with compact support. 
There exists a natural group homomorphism $R : {\rm Isot}^r(M) \to {\rm Diff}^r(M)$ : $R(F) = F_1$. 
Any subgroup $\cG$ of ${\rm Isot}^r(M)$ induces the subgroup $R(\cG)$ of ${\rm Diff}^r(M)$. 
For a subset $C$ of $M$ we have the following subgroups of the groups ${\rm Isot}^r(M)$ and ${\rm Diff}^r(M)$:  \\[1mm] 
\hspace*{3mm} ${\rm Isot}^r(M; C)_0 = \{ F \in {\rm Isot}^r(M) \mid F_0 = \id_M, \mbox{$F = \id$ on $U \times I$ for some neighborhood $U$ of $C$ in $M$} \}$, \\[0.5mm] 
\hspace*{3mm} ${\rm Isot}^r_c(M; C)_0 = {\rm Isot}^r(M; C)_0 \cap {\rm Isot}^r_c(M)$, \\[0.5mm]
\hspace*{3mm} ${\rm Diff}^r(M; C)_0 = R({\rm Isot}^r(M; C)_0\big)$, 
\hspace*{5mm} ${\rm Diff}^r_c(M; C)_0 = R({\rm Isot}^r_c(M; C)_0\big)$. \\ 
Note that ${\rm Diff}^r_c(M; C)_0 \subsetneqq {\rm Diff}^r(M; C)_0 \cap {\rm Diff}^r_c(M; C)$ in general. 
When $C = \partial M$, we simply write as ${\rm Diff}^r(M; \partial)_0$, etc. Under this notation, for a compact manifold $W$ the restriction map 
${\rm Diff}^r(W; \partial)_0 \to {\rm Diff}^r_c({\rm Int}\,W)_0$ is a group isomorphism.
For notational simplicity, when $r = \infty$, we usually omit the superscript $\infty$ from the above notations. 

\subsection{Extended conjugation-invariant norms} \mbox{} 

Suppose $G$ is a group. 
For simplicity we use the notation $b^a := aba^{-1}$ for $a,b \in G$. 
The next formula is useful in the change of order of factors in multiplications. 

\begin{fact}\label{fact_conjugation}  
$ab = b^aa$ \ \ and more generally \ \  
$a_1b_1 \cdots a_s b_s
= b_1^{a_1}b_2^{a_1a_2} b_3^{a_1a_2a_3} \cdots b_s^{a_1 \cdots a_s} (a_1 \cdots a_s)$ \ in $G$. 
\end{fact} 

We recall basic facts on conjugation-invariant norms \cite{BIP, Tsu09}. 
An extended conjugation-invariant norm on $G$ is a function $q : G\to [0,\infty]$ which satisfies the following conditions~: \ \ 
for any $g, h\in G$
\bit 
\itemi $q(g)=0$ iff $g=e$ \hsh 
(ii) $q(g^{-1})=q(g)$ \hsh 
(iii) $q(gh)\leq q(g)+ q(h)$ \hsh 
(iv) $q(hgh^{-1})= q(g)$. 
\eit 
A conjugation-invariant norm on $G$ is an extended conjugation-invariant norm on $G$ with values in $[0, \infty)$. 
(Below we abbreviate ``an (extended) conjugation-invariant norm'' to an (ext.) conj.-invariant norm.) 

Note that for any ext.~conj.-invariant norm $q$ on $G$ the inverse image  
$N := q^{-1}([0, \infty))$ is a normal subgroup of $G$ and $q|_N$ is a conj.-invariant norm on $N$. Conversely, 
if $N$ is a normal subgroup of $G$ and $q : N \to [0, \infty)$ is conj.-invariant norm on $N$, then 
its trivial extension $q : G \to [0, \infty]$ by $q = \infty$ on $G - N$ is an ext.~conj.-invariant norm on $G$. 
For any ext.~conj.-invariant norm $q$ on $G$, 
the $q$-diameter of a subset $A$ of $G$ is defined by \ $q\hspace{0.2mm}d\,A := \sup\,\{ q(g) \mid g \in A \}$.

\begin{example}\label{exp_norm} (Basic construction)

Suppose $S$ is a subset of $G$. Its normalizer is denoted by $N(S)$.  
If $S$ is symmetric ($S = S^{-1}$) and conjugation-invariant ($gSg^{-1} = S$ for any $g \in G$), then 
$N(S) = S^\infty := \bigcup_{k=0}^\infty S^k$ and 
the ext.~conj.-invariant norm \  
$q_{(G,S)} : G \lra {\Bbb Z}_{\geq 0} \cup \{ \infty \}$ \ is defined by \\[2mm] 
\hspace*{20mm} $q_{(G,S)}(g) := 
\left\{ \hspace{-1mm}
\bary[c]{l}
\min \{ k \in {\Bbb Z}_{\geq 0} \mid \text{$g = g_1 \cdots g_k$ for some $g_1, \cdots, g_k \in S$}\}
\hsh (g \in N(S)), \\[2mm] 
\infty \hsh (g \in G - N(S)).
\eary \right.$ \\[2mm] 
Here, the empty product $(k=0)$ denotes the unit element $e$ in $G$ and $S^0 = \{ e \}$. 
\end{example}

\begin{example}\label{exp_norm} (Commutator length)  

The symbol $G^c$ denotes the set of commutators in $G$. 
Since $G^c$ is symmetric and conjugation-invariant in $G$, 
we have $[G, G] = N(G^c) = (G^c)^\infty$ and obtain the ext.~conj.-invariant norm \ 
$q_{(G,G^c)} : G \lra {\Bbb Z}_{\geq 0} \cup \{ \infty \}$. 
We denote $q_{(G,G^c)}$ by $cl_G$ and call it the commutator length of $G$. 
The diameters $q_{(G,G^c)}d\,G$ and $q_{(G,G^c)}d\,A$ $(A \subset G)$ are denoted by 
$cld\,G$ and $cld_G\,A \equiv cld\,(A, G)$ and called the commutator length diameter of $G$ and $A$ respectively. 
Sometimes we write $cl(g) \leq k$ in $G$ instead of $cl_G(g) \leq k$. 

A group $G$ is perfect if $G = [G, G]$, that is, any element of $G$ is written as a product of commutators in $G$.
We say that a group $G$ is uniformly perfect if $cld\,G < \infty$, that is, 
any element of $G$ is written as a product of a bounded number of commutators in $G$.

More generally, suppose $S$ is a subset of $G_c$ and it is symmetric and conjugation-invariant in $G$.  
Then, $N(S) \subset [G,G]$ and we obtain $q_{(G,S)}$, which is denoted by $cl_{(G, S)}$ 
and is called the commutator length of $G$ with respect to $S$. 
The diameters $q_{(G,S)}d\,G$ and $q_{(G,S)}d\,A$ $(A \subset G)$ are denoted by 
$cld_S\,G$ and $cld_{(G,S)}\,A \equiv cld_S\,(A, G)$ respectively. 
\end{example} 

\subsection{Commutator length of diffeomorphism groups} \mbox{} 

In this paper we are concerned with the commutator length on diffeomorphism groups. 
In Section 6 we also study the boundedness and uniform simplicity of diffeomorphism groups, 
where ``the commutator length supported in balls'' plays an important role. 
Since these commutator lengthes are closely related, in Sections 2 $\sim$ 5 we treat them simultaneously. 

We start with the definition of the commutator length supported in balls. 
Suppose $M$ is an $n$-manifold possibly with boundary, $1\leq r \leq \infty$ and $C \in {\cal F}(M)$. 

\begin{notation}\label{notation_B}
Let ${\cal B}^r(M, C)$, ${\cal B}^r_f(M, C)$ and ${\cal B}^r_d(M, C)$ denote the collections of 
(i) all $C^r$ $n$-balls in ${\rm Int}\,M - C$,  
(ii) all finite disjoint unions of $C^r$ $n$-balls in ${\rm Int}\,M - C$ and 
(iii) all discrete unions of $C^r$ $n$-balls in ${\rm Int}\,M - C$ which are closed in $M$, respectively. 
The condition (iii) is restated as (iii)$'$ all discrete unions of $C^r$ $n$-balls in $M$ which are included in ${\rm Int}\,M - C$. 
(cf. the 2nd paragraph in Section 2.1)
\end{notation}

Consider the group $G \equiv {\rm Diff}^r(M, \partial M \cup C)_0$. 
In each case where ${\cal B} = {\cal B}^r(M, C)$, ${\cal B}^r_f(M, C)$ or ${\cal B}^r_d(M, C)$, 
consider the subset $S$ of $G^c$ defined by \\
\hspace*{50mm} $S := \bigcup \{ {\rm Diff}^r(M, M_D)_0^c \mid D \in {\cal B} \}$. \\[1mm] 
Then, $S$ is symmetric and conjugation-invariant in $G$. 
Therefore, we obtain the commutator length $cl_{(G,S)}$ of $G$ with respect to $S$. 
Depending on the selection of ${\cal B}$, we denote this commutator length by $clb$, $clb^f$ or $clb^d$ and 
the commutator length diameter by $clbd$, $clb^fd$ or $clb^dd$, respectively. 

\begin{remark}\label{rmk_clb=clb^f} \mbox{}
\benum
\item Since $clb \geq clb^f \geq clb^d$ in $G$ since ${\cal B}^r(M, C) \subset {\cal B}^r_f(M, C) \subset {\cal B}^r_d(M, C)$. 
\item \bit 
\itemI If $M - C$ is connected, then $clb = clb^f$ in $G$, since   
${\rm Int}\,M - C$ is also connected and any $D \in {\cal B}^r_f(M, C)$ is included in some $E \in {\cal B}^r(M, C)$. 
\itemII If $M - C$ is relatively compact in $M$, then ${\cal B}^r_f(M, C) = {\cal B}^r_d(M, C)$ and 
$clb^f = clb^d$ in $G$. 
\eit 
\item For $S_f^r := \bigcup \{ {\rm Diff}^r(M, M_D)_0^c \mid D \in {\cal B}^r_f(M, C) \}$, the following holds. \\
(i) $N(S_f^r)\subset {\rm Diff}^r_c(M, \partial M \cup C)_0$ \ 
and \ (ii) $N(S_f^r)= {\rm Diff}^r_c(M, \partial M \cup C)_0$ when $r \neq n+1$. 

\eenum
\end{remark}

In the above notations the symbol $C$ is omitted when $C = \emptyset$ and the script $r$ is omitted when $r = \infty$. 

\begin{remark}\label{rmk_clb^f} \mbox{} Suppose $O \in {\cal O}(M)$ and 
consider the groups ${\cal G} := {\rm Diff}^r_c(M, M_O)_0$ and ${\cal H}:= {\rm Diff}^r_c(O)_0$. 
\benum[(1)]  
\item 
The restriction induces the canonical isomorphism ${\cal G} \cong {\cal H}$ and 
the definition of $clb^f_{\cal G}$ directly implies that $clb^f_{\cal G}(g) = clb^f_{\cal H}(g|_O)$ for any $g \in {\cal G}$. 

\item 
Suppose $M'$ is another $n$-manifold possibly with boundary and $O' \in {\cal O}(M')$. 
Let ${\cal G}' := {\rm Diff}^r_c(M', M'_{O'})_0$ and ${\cal H}':= {\rm Diff}^r_c(O')_0$. 
If there exists a $C^r$ diffeomorphism $\phi : O \cong O'$, then it induces the isomorphism by conjugation,  
$\phi^\ast : {\cal H} \cong {\cal H}'$, $\phi^\ast(h) = \phi h \phi^{-1}$ $(h \in {\cal H})$ and we have  
$clb^f_{\cal H}(h) = clb^f_{\cal H'}(\phi h \phi^{-1})$ $(h \in {\cal H})$. 
It follows that $clb^f_{\cal G}(g) = clb^f_{\cal G'}(\chi(g))$ $(g \in {\cal G})$ under 
the isomorphism $\chi : {\cal G} \cong {\cal H} \stackrel{\phi^\ast} {\cong}  {\cal H}' \cong {\cal G}'$. 
\eenum 
\end{remark}

\subsection{Basic estimates on $cl$ and $clb^f$ and the absorption\,/\,displacement property} \mbox{} 

The estimates of $cl$ and $clb^f$ is based on the next lemma 
(\cite[Theorem 2.1]{Tsuboi2}, \cite[Proof of Theorem 1.1\,(1)]{Tsuboi3}, \cite{BIP}, \cite{Tsu09}). 
We include its proof to confirm the supports of related diffeomorphisms for later use. 

\begin{lem}\label{basic_lemma_cl} $($\cite{BIP, Tsuboi2, Tsu09, Tsuboi3}$)$ 
Suppose $M$ is an $n$-manifold possibly with boundary, $1 \le r \le \infty, \, r \neq n+1$, 
$F \in {\rm Isot}^r_c(M, \partial)_0$ and $f := F_1 \in {\rm Diff}^r_c(M, \partial)_0$. 
Assume that there exist $U \in {\cal K}(M)$, $V, W \in {\cal K}({\rm Int}\,M)$ and $\phi \in {\rm Diff}^r(M, M_U \cup \partial M)_0$ such that 
$U \supset V \supset W$, $\phi(V) \subset W$ and ${\rm supp}\,F \subset {\rm Int}_M V - W$. Then, the following holds. 
\benum 
\item[$(1)$] 
$cl\,f \leq 2$ in ${\rm Diff}^r(M, M_{U} \cup \partial M)_0$. 
Moreover, there exists a factorization \ $f = gh = (ghg^{-1})g$ \\
\hspace*{20mm} for some $g \in {\rm Diff}^r(M, M_{U} \cup \partial M)_0^c \cap {\rm Diff}^r(M, M_{V})_0$ and $h \in {\rm Diff}^r(M, M_W)_0^c$. 

\item[$(2)$]  
In addition, if 
$clb^f \phi \leq k$ in ${\rm Diff}^r(M, M_U \cup \partial M)_0$, then \\   
$clb^f(f) \leq 2k+1$ in ${\rm Diff}^r(M, M_U \cup \partial M)_0$ and 
we can take $g$ and $h$ in $(1)$ so that \\ 
\hspace*{20mm} $clb^f(g) \leq 2k$ \ in ${\rm Diff}^r(M, M_U \cup \partial M)_0$ \ \ and \ \ $clb^f(h) \leq 1$ \ in ${\rm Diff}^r(M, M_W)_0$. 
\eenum 
\end{lem}  

\begin{proof} Let $O:= {\rm Int}_M V - W \in {\cal O}({\rm Int}\,M)$. 
Since $f\in{\rm Diff}^r(M, M_O)_0$, we have a factorization \\[1mm] 
\hspace*{10mm} $f=[a_1,b_1]\cdots[a_s,b_s]$ \ \ for some $D_i \in {\cal B}(O)$ and 
$a_i, b_i \in {\rm Diff}^r(M, M_{D_i})_0$ $(i=1, \cdots, s)$ (\cite{Th, Ma}). \\[1mm] 
The conjugation by $\phi$  induces the factorization $f=hg_0 = gh$, where 
\vskip 1mm 
\bit 
\item[] $h := [A,B]$, $g_0 := [\phi,H^{-1}]$, $g := g_0^{h}$ and 
\vskip 1mm 
\item[] $\ds H :=\prod_{i=1}^s \big( \phi^{s-i}[a_1,b_1]\cdots[a_i,b_i]\phi^{i-s}\big)$, \hsf 
$\ds A :=\prod_{i=0}^{s-1}\big(\phi^{s-i}a_{i+1}\phi^{i-s}\big)$, \hsf $\ds B :=\prod_{i=0}^{s-1}\big(\phi^{s-i}b_{i+1}\phi^{i-s}\big)$.
\eit 
\vskip 1mm 
\benum 
\item Note that $H \in {\rm Diff}^r(M, M_{V})_0$ and $A, B \in {\rm Diff}^r(M, M_{W})_0$. 
Hence, (i) $h \in {\rm Diff}^r(M, M_{W})_0^c$, (ii) $g_0, g \in {\rm Diff}^r(M, M_{U} \cup \partial M)_0^c$ since $\phi, H, h \in {\rm Diff}^r(M, M_U \cup \partial M)_0$ 
and (iii) $g = fh^{-1} \in {\rm Diff}^r(M, M_{V})_0$. 
\vskip 2mm 

\item Let $\ds D := \bigcup_{i=1}^s \phi^{s-i+1}(D_i) \in {\cal B}_f({\rm Int}_MW)$. \vspace{1mm} 
Then, $A, B \in {\rm Diff}^r(M, M_D)_0$ 
and $h \in {\rm Diff}^r(M, M_D)_0^c$, \break so that \vspace{1mm} 
$clb^f(h) \leq 1$ in ${\rm Diff}^r(M, M_W)_0$.  
On the other hand, in ${\rm Diff}^r(M, M_U \cup \partial M)_0$ we have \\ 
\hsp \hsh 
$\bary[t]{@{}l@{ \ }l}
clb^f(g) & = \, clb^f(g_0) \, 
= \, clb^f(\phi H^{-1}\phi^{-1}H) \, \leq \, clb^f(\phi) + clb^f(H^{-1}\phi^{-1}H) \\[2.5mm]
& = \, clb^f(\phi) + clb^f(\phi^{-1}) \, = \, 2clb^f(\phi) \, \leq \, 2k, \\[2mm] 
clb^f(f) & \leq clb^f(g) + clb^f(h) \leq 2k+1.  
\eary$  
\eenum 
\vskip -5mm 
\end{proof}

\begin{example}\label{example_displaceable} Suppose $M$ is an $n$-manifold without boundary and $1 \le r \le \infty, \, r \neq n+2$. 
\benum
\item $cld\,{\rm Diff}^r_c(M \times I, \partial)_0 \leq 2$. 
\item $clb^fd\,{\rm Diff}^r(M \times I, \partial)_0 \leq 2\, clb^f(\phi_\xi)+1$ \ in ${\rm Diff}^r(M \times I, \partial)_0$, \\
\hsh if $M$ is a closed $n$-manifold, $\xi \in {\rm Diff}^r(I, \partial)_0 - \{ \id_I \}$ and 
$\phi_\xi := \id_M \times \xi \in {\rm Diff}^r(M \times I, \partial)_0$. 
\eenum 
\end{example} 

\begin{proof} (1) The assertion easily follows  from Lemma~\ref{basic_lemma_cl}\,(1). 
\benum[(1)] 
\item[(2)] Given any $f \in {\rm Diff}^r(M \times I, \partial)_0$. 
Take $F \in {\rm Isot}^r(M \times I, \partial)_0$ with $F_1 = f$. 
Then, ${\rm supp}\,F \subset M \times (\alpha, \beta)$ for some closed interval $J \equiv [\alpha, \beta] \subset (0,1)$. 

Since $\xi \neq \id_I$, there exists a closed interval $K \subset (0,1)$ such that 
$\xi(K) \cap K = \emptyset$. 
Take $\eta \in {\rm Diff}(I, \partial)_0$ with $\eta(K) = J$ and 
let $\lambda := \eta \xi \eta^{-1} \in {\rm Diff}^r(I, \partial)_0$. 
Then, $\lambda(J) \cap J = \emptyset$ and so $[\gamma, \delta] := \lambda(J) \subset (0, \alpha) \cup (\beta, 1)$. 
Take $0 < \e_1 < \e_2 < 1$ such that 
$\lambda = \id$ on $[0,\e_1] \cup [\e_2, 1]$ and 
$J \cup \lambda(J) \subset (\e_1, \e_2)$. 
\vskip 2mm 
\noindent Define closed intervals $I_2 \subset I_1 \subset [\e_1, \e_2]$ by \ \ 
$(I_1, I_2) := 
\left\{ \bary[c]{@{\,}ll}
([\e_1, \beta], [\e_1, \delta]) & \text{if $[\gamma, \delta] \subset (0, \alpha)$,} \\[2mm]  
([\alpha, \e_2], [\gamma, \e_2]) & \text{if $[\gamma, \delta] \subset (\beta, 1)$.}
\eary \right.$ \\[2mm] 
Then, $J \subset I_1 - I_2$ and $\lambda(I_1) = I_2$.

For any $\zeta \in {\rm Diff}^r(I, \partial)_0$ we set $\phi_\zeta:= \id_M \times \zeta \in {\rm Diff}^r(M \times I, \partial)_0$. 
It follows that \\ 
\hspp $\phi_\lambda = \phi_\eta \phi_\xi \phi_\eta{}^{-1} \in {\rm Diff}^r(M \times I, \partial)_0$ \ \ and \ \ 
$clb^f \phi_\lambda = clb^f \phi_\xi$ in ${\rm Diff}^r(M \times I, \partial)_0$. \\
Let $(U, V, W) := M \times (I, I_1, I_2)$. Then, it follows that $V, W \in {\cal K}({\rm Int}\,(M \times I) )$, \ $U \supset V \supset W$, \ $\phi_\lambda(V) = W$ \ and \ ${\rm supp}\,F \subset {\rm Int}_{M\times I} V - W$. 
Therefore, Lemma~\ref{basic_lemma_cl}\,(2) implies that \\
\hspp $clb^f f \leq 2\,clb^f \phi_\lambda + 1 = 2\, clb^f \phi_\xi +1$ \ \ in ${\rm Diff}^r(M \times I, \partial)_0$. 
\eenum 
\vskip -7mm 
\end{proof}

Lemma~\ref{basic_lemma_cl} is applied in various situations. 
The existence of $V$, $W$ and $\phi$  
is concerned with absorption and displacement of compact subsets  
by isotopies (cf. \cite{BIP, Tsuboi2, Tsuboi3}). 
We inspect this point more closely.  

\begin{setting}\label{setting_absorption_displacement}
Suppose $M$ is an $n$-manifold possibly with boundary, 
$O \in {\cal O}(M)$, 
$K \in {\cal K}(M)$ and 
$L \in {\cal F}(M)$. 
Recall that $\widetilde{M} := M \cup_{\partial M} (\partial M \times [0,1))$ and 
$\widetilde{O} := O \cup (\partial O \times [0,1)) \in {\cal O}(\widetilde{M})$. 

\end{setting}

\begin{defn}\label{def_absorption_displacement} 
In Setting~\ref{setting_absorption_displacement}:  
We use the following terminologies : 
\benum
\item[{[I]}] Consider any condition ${\cal P}$ on $\phi \in {\rm Diff}_c(M, M_O)_0$. \\
(1) For $C \in {\cal K}(O)$ : 
\bit 
\itemI 
$C$ is absorbed to $K$ in $O$ with ${\cal P}$ \\
\LLRA There exists $\phi \in {\rm Diff}_c(M, M_O)_0$ such that 
$\phi(C) \subset K$ and $\phi$ satisfies the condition ${\cal P}$. \\
The diffeomorphism $\phi$ is called an absorbing diffeomorphism of $C$ to $K$ in $O$ with ${\cal P}$. 

\itemII $C$ is weakly absorbed to $K$ in $O$ with ${\cal P}$ \\ 
\hsf \LLRA $C$ is absorbed to ``any neighborhood of $K$ in $M$'' in $O$ with ${\cal P}$. 
\eit 
(2) $K$ has the (weak) absorption property in $O$ with ${\cal P}$ \\ 
\hsh \hsf \LLRA 
Any $C \in {\cal K}(O)$ is (weakly) absorbed to $K$ in $O$ with ${\cal P}$. \\
(3) When $K \in {\cal K}(O)$ : \\
\hsh \hsf $K$ has the (weak) neighborhood absorption property in $O$ with ${\cal P}$ \\
\hsh \hsf \hsf \LLRA Some compact neighborhood of $K$ in $O$ is (weakly) absorbed to $K$ in $O$ with ${\cal P}$. 

\item[{[II]}] 
(1) $K$ is displaceable from $L$ in $O$ (or $K$ has the displacement property for $L$ in $O$) \\
\hsh \hsf \ \LLRA 
There exists $\psi \in {\rm Diff}_c(M, M_O)_0$ such that $\psi(K) \cap L = \emptyset$. \\ 
\hspace*{6mm} The diffeomorphism $\psi$ is called a displacing diffeomorphism of $K$ from $L$ in $O$. \\
(2) $K$ is strongly displaceable from $L$ (or $K$ has the strong displacement property for $L$) \\
\hsh \hsf \ \LLRA $K$ is displaceable from $L$ in any open neighborhood of $K$ in $M$
\eenum
\end{defn} 

\begin{compl}\label{compl_abs_dis} \mbox{}
\benum 
\item In this paper we are concerned with the following conditions ${\cal P}$ on $\phi$ : 
\bit  
\itemI (a) \,rel $L$ \LLRA $\phi(L) \subset L$ 
\hsp (b) \,keeping $L$ invariant \LLRA $\phi(L) = L$ 
\itemII $clb^f \leq k$ \LLRA $clb^f \phi \leq k$ \ in ${\rm Diff}_c(M, M_O)_0$ \hsp ($k \in \IZ_{\geq 0} \cup \{ \infty \}$)
\eit 
These conditions will play important roles in the estimates of $cl$ and $clb^f$ in Sections 3, 4.

\item When $K$ has the weak neighborhood absorption property in $O$, we set \\
\hsp $clb^f (K; O) := \min \{ k \in \IZ_{\geq 0} \cup \{ \infty \} \mid (\ast)_k \}$ : \\
\hsp \hsh \text{$(\ast)_k$\ $K$ has the weak neighborhood absorption property in $O$ with $clb^f \leq k$.}

\item In [I](1) consider the case where the condition ${\cal P}$ is ``rel $L$'' or ``keeping $L$ invariant''. 
If $C \in {\cal K}(O \cap {\rm Int}\,M)$ and the absorbing diffeomorphism $\phi$ is constructed as a truncation of a flow which keeps $L$ invariant, then 
we can cut off the vector field associated to the flow around the boundary and modify $\phi$ so that 
$\phi \in {\rm Diff}_c(M, \partial M \cup M_O)_0$ with still keeping the condition ${\cal P}$. 

\item In [II](2), if $K \in {\cal K}(O \cap {\rm Int}\,M)$, then the isotopy extension theorem is used to modify $\psi$ 
so that $\psi \in {\rm Diff}_c(M, \partial M \cup M_O)_0$ (cf. Lemma~\ref{Isotopy_ext}). 
\item The properties defined in Definition~\ref{def_absorption_displacement} are preserved by appropriate diffeomorphisms. 
For example, consider the following condition on $(M, O, K, C)$ : 
\bit 
\item[] $(\ast)$ \ $C \in {\cal K}(O)$ is weakly absorbed to $K$ in $O \subset M$ with $clb^f \leq k$. 
\eit 
Suppose $M'$ is another $n$-manifold, $O' \in {\cal O}(M')$ and $h : O \cong O'$ is a $C^\infty$ diffeomorphism. 
When $K \in {\cal K}(O)$, from Remark~\ref{rmk_clb^f} it follows that $(M, O, K, C)$ satisfies $(\ast)$ iff $(M', O', h(K), h(C))$ satisfies $(\ast)$
\eenum 
\end{compl} 

\begin{example}\label{exp_displacement} \mbox{} 
Suppose $M$ is an $n$-manifold possibly with boundary. 
We list some situations in which $K$ is strongly displaceable from $L$ in $M$. 

\benum 
\item $($\cite[Lemma 2.1]{Tsuboi3}$)$ 
Suppose $K$ is a compact $k$-dimensional stratified subset of ${\rm Int}\,M$ and $L$ is an $\ell$-dimensional stratified subset of $M$. 

\bit 
\itemI If $k + \ell < n$, then $K$ is strongly displaceable from $L$. 
\itemII In the case $k +\ell = n$ : 
\bit 
\itema If $K \subset O \in {\cal O}(M)$ and $Cl_M(K - K^{(k-1)})$ is displaceable from $Cl_M(L - L^{(\ell-1)})$ in $O$, then 
$K$ is displaceable from $L$ in $O$. 

\itemb If $Cl_M(K - K^{(k-1)})$ is strongly displaceable from $Cl_M(L - L^{(\ell-1)})$ in $M$, then $K$ is strongly displaceable from $L$ in $M$. 
\eit 
\eit  

\item Suppose $K \in {\cal K}({\rm Int}\,M)$ and $L \in {\cal F}(M)$. 
\bit 
\itemI If $U$ is an open $n$-disk in $M$ with $K \subset U \nsubseteq L$, then $K$ is displaceable from $L$ in $U$. 
\itemII If $K$ has arbitrarily small open $n$-disk neighborhoods in $M$ which is not included in $L$, then $K$ is strongly displaceable from $L$ in $M$. 
\eit 

\item If $L$ is a submanifold of $M$, $K$ is a compact subset of $L$ and 
the normal bundle of $L$ in $M$ admits a non-vanishing section over $K$, then $K$ is strongly displaceable from $L$. 
\eenum 
\end{example}

We list some basic facts on the absorption/displacement property. 
The next lemma follows directly from the definitions (and a simple argument using the collar for (1)(ii)) . 

\begin{lem}\label{lem_absorption_displacement} \mbox{} 
In Setting{\rm ~\ref{setting_absorption_displacement}} : 
\benum 
\item[{\rm (1)}]  
\bit 
\itemI Suppose $K_1, K_2, K_3 \in {\cal K}(O)$. 
If $K_1$ is $($weakly$)$ absorbed to $K_2$ in $O$ rel $L$ and $K_2$ is $($weakly$)$ absorbed to $K_3$ in $O$ rel $L$, then 
$K_1$ is $($weakly$)$ absorbed to $K_3$ in $O$ rel $L$. 

\itemII 
If $K$ has the weak absorption property in $O \subset M$ $($keeping $K$ invariant$)$, then 
$K$ has the weak absorption property in $\widetilde{O} \subset \widetilde{M}$ $($keeping $K$ invariant$)$.  
\eit 

\item[{\rm (2)}]  
\bit 
\itemI If $K$ is displaceable from $K \cup L$ in $O$,  
then a sufficiently small compact neighborhood $K_0$ of $K$ in $O$ is displaceable from $K_0 \cup L$. 

\itemII 
{\rm (a)} If $K$ is weakly absorbed to $K' \in {\cal K}(O)$ in $O$ and $K'$ is displaceable from $L$ in $O$, 
then $K$ is displaceable from $L$ in $O$. 

{\rm (b)} If $K$ is weakly absorbed to $K' \in {\cal K}(O)$ in $O$ rel $L$ and $K'$ is displaceable from $K' \cup L$ in $O$, 
then $K$ is displaceable from $K \cup L$ in $O$. 

\itemiii If $K$ is displaceable from $L$ in $O$ and $K$ has the weak absorption property in $O$ keeping $K$ and $L$ invariant, then 
$K$ is strongly displaceable from $L$ in $O$. 
\eit 
\eenum 
\end{lem} 

\begin{lem}\label{lem_clb^f_basic} In Setting{\rm ~\ref{setting_absorption_displacement}}; \ 
Suppose $K$ has the weak neighborhood absorption property in $O$. 
\benum 
\item[$(1)$] Suppose $K' \in {\cal K}(O)$ is weakly absorbed to $K$ in $O$ keeping $K$ invariant.
\bit 
\itemI If $K$ has the weak neighborhood absorption property in $O$ with $clb^f \leq k$, 
then some compact neighborhood $V$ of $K'$ in $O$ is weakly absorbed to $K$ in $O$ with $clb^f \leq k$. 

\itemII If $K \subset K'$, 
then $clb^f(K'; O) \leq clb^f(K;O)$. 
\eit 
\item[$(2)$] If $K$ has  
the weak neighborhood absorption property in $O$ with $clb^f \leq k$
and the weak absorption property in $O$ keeping $K$ invariant, then 
$K$ has the weak absorption property in $O$ with $clb^f \leq k$. 

\eenum 
\end{lem}

\begin{proof} 
(1)\,(i) By the assumption $K$ admits a compact neighborhood $U$ in $O$ which is weakly absorbed to $K$ in $O$ with $clb^f \leq k$. 
By the assumption on $K'$, there exists $\chi \in {\rm Diff}_c(M, M_O)_0$ such that $\chi(K') \subset {\rm Int}_MU$ and $\chi(K) = K$.  
We can find  a compact neighborhood $V$ of $K'$ in $O$ such that $\chi(V) \subset U$. 
Then, $V$ is weakly absorbed to $K$ in $O$ with $clb^f \leq k$.
In fact, given any neighborhood $W$ of $K$ in $O$. 
Since $\chi(W)$ is a neighborhood of $K$ in $O$, the choice of $U$ implies the existence of $\phi \in {\rm Diff}_c(M, M_O)_0$ such that $\phi(U) \subset \phi(W)$ and $clb^f \phi \leq k$ in ${\rm Diff}_c(M, M_O)_0$. 
Then, $\psi := \chi^{-1}\phi \chi \in {\rm Diff}_c(M, M_O)_0$ satisfies the required conditions 
$\psi(K') \subset W$ and 
$clb^f \psi = clb^f (\chi^{-1}\phi \chi) = clb^f \phi \leq k$ \ in ${\rm Diff}_c(M, M_O)_0$.  

The remaining statements readily follow from (1)(i). 
\end{proof} 

\begin{setting}\label{setting_basic_cl_clb^f}
Suppose $M$ is an $n$-manifold possibly with boundary, $1 \le r \le \infty, \, r \neq n+1$, $O \in {\cal O}({\rm Int}\,M)$, $L \in {\cal F}(M)$ and 
$k \in \IZ_{\geq 0}$. 
\end{setting}

Consider the following statements : 
\bit 
\item[{\small $(\#)$}] $cld\,{\rm Diff}^r_c(M, M_O)_0 \leq 2$ \ and \ 
any $f \in {\rm Diff}^r_c(M, M_O)_0$ has a factorization $f = gh$ \\
\hspp for some \ $g \in {\rm Diff}^r_c(M, M_O)_0^c$ \ and \ $h \in {\rm Diff}^r_c(M, M_O \cup L)_0^c$. 
\vskip 1mm 

\item[$(\flat)$]   
\benum 
\item[$(1)$]\hspace{0.2mm}{\rm (i)}\, $cld\,{\rm Diff}^r_c(M, M_{O})_0 \leq 2$, \hsh 
{\rm (ii)} $clb^f\!d\,{\rm Diff}_c^r(M, M_O)_0 \leq 2k+1$.  
\item[$(2)$] any $f \in {\rm Diff}^r_c(M, M_O)_0$ has a factorization $f = gh$ such that \\ 
\hsp \btab[t]{lll}
$g \in {\rm Diff}^r_c(M, M_O)_0^c$ & and & $clb^f(g) \leq 2k$ in ${\rm Diff}^r_c(M, M_O)_0$, \\[2mm]
$h \in {\rm Diff}^r_c(M, M_O \cup L)_0^c$ & and & $clb^f(h) \leq 1$ in ${\rm Diff}^r_c(M, M_O \cup L)_0$. 
\etab 
\eenum  
\eit

\begin{lem}\label{basic_lemma_cl-2} In Setting~{\rm \ref{setting_basic_cl_clb^f}};  
Suppose $f \in {\rm Diff}^r_c(M, M_O)_0$.
\benum[(1)] 
\item[{\rm [I]}] There exists a factorization \ $f = gh$ \ \ for some \ $g \in {\rm Diff}^r_c(M, M_O)_0^c$ \ and \ $h \in {\rm Diff}^r_c(M, M_O \cup L)_0^c$ \\
if there exists $F \in {\rm Isot}^r_c(M, M_O)_0$ with $F_1 = f$, $K \in {\cal K}(O)$ and a compact neighborhood $U$ of $K$ in $O$ 
which satisfy one of the following conditions $(1)$ and $(2)$:

\bit 
\item[$(1)$] 
\mbox{}\,{\rm (i)}\,
{\rm (a)} ${\rm supp}\,F$ is absorbed to ${\rm Int}_M U - K$ in $O$ rel $L$, \hsf 
{\rm (b)} $U$ is weakly absorbed to $K$ in $O$. \\
{\rm (ii)} $K$ is strongly displaceable from $L$ in $M$. 

\item[$(2)$] 
\mbox{}\,{\rm (i)}\, {\rm (a)} ${\rm supp}\,F$ is weakly absorbed to $K$ in $O$ rel $L$, \hsf 
{\rm (b)} $U$ is weakly absorbed to $K$ in $O$. \\
{\rm (ii)} $K$ is displaceable from $K \cup L$ in ${\rm Int}_MU$. 
\eit 
\item[{\rm [II]}] In {\rm [I]\,(1)}, suppose $K$ and $U$ satisfy the following stronger condition : 
\bit 
\item[]
\bit 
\itemI {\rm (b)}$'$ $U$ is weakly absorbed to $K$ in $O$ with $clb^f \leq k$. 
\eit 
\eit 
Then, $clb^f(f) \leq 2k+1$ in ${\rm Diff}^r_c(M, M_O)_0$ and in the factorization $f = gh$ we can take 
$g$ and $h$ with 
\hspace*{25mm} $clb^f(g) \leq 2k$ in ${\rm Diff}^r_c(M, M_O)_0$ \ \ and \ \  
$clb^f(h) \leq1$ in ${\rm Diff}^r_c(M, M_O \cup L)_0$.  
\eenum 
\end{lem}

\begin{proof} 
{[I]}\,(1) \& [II] (Alterations for [II] are shown in parentheses.) 
By the condition (i)(a) there exist $\chi \in {\rm Diff}_c(M, M_O)_0$ such that 
$\chi({\rm supp}\,F) \subset {\rm Int}_M U - K$ and $\chi(L) \subset L$. 
Consider the open neighborhood $U_0 := {\rm Int}_M U - \chi_1({\rm supp}\,F)$ of $K$ in ${\rm Int}_M U$. 
By the condition (ii) there exist $\eta \in {\rm Diff}_c(M, M_{U_0})_0$ such that $\eta(K) \cap L = \emptyset$. 
Take a compact neighborhood $W_0$ of $\eta(K)$ in $U_0 - L$. 
By the condition (i)(b) ((i)(b)$'$) there exists $\psi' \in {\rm Diff}_c(M, M_O)_0$ such that \\
\hspace*{15mm} $\psi'(U) \subset \eta^{-1}({\rm Int}_M W_0)$ \hsf (and $\ clb^f(\psi') \leq k$ in ${\rm Diff}_c(M, M_O)_0$). \\
Let $\psi := \eta\psi' \eta^{-1} \in {\rm Diff}_c(M, M_O)_0$. 
Then we have \\  
\hspace*{15mm} $\psi(U) \subset {\rm Int}_M W_0$ \hsf (and \ $clb^f(\psi) = clb^f(\eta \psi' \eta^{-1}) = clb^f(\psi') \leq k$ \ in \ ${\rm Diff}_c(M, M_O)_0$). \\
Consider \ $V := \chi_1^{-1}(U), \ W := \chi_1^{-1}(W_0) \in {\cal K}(O)$ \  
and \ $\phi := \chi^{-1}\psi\chi \in {\rm Diff}_c(M,M_O)_0.$ \ 
It follows that \\
\hspace*{20mm} $W \subset {\rm Int}_MV - L \subset O - L$, \hsf 
${\rm supp}\,F \subset {\rm Int}_MV - W$, \hsf $\phi(V) \subset W$ \\
\hspace*{40mm} (and \ $clb^f(\phi) = clb^f(\chi^{-1}\psi\chi) = clb^f(\Psi_1) \leq k$ \ in \ ${\rm Diff}_c(M, M_O)_0$). \\
Since ${\rm Diff}^r(M, M_W)_0 \subset {\rm Diff}^r_c(M, M_O \cup L)_0$, 
the conclusion follows from Lemmas~\ref{basic_lemma_cl}. 

{[I]}\,(2) 
By the condition (ii) there exists $\phi \in {\rm Diff}(M, M_U)_0$ such that $\phi(K) \cap (K \cup L) = \emptyset$. 
We show that the condition (1) is satisfied for $K_1 := \phi_1(K) \in {\cal K}({\rm Int}_MU)$ instead of $K$. 

Since ${\rm Int}_MU - K_1$ is a neighborhood of $K$, by the condition (i)(a) 
${\rm supp}\,F$ is absorbed to ${\rm Int}_M U - K_1$ in $O$ rel $L$. 
From the condition (i)(b) and Lemma~\ref{lem_absorption_displacement}\,(1) 
it follows that $U$ is weakly absorbed to $K_1$ in $O$.
Since $K_1 \cap L = \emptyset$, obviously $K_1$ is strongly displaceable from $L$ in $M$.
Therefore, the assertion follows from [I]\,(1). 
\end{proof}

\begin{lem}\label{basic_lemma_cl-bdry}
 In Setting~{\rm \ref{setting_basic_cl_clb^f}};  
\benum[(1)] 
\item[{\rm [I]}] The statement $(\#)$ holds,  
if the following condition is satisfied : 
\benum 
\item[$(\ast)$] There exists $O_0 \in {\cal O}(M)$ with $O_0 \cap {\rm Int}\,M = O$, 
$K_0 \in {\cal K}(O_0)$ and a compact neighborhood $U_0$ of $K_0$ in $\widetilde{O}_0$ such that 
\bit 
\itemI 
{\rm (a)} $U_0$ is weakly absorbed to $K_0$ in $\widetilde{O}_0$, \\
{\rm (b)} $K_0$ has the weak absorotion property in $O$ rel $L$ and 
\itemII $K_0$ is displaceable from $K_0 \cup L \cup (\partial M \times [0,1))$ in ${\rm Int}_{\widetilde{M}}U_0$, 
\eit 
\eenum 

\vskip 1mm 
\item[{\rm [II]}] The statement $(\flat)$ holds,
 if in {\rm [I]} $K_0$ and $U_0$ satisfy the following stronger condition : 
\benum 
\item[]
\bit 
\itemI\hspace{-1mm}{\rm (a)$'$} $U_0$ is weakly absorbed to $K_0$ in $\widetilde{O}_0$ with $clb^f \leq k$ in ${\rm Diff}_c(\widetilde{M}, \widetilde{M}_{\widetilde{O}_0})_0$. 
\eit 
\eenum 
\eenum
\end{lem}

\begin{proof} (Alterations for [II] are shown in parentheses.)  
Given any $f \in {\rm Diff}^r_c(M, M_O)_0$. There exists $F \in {\rm Isot}^r_c(M, M_O)_0$ with $F_1 = f$. 
By the condition (ii) there exists $\psi \in {\rm Diff}(\widetilde{M}, \widetilde{M}_{U_0})_0$ 
such that $K := \psi(K_0) \subset {\rm Int}\,M - (K_0 \cup L)$. 
Note that $K \in {\cal K}(O)$. 
Since ${\rm supp}\,F \in {\cal K}(O)$ and $K_0 \subset ({\rm Int}_{\widetilde{M}}U_0 - K) \cap M \in{\cal O}(O_0)$, 
by the condition (i)(b) there exists $h \in {\rm Diff}_c(M, M_{O})_0$ 
such that $h({\rm supp}\,F) \subset {\rm Int}_{\widetilde{M}}U_0 - K$ and $h(L) \subset L$.
We can push $U_0$ inward using a bicollar of $\partial O_0$ in $\widetilde{O}_0$ to obtain 
$\phi \in {\rm Diff}_c(\widetilde{M}, \widetilde{M}_{\widetilde{O}_0} \cup H_1({\rm supp}\,F) \cup K)_0$ such that 
$U := \phi(U_0) \subset O$. 
Then $U$ is a compact neighborhood of $K$ in $O$. 

The conclusions follow from Lemma~\ref{basic_lemma_cl-2}, once 
we show that $F$, $K$ and $U$ satisfies the conditions in Lemma~\ref{basic_lemma_cl-2}\,[I](1) (and [II]). 
Since $h({\rm supp}\,F) \subset {\rm Int}_M U - K$, it follows that 
${\rm supp}\,F$ is absorbed by $h$ to ${\rm Int}_M U - K$ in $O$ rel $L$. 
Since $K \cap L = \emptyset$, obviously $K$ is strongly displaceable from $L$ in $M$. 

It remains to show that $U$ is weakly absorbed to $K$ in $O$ (with $clb^f \leq k$). 
This follows from the assumption (i)(a) ((i)(a)$'$), Complement~\ref{compl_abs_dis} and the following observation. 
There exists a collar $E \cong \partial O_0 \times I$ of $\partial O_0$ in $O_0 - U$ and 
a $C^\infty$ diffeomorphism $\theta : \widetilde{O}_0 \cong O$ such that 
$\theta = \id$ on $O - E$. 
Then $(U, K)$ corresponds to $(U_0, K_0)$ under the $C^\infty$ diffeomorphism 
$\theta \phi\psi : \widetilde{O}_0 \cong O$. 
This completes the proof. 
\end{proof}

Finally we remove the conditions on the neighborhood $U_0$ from Lemma~\ref{basic_lemma_cl-bdry} to obtain a more practical criterion. 

\begin{lem}\label{basic_lemma_cl-bdry_final} In Setting~{\rm \ref{setting_basic_cl_clb^f}};  
The statement $(\flat)$ holds, if one of the conditions {\rm [A]} and {\rm [B]} is satisfied. 
\vskip 2mm 
\benum[(1)]  
\item[{\rm [A]}] There exist $O_0 \in {\cal O}(M)$ with $O = O_0 \cap {\rm Int}\,M$ and $K \in {\cal K}(O_0)$ such that  
\bit 
\item[$(\natural)$\,] $clb^f(K, \widetilde{O}_0) \leq k$. 
\item[$(\ast)$\,] $K$ satisfies one of the conditions $(\ast)_1$ and $(\ast)_2$ : 
\bit 
\item[$(\ast)_1$] 
{\rm (i)} $K$ has the weak absorption property {\rm (a)} in $O$ rel $L$ and {\rm (b)} in $O_0$ keeping $K$ invariant. \\
{\rm (ii)} $K$ is  displaceable from $K \cup L \cup (\partial M \times [0,1))$ in $\widetilde{O}_0$. 

\item[$(\ast)_2$]
{\rm (i)} $K$ has the weak absorption property {\rm (a)} in $O$ rel $L$ and {\rm (b)} in $O_0$. \\
{\rm (ii)} $K$ is strongly displaceable from $K \cup L \cup (\partial M \times [0,1))$ in $\widetilde{M}$. 
\eit 
\eit 

\item[{\rm [B]}] There exist $K \in {\cal K}(O)$ such that  
\bit 
\item[$(\natural)$\,] $clb^f(K, O) \leq k$. 
\item[$(\ast)$\,] $K$ satisfies one of the conditions $(\ast)_1$ and $(\ast)_2$ : 
\bit 
\item[$(\ast)_1$] 
{\rm (i)} $K$ has the weak absorption property in $O$ rel $L$ and keeping $K$ invariant. \\
{\rm (ii)} $K$ is  displaceable from $K \cup L$ in $O$. 

\item[$(\ast)_2$]
{\rm (i)} $K$ has the weak absorption property in $O$ rel $L$. \\
{\rm (ii)} $K$ is strongly displaceable from $K \cup L$ in $M$. 
\eit 
\eit 

\eenum 
\end{lem} 

\begin{proof} \mbox{} [A] By Lemma~\ref{basic_lemma_cl-bdry} it suffices to show that 
$K$ satisfies the following conditions. 
\benum 
\item[$(\ddagger)$] 
\bit 
\itemI $K$ has the weak absorotion property in $O$ rel $L$, 

\itemII there exists a compact neighborhood $U$ of $K$ in $\widetilde{O}_0$ such that 
\bit 
\itema $K$ is displaceable from $K \cup L \cup (\partial M \times [0,1))$ in ${\rm Int}_{\widetilde{M}} U$, 
\itemb $U$ is weakly absorbed to $K$ in $\widetilde{O}_0$ with $clb^f \leq k$ in ${\rm Diff}_c(\widetilde{M}, \widetilde{M}_{\widetilde{O}_0})_0$. 
\eit 
\eit 
\eenum 
The condition $(\ddagger)$(i) is included as $(\ast)_1$(i)(a) and $(\ast)_2$(i)(a). 
For the condition $(\ddagger)$(ii), first note that 
the condition $(\natural)$ means that 
\benum 
\item[$(\natural)$] $K$ has the weak neighborhood absorption property in $\widetilde{O}_0$ with $clb^f \leq k$ 
in ${\rm Diff}_c(\widetilde{M}, \widetilde{M}_{\widetilde{O}_0})_0$. 
\eenum 

\benum[(1)] 
\item[] \hspace{-7mm} $(\ast)_1$ case : 
By $(\ast)_1$(ii) there exists a compact neighborhood $U$ of $K$ in $\widetilde{O}_0$ for which  holds $(\ddagger)$(ii)(a). 
By Lemma~\ref{lem_absorption_displacement}\,(1)(ii) $K$ has the weak absorption property in $\widetilde{O}_0$ keeping $K$ invariant.
Hence, by Lemma~\ref{lem_clb^f_basic}\,(2) and $(\natural)$ it follows that 
\bit 
\item[$(\natural)'$] $K$ has the weak absorption property in $\widetilde{O}_0$ with $clb^f \leq k$ 
in ${\rm Diff}_c(\widetilde{M}, \widetilde{M}_{\widetilde{O}_0})_0$. 

\eit 
This implies $(\ddagger)$(ii)(b). 

\item[] \hspace{-7mm} $(\ast)_2$ case :  By the condition $(\natural)$ there exists a compact neighborhood $U$ of $K$ in $\widetilde{O}_0$ which satisfies $(\ddagger)$(ii)(b). 
Then $(\ddagger)$(ii)(a) follows from $(\ast)_2$(ii).
\eenum 

The condition {[B]} is a special case of [A] (taking $O_0 = O$). 
\end{proof} 

\subsection{Factorization of isotopies on manifolds} \mbox{} 

This subsection includes some remarks on factorization of isotopies. 
In the study of commutator length of diffeomorphisms, modification and factorization of isotopies are important subjects.
The next factorization lemma is repeatedly applied in the subsequent sections (cf.\,\cite[Lemma 2.3]{Tsuboi3}). 
Recall that $\pi_M : M \times I \to M$ is the projection onto $M$. 

\begin{lem}~\label{Isotopy_ext} 
Suppose $M$ is a compact $n$-manifold possibly with boundary, $K, L \in {\cal K}(M)$ and 
$F \in {\rm Isot}^r(M, \partial M \cup (K \cap L))_0$ satisfies $\pi_M F((K-L) \times I) \cap L = \emptyset$. 
Then there exists $G \in {\rm Isot}^r(M, \partial M \cup L)_0$ such that  $G = F$ on $[$a neighborhood of $K] \times I$.
It induces a factorization $F= GH$, where $H = G^{-1}F \in {\rm Isot}^r(M, \partial M \cup K)_0.$
\end{lem}

The isotopy $G$ in Lemma~\ref{Isotopy_ext} is constructed as follows : 
Let $X$ be the velocity vector field of $F$. 
There exists a vector field $Y$ on $M \times I$ such that $Y = X$ on a neighborhood of $F(K \times I)$ and $\ds Y = \frac{\partial}{\partial t}$ on [a neighborhood of $\partial M \cup L] \times I$. 
The isotopy $G$ is obtained by integrating $Y$ on $M \times I$. 

\begin{lem}\label{lem_factorization_open} \mbox{}  
Suppose $M$ is an open $n$-manifold, $F \in {\rm Isot}^r(M)_0$ 
and ${\mathcal R}$ is a cofinal subcollection of ${\mathcal S}{\mathcal M}_c(M)$. 
\begin{enumerate} 
\item[$(1)$] For any  $K \in {\mathcal S}{\mathcal M}_c(M)$ there exists $L \in {\mathcal R}$ such that $F(K \times I) \Subset L \times I$ and 
$F(M_L \times I) \Subset M_K \times I$.

\item[$(2)$] For any $K_k \in {\mathcal S}{\mathcal M}_c(M)$ $(k \geq 1)$ there exists an exhausting sequence $\{ M_k \}_{k \geq 1}$ of $M$ such that 
$K_k \Subset M_k \in {\mathcal R}$ $(k \geq 1)$ and 
$F(M_{4k,4k+1} \times I) \Subset M_{4k-1, 4k+2} \times I$ $(k \geq 0)$. 
$($In the inductive construction of $\{ M_k \}_{k \geq 1}$, given $M_k$, we can take $M_{k+1}$ arbitrarily large.$)$

\item[$(3)$] Let $\{ M_k \}_{k \geq 1}$ be an exhausting sequence of $M$ such that $F(M_{4k,4k+1} \times I) \Subset M_{4k-1, 4k+2} \times I$ $(k \geq 0)$
and set $M' := \bigcup_{k \geq 0} M_{4k+2, 4k+3}$ and $M'' := \bigcup_{k \geq 0} M_{4k, 4k+1}$. 
Then there exists a factorization $F = GH$ for some $G \in {\rm Isot}^r(M, M')_0$ and $H \in {\rm Isot}^r(M, M'')_0$. 
\end{enumerate} 
\end{lem}

\begin{proof} \mbox{} 
\begin{enumerate} 
\item Since $F(K \times I), F^{-1}(K \times I) \in {\cal K}(M \times I)$, 
there exists $L \in {\mathcal R}$ such that $F(K \times I) \cup F^{-1}(K \times I) \Subset L \times I$. 
Then $L$ satisfies the required conditions. 
 
\item There exist $M_1, M_2, M_3 \in {\cal R}$ such that \\
\hspace*{25mm} $K_i \Subset M_i$ $(i=1,2,3)$, \ $F(M_{1} \times I) \Subset M_{2} \times I$ \ and \ $M_{2} \Subset M_{3}$. \\
Inductively, for $k \geq 1$, given $M_{4k-1} \in {\cal R}$, by (1) we can find $M_{4k+i} \in {\cal R}$ $(i=0, 1, 2, 3)$ 
such that \\
\hspace*{12mm} 
$K_{4k+i} \Subset M_{4k+i}$ $(i=0,1,2,3)$, \ \ $F(M_{M_{4k}} \times I) \Subset M_{M_{4k-1}}\times I$, \ \ 
$M_{4k} \Subset M_{4k+1}$, \\ 
\hspace*{12mm} $F(M_{4k+1} \times I) \Subset M_{4k+2} \times I$, \ \ 
$M_{4k+2} \Subset M_{4k+3}$. \\
The sequence $\{ M_k \}_{k\geq 1}$ is obtained by iteration of this procedure. 

\item For each $k \geq 0$, since $F(M_{4k, 4k+1} \times I) \Subset M_{4k-1, 4k+2} \times I$, 
by the isotopy extension theorem  we obtain $G^k \in {\rm Isot}^r(M_{4k-1,4k+2}; \partial)_0$ such that 
$G^k = F$ on $[$a neighborhood of $M_{4k, 4k+1}] \times I$. 
The isotopy $G \in {\rm Isot}^r(M; M')_0$ is defined by $G = G^k$ on $M_{4k-1,4k+2}$ $(k \geq 0)$. 
Then $G = F$ on $[$a~neighborhood of $M''] \times I$ and $H := G^{-1}F \in {\rm Isot}^r(M; M'')_0$. 
\end{enumerate} 
\vskip -7mm 
\end{proof}

Here we give a simple application of Lemma~\ref{lem_factorization_open}.

\begin{lem}\label{lemma_subseq} 
Suppose $M$ is an open $n$-manifold. 
\benum
\item[{\rm (1)}] 
$cld\,{\rm Diff}^r(M)_0 \leq \max\{ \ell + m, 2m\}$ \ \ and \ \ $cld\,{\rm Diff}_c^r(M)_0 \leq \ell$ \\
if $M$ has an exhausting sequence ${\mathcal L} = \{ L_i \}_{i \geq 1}$ such that \\ 
\hspace*{10mm} $cld\,{\rm Diff}^r(L_i; \partial)_0 \leq \ell$ $(1 \leq i  < \infty)$ \ and \ $cld\,{\rm Diff}^r(L_{ij}; \partial)_0 \leq m$ \ $(1 \leq i < j < \infty)$. 
\item[{\rm (2)}]
$clb^d d\,{\rm Diff}^r(M)_0 \leq \max\{ \ell + m, 2m\}$ \ \ and \ \  $clb^f d\,{\rm Diff}_c^r(M)_0 \leq \ell$ \\ 
if $M$ has an exhausting sequence ${\mathcal L} = \{ L_i \}_{i \geq 1}$ such that \\ 
\hspace*{10mm} $clb^f d\,{\rm Diff}^r(L_i; \partial)_0 \leq \ell$ $(1 \leq i  < \infty)$ \ and \ $clb^f d\,{\rm Diff}^r(L_{ij}; \partial)_0 \leq m$ \ $(1 \leq i < j < \infty)$. 
\eenum 
\end{lem}

\begin{proof} The assertions for ${\rm Diff}_c^r(M)_0$ are obvious. 
To obtain the estimates for ${\rm Diff}^r(M)_0$,  
let $q = cl$ in (1) and $q = clb^d$ in (2). 
Given any $f \in {\rm Diff}^r(M)_0$. 
Take $F \in {\rm Isot}^r(M)_0$ with $F_1 = f$. 
By Lemma~\ref{lem_factorization_open} 
we have 
\bit 
\itemI a subsequence $M_k = L_{i(k)}$ $(k \geq 1)$ of ${\mathcal L}$ 
such that  $F(M_{4k, 4k+1} \times I) \Subset M_{4k-1, 4k+2} \times I$ $(k \geq 0)$ and 
\itemII a factorization $F = GH$ for some $G \in {\rm Isot}^r(M, M')_0$ and $H \in {\rm Isot}^r(M, M'')_0$, 
where \\
\hspace*{30mm} $M' := \bigcup_{k \geq 0} M_{4k+2, 4k+3}$ \ \ and \ \ $M'' := \bigcup_{k \geq 0} M_{4k, 4k+1}$. 
\eit 

Hence, the assumption implies that 
$q(G_1) \leq \max\{ \ell, m\}$, $q(H_1) \leq m$ and $q(f) \leq \max\{ \ell, m\} + m$. 
\end{proof}

\begin{prop}\label{prop_product_end} 
Suppose $W$ is a compact $n$-manifold  with boundary and $1 \le r \le \infty, \, r \neq n+1$. 
\benum
\item[$(1)$] $cld\, {\rm Diff}^r({\rm Int}\,W)_0 \leq \max\{ cld\,{\rm Diff}^r(W; \partial)_0 , 2 \} + 2$. 
\item[$(2)$] $clb^d d\,{\rm Diff}^r({\rm Int}\,W)_0 \leq \max\{ clb^f d\,{\rm Diff}^r(W; \partial)_0, m \} + m$, \\
\hspp where \ \ $m:= clb^fd\,{\rm Diff}^r(\partial W \times I; \partial)_0 < \infty$ 
\ \ $($cf. Example~{\rm \ref{example_displaceable}\,(2)}$)$
\eenum 
\end{prop} 

\begin{proof} Let $q = cl$ in (1) and $q = clb^f$ in (2).
Take a collar $N \cong \partial W \times I$ of $\partial W = \partial W \times \{ 0 \}$ in $W$ 
and consider the exhausting sequence ${\mathcal L} = \{ L_i \}_{i \geq 1}$ in ${\rm Int}\,W$ defined by $L_i := W_N \cup (\partial W \times [1/i,1])$ $(i \geq 1)$. 
Since $L_i$ $(1 \leq i < \infty)$ are diffeomorphic to $W$ and $L_{ij}$ $(1 \leq i < j < \infty)$ are diffeomorphic to $\partial W \times I$, 
it follows that $qd\,{\rm Diff}^r(L_i; \partial)_0 = qd\,{\rm Diff}^r(W; \partial)_0$ $(1 \leq i < \infty)$ and 
$qd\,{\rm Diff}^r(L_{ij}; \partial)_0 = qd\,{\rm Diff}^r(\partial W \times I; \partial)_0$ $(1 \leq i < j < \infty)$. 
Hence the conclusion follows from Lemma~\ref{lemma_subseq} and Example~\ref{example_displaceable}. 
\end{proof}

\begin{example}\label{exp_portable} 
If $W$ is a compact $n$-manifold with boundary and has a handle decomposition with no handles of indices greater than $(n-1)/2$, then 
$cld\,{\rm Diff}^r(W; \partial)_0 \leq 2$ (\cite[Theorem 4.1]{Tsuboi1}) and hence $cld\, {\rm Diff}^r({\rm Int}\,W)_0 \leq 4$ 
(cf \cite{R2} for portable manifolds and related topics). 
\end{example}

\section{Handle decompositions and triangulations of manifolds} 

To explain the aim of this section, first we recall the strategy in \cite{BIP, Tsuboi2, Tsu09, Tsuboi3} 
for estimate of $cl$ and $clb^f$ (i.e., effective factorization to commutators) 
and understand the trouble which occurs in the bounded case. 

Suppose $M$ is a closed $n$-manifold and $f \in {\rm Diff}(M)_0$.
For simplicity, we assume that $n = 2m+1$ and 
consider any handle decomposition ${\cal H}$ of $M$, the gradient flow $\phi_t$ induced from ${\cal H}$, 
the core $m$-skeleton $P$ and the dual $m$-skeleton $Q$. 
Then, by an argument of removing intersections between $Q$ and a track of $P$ under an isotopy  
it is shown that $f$ has a factorization $f = agh$ (up to conjugate), where 
$g \in {\rm Diff}(M, Q)_0$, $h \in {\rm Diff}(M, P)_0$ and $a \in {\rm Diff}(M, M_D)_0^c$ for some $D \in {\cal B}_f(M)$. 
Since any compact subset of $M - Q$ is attracted to $P$ under the flow $\phi_t$ and $P$ is displaceable from itself in $M - Q$, 
Basic Lemma~\ref{basic_lemma_cl} induces a factorization of $g$ into commutators. 
Similarly, since any compact subset of $M - P$ is attracted to $Q$ under the flow $\phi_t^{-1}$ and $Q$ is displaceable from itself in $M - P$, 
Basic Lemma~\ref{basic_lemma_cl} induces a factorization of $h$ into commutators. 
This procedure yields a factorization of $f$ into commutators. 

When $M$ is a compact $n$-manifold with nonempty boundary and 
${\cal H}$ is a usual handle decomposition of $M$ (cf. Convention~\ref{convention_handle_decomp}), 
the flow $\phi_t$ is transverse to $\partial M$ and 
$Q$ intersects $\partial M$ transversaly, while $P$ is in ${\rm Int}\,M$. 
Moreover, 
any compact subset of ${\rm Int}\,M - P$ is attracted to the union of $Q$ and $\partial M$ by $\phi_t^{-1}$.
However, the union $Q \cup \partial M$ is not displaceable from itself even if we attach an outer collar to $M$, 
though each summand $Q$ and $\partial M$ is displaceable from $Q \cup \partial M$. 
Thus, the factorization argument breaks down at this point. 

This observation means that, in order to apply Basic Lemma~\ref{basic_lemma_cl} to both $P$ and $Q$, 
we need a flow on $M$ with connecting between $P$ and $Q$ and keeping $\partial M$. 
In this section we see that any triangulation of $M$ induces such flows (Section 3.2). 
We also observe that,
in the case of an open manifold $M$ with a handle decomposition ${\cal H}$, 
the above argument still works if we use 
a pair $(N, L)$ of compact submanifolds of $M$ such that $N \subset L$, $N$ is ${\cal H}$-saturated and $L$ is ${\cal H}^\ast$-saturated 
(cf. Propositions~\ref{prop_2m_cpt_no-m-h}, ~\ref{cl_odd_cpt_bd_handle}, ~\ref{cl_even_cpt_bd_handle}).  
It might be possible to unify these arguments by considering handle decompositions which are transverse to boundary (cf. Section 3.1). 
We postpone such a general approach to a succeeding work.  

\subsection{Handle decompositions of manifolds} 
\subsubsection{\bf Submanifolds saturated with respect to a handle decomposition} \mbox{} 

\begin{convention}\label{convention_handle_decomp}
Suppose $M$ is a $C^\infty$ $n$-manifold possibly with boundary. 
In this article a handle decomposition ${\cal H}$ of $M$ means a locally finite family of $C^\infty$ handles in $M$ of index $k = 0,1, \cdots, n$ 
which satisfies the following conditions : 
(i) $M$ is constructed from ${\cal H}$ by the handle attachment with starting from the 0-handles, that is, 
there exists a sequence of $n$-submanifolds $M_0 \subset M_1 \subset \cdots \subset M_n = M$, where 
$M_0$ is the union of the 0-handles in ${\cal H}$ 
and inductively $M_k$ is obtained by attaching the $k$-handles in ${\cal H}$ to $M_{k-1}$ (with corners being smoothed), 
(ii) any two handles in ${\cal H}$ of same index are disjoint and 
(iii) the handles are attached generically so that 
if $h, k \in {\cal H}$, $h \cap k \neq \emptyset$ and ${\rm ind}\, h > {\rm ind}\, k$, then 
the interior of the attaching region of $h$ intersects the interior of the coattaching region of $k$.
Hence, in our convention  
$\partial M$ occurs only as a portion where no handles are attached. 
\end{convention}

\begin{setting}\label{setting_handle_decomp} 
Below we assume that $M$ is an $n$-manifold without boundary and ${\cal H}$ is a handle decomposition of $M$. 
By ${\cal H}^\ast$ we denote the dual handle decomposition of $M$ for ${\cal H}$. 
For $N \in {\mathcal S}{\mathcal M}(M)$, let ${\cal H}|_N : = \{ h \in {\cal H} \mid h \subset N \}$. 
We say that $N \in {\mathcal S}{\mathcal M}(M)$ is ${\cal H}$-saturated if $N = \bigcup {\cal H}|_N$ and ${\cal H}|_N$ forms a handle decomposition of $N$. 
Let ${\mathcal S}{\mathcal M}(M; {\cal H}) := \{ N \in {\mathcal S}{\mathcal M}(M) \mid \text{$N$ is ${\cal H}$-saturated.} \}$ and 
${\mathcal S}{\mathcal M}_c(M; {\cal H}) := {\mathcal S}{\mathcal M}_c(M) \cap {\mathcal S}{\mathcal M}(M; {\cal H})$. We regard as $\emptyset \in {\mathcal S}{\mathcal M}_c(M; {\cal H})$. 
These notations are also used for ${\cal H}^\ast$. 
\end{setting}

The next lemma easily follows from the definitions.

\begin{lem}\label{H-submfd} \mbox{} In Setting~\ref{setting_handle_decomp} : 
\benum 
\item[$(1)$] For $N \in {\mathcal S}{\mathcal M}(M)$
\bit 
\itemI $N$ is ${\cal H}$-saturated. \LLRA 
\btab[t]{@{}ll}
{\rm (a)} & $N = |{\cal H}|_N|$ and \\[2mm] 
{\rm (b)} & if $h \in {\cal H}|_N$, $k \in {\cal H}$, $h \cap k \neq \emptyset$ and ${\rm ind}\,k < {\rm ind}\,h$, then $k \subset N$. 
\etab 
\vskip 2mm 

\itemII $N$ is ${\cal H}^\ast$-saturated. \LLRA 
\btab[t]{@{}ll}
{\rm (a)} & $N =|{\cal H}|_N|$ and \\[2mm] 
{\rm (b)} & if $h \in {\cal H}|_N$, $k \in {\cal H}$, $h \cap k \neq \emptyset$ and ${\rm ind}\,h < {\rm ind}\,k$, then $k \subset N$. 
\etab 
\eit 
\vskip 2mm 
\item[$(2)$] 
\bit 
\itemI $N_1, N_2 \in {\mathcal S}{\mathcal M}(M; {\cal H})$ \LRA $N_1 \cap N_2, N_1 \cup N_2 \in {\mathcal S}{\mathcal M}(M; {\cal H})$
\itemII $N_1, N_2 \in {\mathcal S}{\mathcal M}(M)$ are disjoint, $N_1 \cup N_2 \in {\mathcal S}{\mathcal M}(M; {\cal H})$ 
\LRA $N_1, N_2 \in {\mathcal S}{\mathcal M}(M; {\cal H})$

\itemiii {\rm (a)} $N \in {\mathcal S}{\mathcal M}(M; {\cal H})$ \LRA $M_N \in {\mathcal S}{\mathcal M}(M; {\cal H}^\ast)$ \\
{\rm (b)} $N \in {\mathcal S}{\mathcal M}(M; {\cal H}^\ast)$ \LRA $M_N \in {\mathcal S}{\mathcal M}(M; {\cal H})$ 

\itemiv Suppose $L, N \in {\mathcal S}{\mathcal M}(M)$ and $L \Subset N$. 
\bit 
\itema $N_L \in {\mathcal S}{\mathcal M}(M; {\cal H})$ \LLRA $L \in {\mathcal S}{\mathcal M}(M; {\cal H}^\ast)$ and $N \in {\mathcal S}{\mathcal M}(M; {\cal H})$
\itemb $N_L \in {\mathcal S}{\mathcal M}(M; {\cal H}^\ast)$ \LLRA $L \in {\mathcal S}{\mathcal M}(M; {\cal H})$ and $N \in {\mathcal S}{\mathcal M}(M; {\cal H}^\ast)$
\eit 
\eit 
\item[$(3)$] For any $L \in {\cal S}{\mathcal M}_c(M)$ there exists
\bit 
\item[] $N_1 \in {\cal S}{\mathcal M}_c(M; {\cal H})$ with $L \subset N_1$ \ and \ 
$N_2 \in {\cal S}{\mathcal M}_c(M; {\cal H}^\ast)$ with $L \subset N_2$. 
\eit 

\eenum 
\end{lem}

We need a more sophisticated version of Lemma~\ref{lem_factorization_open}\,(2). 

\begin{lem}\label{lem_factorization_open-2} 
Suppose $M$ is an open $n$-manifold, ${\cal H}$ is a handle decomposition of $M$, 
${\mathcal R}$ is a cofinal subcollection of ${\mathcal S}{\mathcal M}_c(M)$, 
$K_k \in {\mathcal S}{\mathcal M}_c(M)$ $(k \geq 1)$ and  
${\cal F}$ is a locally finite family of compact subsets of $M$.  
Then, for any $F \in {\rm Isot}^r(M)_0$ the following hold : 
\benum[(1)] 
\item[{\rm [I]}] There exists an exhausting sequence $\{ M_k \}_{k \geq 1}$ of $M$ 
which satisfies the following conditions : 

\benum
\item[$(1)$] {\rm (i)} $K_k \Subset M_k \in {\mathcal R}$ $(k \geq 1)$ \hsh {\rm (ii)}  
$F(M_{4k,4k+1} \times I) \Subset M_{4k-1, 4k+2} \times I$ $(k \geq 0)$ 

\item[$(2)$] There exist $N_k'  \in {\cal S}{\cal M}_c(M; {\cal H})$ and $N_k'' \in {\cal S}{\cal M}_c(M; {\cal H}^\ast)$ $(k \geq 0)$ 
such that for each $k \geq 0$ 
\bit 
\itema $M_{4k-1, 4k+2} \FSubset N_k' \FSubset N_k'' \ \Subset \ M_{4k-2, 4k+3}$ 
\ \ and \ \ {\rm (b)} $N_k'' \cap N_{k+1}'' = \emptyset$. 
\eit 

\item[$(3)$] There exist $L_k'  \in {\cal S}{\cal M}_c(M; {\cal H})$ and $L_k'' \in {\cal S}{\cal M}_c(M; {\cal H}^\ast)$ $(k \geq 1)$ such that for each $k \geq 1$ 
\bit 
\itema $M_{4k-3, 4k} \FSubset L_k' \FSubset L_k'' \ \Subset \ M_{4k-4, 4k+1}$  
\ \ and \ \ {\rm (b)} $K_1 \cap L_1'' = \emptyset$, \ $L_k'' \cap L_{k+1}'' = \emptyset$. 
\eit 
\eenum 
Here, the notation $L \FSubset N$ means that $St(L, {\cal F}) \Subset N$. 
\item[{\rm [II]}] There exists a factorization $F = GH$ for some $G \in {\rm Isot}^r(M, M')_0$ and $H \in {\rm Isot}^r(M, M'')_0$, \\
\hspace*{30mm} where \ $M' := \bigcup_{k \geq 0} M_{4k+2, 4k+3}$ \ and \ $M'' := \bigcup_{k \geq 0} M_{4k, 4k+1}$. \\
In addition, we have the following estimates for the commutator lengths $q = cl$ and $clb^d$ : 
\bit 
\item[$(1)$] 
$q(G_1) \leq \ell_1$ in ${\rm Diff}^r(M)_0$ 
if \ $qd({\rm Diff}^r(M, M_{M_{4k-1, 4k+2}})_0, {\rm Diff}^r(M, M_{N_k''})_0) \leq \ell_1$ for each $k \geq 0$. 
\item[$(2)$] $q(H_1) \leq \ell_2$ in ${\rm Diff}^r(M)_0$ 
if \ $qd({\rm Diff}^r(M, M_{M_{4k-3, 4k}})_0, {\rm Diff}^r(M, M_{L_k''})_0) \leq \ell_2$ for each $k \geq 1$. 
\item[$(3)$] $q(F_1) \leq \ell_1 + \ell_2 $ in ${\rm Diff}^r(M)_0$ if $q(G_1) \leq \ell_1$ and $q(H_1) \leq \ell_2$. 
\eit 
\eenum 
\end{lem} 

\begin{proof}
{[I]} We repeat the proof of Lemma~\ref{lem_factorization_open}\,(2) to achieve the condition (1).   
In this inductive construction of $\{ M_k \}_{k \geq 1}$  
we can take $M_{k+1}$ arbitrarily large for a given $M_k$. 
To achieve the conditions (2), (3), in the sequence $\{ M_k \}_{k \geq 1}$ we insert the following compact $n$-submanifolds of $M$, 
based on Lemma~\ref{H-submfd}\,(3). \\[1mm]
\hspace*{25mm} $N_{k,-}', \ N_{k,+}', \ N_{k,-}'', \ N_{k,+}''$ \ $(k \geq 0)$, \hsh 
$L_{k,-}', \ L_{k,+}', \ L_{k,-}'', L\ _{k,+}''$ \ $(k \geq 1)$ \\[1mm] 
We choose these submanifolds so to satisfy the following conditions :

\bit 
\itemI (a) \ $N_{0,-}' = N_{0, -}'' = \emptyset$, \hsh $M_{4k+2} \FSubset N_{k,+}' \FSubset N_{k,+}'' \Subset N_{k+1,-}'' \FSubset N_{k+1,-}' \FSubset M_{4k+3}$ \ \ $(k \geq 0)$ \\[0.5mm] 
(b) \ $K_1 \Subset L_{1,-} '' \FSubset L_{1,-}' \FSubset M_1$,  \\[0.5mm] 
\hspace*{6mm} $M_{4k} \FSubset L_{k,+}' \FSubset L_{k,+}'' \Subset L_{k+1,-}'' \FSubset L_{k+1,-}' \FSubset M_{4k+1}$ \ \ $(k \geq 1)$
\vskip 1mm 
\itemII (a) \ $N_{k,+}', N_{k,-}'' \in {\cal S}{\mathcal M}_c(M; {\cal H})$, \ 
$N_{k,-}', N_{k,+}'' \in {\cal S}{\mathcal M}_c(M; {\cal H}^\ast)$ \ $(k \geq 0)$ \\[0.5mm] 
(b) \ $L_{k,+}', L_{k,-}'' \in {\cal S}{\mathcal M}_c(M; {\cal H})$, \ 
$L_{k,-}', L_{k,+}'' \in {\cal S}{\mathcal M}_c(M; {\cal H}^\ast)$ \ $(k \geq 1)$ 
\eit  
Note that $St(K, {\cal F})$ is compact for any $K \in {\cal K}(M)$. 
Then, from Lemma~\ref{H-submfd}\,(2)(iv) it follows that \\[0.5mm] 
\hspace*{3mm} $N_k' := (N_{k,+}')_{N_{k,-}'}$, \ $N_k'' := (N_{k,+}'')_{N_{k,-}''}$ \ $(k \geq 0)$ 
\ \ and \ \ $L_k' := (L_{k,+}')_{L_{k,-}'}$, \ $L_k'' := (L_{k,+}'')_{L_{k,-}''}$ \ $(k \geq 1)$ \\[1mm] 
satisfy the required conditions (2) and (3). 

{[II]} By [I](1)(ii) and Lemma~\ref{lem_factorization_open}\,(3) we have the factorization $F = GH$. 
\benum 
\item For each $k \geq 0$ consider $g_k \in {\rm Diff}^r(M, M_{M_{4k-1,4k+2}})_0$ defined by $g_k = G_1$ on $M_{4k-1,4k+2}$. 
Then $q(g_k) \leq \ell_1$ in ${\rm Diff}^r(M, M_{N_k''})_0$ by the assumption. 
Since $G_1$ is ``the discrete sum'' of $g_k$ $(k \geq 0)$ by [I](2), it follows that $q(G_1) \leq \ell_1$ in ${\rm Diff}^r(M)_0$.  
\eenum

The statment (2) is shown by a similar argument for $H$. 
\end{proof}

\subsubsection{\bf The gradient flow and the complex associated to a handle decomposition} \mbox{} 

\begin{setting}\label{setting_handle_decomp_dual-skeleton} 
Suppose $M$ is an $n$-manifold without boundary and ${\cal H}$ is a handle decomposition of $M$. 
\end{setting} 

Each handle of index $m$ in ${\cal H}$ has the core $m$-disk and a cocore $(n-m)$-disk. These disks intersect at the center point transversely. 
There exists a Morse function $f$ on $M$ and a Riemannian metric on $M$ with the properties that 
$f$ has a critical point of index $m$ at the center point of each handle of index $m$ and that 
with respect to the induced (downward) gradient flow $\phi_t$ on $M$ the unstable manifold of each critical point of $f$ of index $m$ is  
an open $m$-disk and the total of those disks provides $M$ with a structure of a $C^\infty$ CW-complex, which we denote by $P_{\cal H}$ and call the core complex of ${\cal H}$ for short. As usual $P_{\cal H}^{(k)}$ denotes the $k$-skeleton of the complex $P_{\cal H}$. 
The core complex $P_{\cal H^\ast}$ of ${\cal H}^\ast$ is obtained by considering the stable manifolds of each critical point of $f$ under the flow $\phi_t$. 
The flow $\phi_t$ keeps each cell and skeleton of $P_{\cal H}$ and $P_{\cal H^\ast}$ invariant. 
(When $\partial M \neq \emptyset$, the partial gradient flow $\phi_t$ $(t \geq 0)$ is defined and it is transverse to $\partial M$.) 
Each $k$-handle of ${\cal H}$ determines a unique $k$-cell of $P_{\cal H}$ and a unique $(n-k)$-cell of $P_{\cal H^\ast}$. 
To each $k$-cell $\sigma$ of $P_{\cal H}$ the associated $(n-k)$-cell of $P_{\cal H^\ast}$ is denoted by $\sigma^\ast$ and called the dual cell to $\sigma$. 

\begin{setting_3.2+} 
Let $P = P_{\cal H}^{(k)}$ and $Q = P_{\cal H^\ast}^{(n-k-1)}$ $(k=0,1, \cdots, n)$. 
\end{setting_3.2+}

The gradient flow $\phi_t$ has the following properties : 
$\phi_t(M - Q) = M - Q$ $(t \in \IR)$ and $M - Q$ is attracted to $P$. 

\begin{setting_3.2++} 
Suppose $N \in {\cal S}{\cal M}_c(M; {\cal H})$ and 
$P_N := (P_{{\cal H}|_N})^{(k)}$ (the $k$-skeleton of the core complex of ${\cal H}|_N$). 
\end{setting_3.2++}

Then $P_N$ is a finite subcomplex of $P$ and  $P_N \subset {\rm Int}\,N$. The gradient flow $\phi_t$ satisfies the following conditions : 
The flow $\phi_t$ is transverse to $\partial N$, 
$\phi_t(N) \subset N$, $\phi_t(N - Q) \subset N - Q$ and $N - Q$ is attracted to $P_N$.   
Therefore, the following holds : 
\bit 
\item[$(\ast)$] $P_N$ has the weak absorption property in ${\rm Int}\,N - Q$ keeping $P$ invariant. 
\eit 
More precisely, for any $C \in {\cal K}({\rm Int}\,N - Q)$ and any neighborhood $U$ of $P_N$ in ${\rm Int}\,N - Q$
there exists $H \in {\rm Isot}(M; M_{N} \cup Q \cup P_N)_0$ such that $H_1(C) \subset U$, $H_1(P) = P$ and 
$H_t$ makes each cell of $P_{\cal H}$ and $P_{\cal H^\ast}$ invariant. 
In fact, $H$ is obtained by cutting off the gradient vector field associated to $\phi_t$ around $M_{N} \cup Q \cup P_N$ and 
truncating the induced flow.  

To apply Lemma~\ref{basic_lemma_cl-bdry_final} in this situation, 
we need an estimate on $clb^f (P_N; {\rm Int}\,N - Q)$ (cf. Complement~\ref{compl_abs_dis}\,(2)). 
The next lemma follows from \cite[Proof of Theorem 2.2, Remark 2.1]{Tsu09} and 
provides with this estimate.  
The following notations are introduced in \cite{Tsu09}. 

\begin{notation}\label{notation_c(P)} Suppose $M$ is an $n$-manifold possibly with boundary. 
\benum 
\item For a stratified subset $P$ of $M$, let 
$c(P) := \# \{ i \in \IZ_{\geq 0} \mid P^{(i)} - P^{(i-1)} \neq \emptyset \} \leq \dim P +1 \leq n+1$. 

\item Suppose ${\cal H}$ is a handle decomposition of $M$. 
The index of a handle $h \in {\cal H}$ is denoted by ${\rm ind}\,h$. 
For any subset ${\cal C} \subset {\cal H}$, let $c({\cal C}) : = \# \{ {\rm ind}\,h \mid h \in {\cal C} \} \leq n+1$.  
Note that $c({\cal H}) = c(P_{\cal H})$ and $c({\cal H}^{(k)})= c(P_{\cal H}^{(k)}) \leq k+1$, 
where ${\cal H}^{(k)} := \{ h \in {\cal H} \mid \text{the index of $h$} \leq k \}$. 
\eenum 
\end{notation}

\begin{lem}\label{lem_reg_nbd} 
Suppose $M$ is an $n$-manifold possibly with boundary, ${\cal H}$ is a handle decomposition of $M$, 
$O \in {\cal O}({\rm Int}\,M)$ and $K \in {\cal K}(O)$. 
If $K$ is a finite subcomplex of $P_{\cal H}$, then $clb^f(K, O) \leq c(K)$. 
\end{lem} 

More precisely, the following holds : There exists a compact neighborhood $U$ of $K$ in $O$ which has the following property. 
\benum
\item[$(\#)$] 
For any neighborhoods $W$ of $K$ and $V$ of $U$ in $O$ there exists $H \in {\rm Isot}_c(M, M_{V} \cup K)_0$ such that 
\bit 
\itema $H_1(U) \subset W$, \hsh {\rm (b)} $H_t(\sigma) = \sigma$ $(\sigma \in  P_{\cal H}, t \in I)$ \ \ and 
\itemc $H$ has a factorization $H = H^{(1)} \cdots H^{(c(K))}$ such that for each $i = 1,2, \cdots, c(K)$ 
there exists $D_i \in {\cal B}_f({\rm Int}_MV)$ for which 
$H^{(i)} \in {\rm Isot}(M, M_{D_i} \cup K)_0$ and $H^{(i)}_1 \in {\rm Diff}(M, M_{D_i})_0^c$ \\
$($in particular, $clb^f H_1 \leq c(K)$ in ${\rm Diff}_c(M, M_{V})_0)$. 
\eit  
\eenum 

The isotopy $H^{(i)}$ is obtained by cutting off the gradient flow $\phi_t$ $(t \geq 0)$ along the open $k(i)$-cells of $K$ 
for the dimensions $0 = k(1) < \cdots < k(c(K)) \leq n$ of cells which appear in $K$.  

From Lemmas~\ref{basic_lemma_cl-bdry_final}\,[B] and ~\ref{lem_reg_nbd} we have the following conclusion. 

\begin{lem}\label{factorization_hdle_cpt} 
In Setting~{\rm \ref{setting_handle_decomp_dual-skeleton}, ~\ref{setting_handle_decomp_dual-skeleton}$^+$, 
~\ref{setting_handle_decomp_dual-skeleton}$^{++}$}; \ 
If $1 \le r \le \infty, \, r \neq n+1$ and $2k < n$,  
then the following holds. 
\benum
\item[$(1)$] {\rm (i)} $cld\,{\rm Diff}^r(M, M_N \cup Q)_0 \leq 2$, \ \ {\rm (ii)} $clb^f\!d\,{\rm Diff}^r(M, M_N \cup Q)_0 \leq 2c(P_N)+1$. 
\item[$(2)$] Any $f \in {\rm Diff}^r(M, M_N \cup Q)_0$ has a factorization $f = gh$ \ such that  \\
\hsp 
\btab[t]{lll} 
$g \in {\rm Diff}^r(M, M_N \cup Q)_0^c$ & and & $clb^f(g) \leq 2c(P_N)$ \ in ${\rm Diff}^r(M, M_N \cup Q)_0$, \\[2mm] 
$h \in {\rm Diff}^r(M, M_N  \cup P \cup Q)_0^c$ & and & $clb^f(h) \leq1$ \ in ${\rm Diff}^r(M, M_N  \cup P \cup Q)_0$. 
\etab 
\eenum 
\end{lem} 

\begin{proof} We apply Lemma~\ref{basic_lemma_cl-bdry_final}\,[B]\,$(\ast)_2$ to $O = {\rm Int}\,N - Q$, $L = P$ and $K := P_N \in {\cal K}(O)$. 
Note that $M_O = M_N \cup Q$ and $K \subset L$. It follows that (i) $clb^f(K, O) \leq c(K)$ by Lemma~\ref{lem_reg_nbd},  
(ii) $K$ has the weak absorption property in $O$ rel $L$ by Setting~\ref{setting_handle_decomp_dual-skeleton}$^{++}$\,$(\ast)$ and 
(iii) $K$ has the strong displacement property for $L$ in $M$, 
since $\dim K \leq \dim L \leq k$ and $2k < n$ (cf. Example~\ref{exp_displacement}\,(1)(i)). 
\end{proof}

\subsubsection{\bf Even-dimensional case} \mbox{} 

Here we include some basic cases in the even-dimension, which are used in Section 5 to treat the general cases. 
First we consider the case of handle decompositions without $m$-handles. 

\begin{prop}\label{prop_2m_cpt_no-m-h}
Suppose $M$ is a $2m$-manifold without boundary, $1 \le r \le \infty, \, r \neq 2m+1$, 
${\cal H}$ is a handle decomposition of $M$, 
$N \in {\cal S}{\cal M}_c(M; {\cal H})$, $L \in {\cal S}{\cal M}_c(M; {\cal H}^\ast)$ and $N \subset L$. 
If ${\cal H}$ has no $m$-handles in $L$, then 
\benum 
\item[$(1)$] $cld\,({\rm Diff}^r(M, M_N)_0,  {\rm Diff}^r(M, M_L)_0) \leq 3$ \ \ and 
\item[$(2)$] $clb^f\!d\,({\rm Diff}^r(M, M_N)_0,  {\rm Diff}^r(M, M_L)_0) \leq 2c({\cal H}|_L)+1$.  
\eenum 
\end{prop}

\begin{proof} 
Let $P = P_{\cal H}^{(m)}$ and $Q = P_{\cal H^\ast}^{(m)}$.
By the assumption any open $m$-cell of $P$ does not intersect $L$ and any $m$-cell of $Q$ does not intersect $N$.  
This implies that \\
\hspace*{10mm} $P \cap L = P^{(m-1)} \cap L$, \ $Q \cap N = Q^{(m-1)} \cap N$, \ $P \cap Q \cap L = \emptyset$ \ and \ $M_N \cup P = M_N \cup P^{(m-1)}$. 

Given any $f \in {\rm Diff}^r(M, M_{N})_0$. Take $F \in {\rm Isot}^r(M, M_{N})_0$ with $F_1 = f$. 
Then there exists $\Phi \in {\rm Isot}(M, M_N \cup P)_0$ such that 
$Q_1 \cap \pi_M F((P \cap L) \times I) = \emptyset$ for $Q_1 := \Phi_1(Q)$. 
By Lemma~\ref{Isotopy_ext} we obtain a factorization $F = GH$, where 
$G \in {\rm Isot}^r(M, M_{N} \cup Q_1)_0$ and $H \in {\rm Isot}^r(M, M_{N} \cup P)_0$. 

Applying Lemma~\ref{factorization_hdle_cpt} to $(M, \Phi_1({\cal H}), N, P^{(m-1)}, Q_1)$, it follows that \ 
$G_1 \in {\rm Diff}^r(M, M_N \cup Q_1)_0$ \ has a factorization \ $G_1 = g_1g_2$ \\
\hsh for some 
\btab[t]{ll}
$g_1 \in {\rm Diff}^r(M, M_N \cup Q_1)_0^c$ & with \ $clb^f g_1 \leq 2c\big(P_{{\cal H}|_N}^{\ \ (m-1)}\big)$ in ${\rm Diff}^r(M, M_N \cup Q_1)_0$ \ and \\[3mm]
$g_2 \in  {\rm Diff}^r(M, M_N \cup P \cup Q_1)_0^c$ & with \ $clb^fg_2 \leq 1$ in ${\rm Diff}^r(M, M_N \cup P \cup Q_1)_0$. 
\etab \\[2mm]  
Since $g_2H_1 \in {\rm Diff}^r(M, M_N \cup P)_0 \subset {\rm Diff}^r(M, M_L \cup P)_0$, we can apply Lemma~\ref{factorization_hdle_cpt} to 
$(M, {\cal H}^\ast, L, Q^{(m-1)}, P)$ to obtain 
a factorization \ $g_2H_1 = h_1h_2$ \ for some $h_1, h_2 \in {\rm Diff}^r(M, M_L \cup P)_0^c$ \ with \\[1mm] 
\hspace{30mm} $clb^f h_1 \leq 2c\big(P_{\cal H^\ast|_L}^{\ \ (m-1)}\big)$ \ and \ $clb^fh_2 \leq 1$ \ in ${\rm Diff}^r(M, M_L \cup P)_0$. 

\benum
\item $f = G_1H_1 = g_1g_2H_1 = g_1h_1h_2$ \ and \ $g_1, h_1, h_2 \in {\rm Diff}^r(M, M_L)_0^c$. 
Hence, $cl\,f \leq 3$ \ in \ ${\rm Diff}^r(M, M_L)_0$.
\item 
Since $c\big(P_{{\cal H}|_N}^{\ \ (m-1)}\big) + c\big(P_{\cal H^\ast|_L}^{\ \ (m-1)}\big) \leq c({\cal H}|_L)$,   
it follows that, in ${\rm Diff}^r(M, M_L)_0$ \\[1mm] 
\hspace{3mm} $f = g_1h_1h_2$ \ \ and \ \ 
$clb^f\,f \leq 2c\big(P_{{\cal H}|_N}^{\ \ (m-1)}\big) + 2c\big(P_{\cal H^\ast|_L}^{\ \ (m-1)}\big) + 1 \leq 2c({\cal H}|_L) +1$. 
\eenum 
\vskip -7.5mm 
\end{proof} 

Next we consider a basic situation related to $m$-handles. 
Here we need to force the displacement property for $m$-cells to obtain some estimates on $cl$ and $clb^f$. 

\begin{setting}\label{setting_handle_even_cpt} 
Suppose $M$ is a $2m$-manifold without boundary, ${\cal H}$ is a handle decomposition of $M$, 
$P = P_{\cal H}^{(m)}$, $Q = P_{\cal H^\ast}^{(m)}$ and 
${\cal C}$, ${\cal C}'$ are some sets of open $m$-cells of $P$ with ${\cal C} \subset {\cal C}'$. 
\end{setting}

The sets of $m$-cells of $P$ and $Q$ are in the 1-1 correspondence $\sigma \llra \sigma^\ast$, 
where $\sigma^\ast$ is the dual $m$-cell to $\sigma$. 
Hence, the set ${\cal C}$ determines 
the set ${\cal C}^\ast = \{ \sigma^\ast \mid \sigma \in {\cal C} \}$ of open $m$-cells of $Q$. 
For simplicity, we set $Q_{\cal C} := Q - |{\cal C}^\ast|$. 
In this context, the gradient flow $\phi_t$ has the following properties : 
$\phi_t(M - Q_{\cal C}) = M - Q_{\cal C}$ and 
$M - Q_{\cal C}$ is attracted to $P^{(m-1)} \cup |{\cal C}|$. 

\begin{setting_3.3+} 
Suppose $N \in {\cal S}{\cal M}_c(M; {\cal H})$, 
$P_N = (P_{{\cal H}|_N})^{(m)}$, 
${\cal C}|_N := \{ \sigma \in {\cal C} \mid \sigma \subset N \}$, 
$O:= {\rm Int}\,N - Q_{\cal C}$, $K := P_N^{(m-1)} \cup |{\cal C}|_N|$ and $L = P^{(m-1)} \cup |{\cal C}'|$. 
\end{setting_3.3+}

Then $M_O = M_{N} \cup Q_{\cal C}$, $K \in {\cal K}(O)$, $L \in {\cal F}(M)$, $K \subset L$, $K$ is a subcomplex of $P_N$ and 
$N \cap  |{\cal C}^\ast| = N \cap  |({\cal C|_N})^\ast|$.  The flow $\phi_t$ has the following properties : 
The flow $\phi_t$ is transverse to $\partial N$, 
$\phi_t(N) \subset N$, $\phi_t(N - Q_{\cal C}) \subset N - Q_{\cal C}$, $N - Q_{\cal C}$ is attracted to $K$ 
and $\phi_t$ keeps each cell of $P_{\cal H}$ and $P_{\cal H^\ast}$ invariant.  
Therefore, the following holds : 
\bit 
\item[$(\sharp)$] $K$ has the weak absorption property in $O$ keeping $P$, $K$ and $L$ invariant.  
\eit
More precisely, for any $C \in {\cal K}(O)$ and any neighborhood $U$ of $K$ in $O$
there exists $H \in {\rm Isot}(M; M_O \cup K)_0$ such that $H_1(C) \subset U$ 
and $H_t$ keeps each cell of $P_{\cal H}$ and $P_{\cal H^\ast}$ invariant. 

\begin{lem}\label{lem_even_cpt_handle}  
In Settings{\rm ~\ref{setting_handle_even_cpt}, {}~\ref{setting_handle_even_cpt}$^+$}; 
Suppose $1 \le r \le \infty, \, r \neq 2m+1$. \\
If $Cl_M |{\cal C}|_N|$ has the displacement property for $Cl_M |{\cal C}'|$ in ${\rm Int}\,N - Q_{\cal C}$, then the following holds. 
\benum
\item[$(1)$]\hspace{0.2mm}{\rm (i)}\, $cld\,{\rm Diff}^r(M, M_{N} \cup Q_{\cal C})_0 \leq 2$, \\
{\rm (ii)} $clb^f\!d\,{\rm Diff}^r(M, M_{N} \cup Q_{\cal C})_0 \leq 2c(K)+1$. 
\item[$(2)$] Any $f \in {\rm Diff}^r(M, M_{N} \cup Q_{\cal C})_0$ has a factorization $f = gh$ \ such that \\ 
\hsp \hsh \btab[t]{@{}l@{ \ }l@{ \ }l}
$g \in {\rm Diff}^r(M, M_{N} \cup Q_{\cal C})_0^c$ & and & $clb^f(g) \leq 2c(K)$ in ${\rm Diff}^r(M, M_{N} \cup Q_{\cal C})_0$ \\[2mm] 
$h \in {\rm Diff}^r(M, M_{N} \cup Q_{\cal C} \cup L)_0^c$ & and & 
$clb^f(h) \leq 1$ in ${\rm Diff}^r(M, M_{N} \cup Q_{\cal C} \cup L)_0$. 
\etab 
\eenum 
\end{lem}

\begin{proof} We apply Lemma~\ref{basic_lemma_cl-bdry_final}\,[B]\,$(\ast)_1$ to $O = {\rm Int}\,N - Q$, $L$ and $K$. 
It follows that (i) $clb^f(K, O) \leq c(K)$ by Lemma~\ref{lem_reg_nbd}, 
(ii) $K$ has the weak absorption property in $O$ rel $L$ and keeping $K$ invariant by Setting~\ref{setting_handle_even_cpt}$^+$\,$(\sharp)$
and (iii) $K$ has the displacement property for $L$ in $O$ by Example~\ref{exp_displacement}\,(1)(ii)(a). 
\end{proof}

\subsection{Triangulations of manifolds} \mbox{} 
\subsubsection{\bf Generalities on triangulations of manifolds} \mbox{} 

Suppose $M$ is an $n$-manifold possibly with boundary 
and ${\mathcal T}$ is a $C^\infty$ triangulation of $M$. 
For each (closed) simplex $\sigma \in {\cal T}$, by the symbol $\stackrel{\circ}{\sigma}$ 
we denote the interior of $\sigma$ (the open simplex for $\sigma$). 
The symbol ${\mathcal T}^{(k)}$ $(k = 0,1,\cdots, n)$ denotes the $k$-skeleton of ${\mathcal T}$ and 
$|{\mathcal T}^{(k)}|$ denotes its realization in $M$, i.e., 
$|{\mathcal T}^{(k)}| = \bigcup \{ \sigma \mid \sigma \in {\mathcal T}^{(k)} \} \subset M$. 
Let $sd^i \hspace{0.2mm}\mathcal{T}$ $(i \geq 0)$ denote the $i$-th barycentric subdivision of  ${\mathcal T}$. 
Sometimes we do not distinguish a subcomplex ${\cal S}$ of ${\cal T}$ and its realization $|{\cal S}|$ in $M$. 

For each $k$-simplex $\sigma$ of ${\cal T}$ its barycenter is denoted by $b(\sigma)$. 
The simplex $\sigma$ determines the dual (closed) $(n-k)$-cell $\sigma^\ast$ and the handle $H_\sigma$ of index $k$. 
These are (piecewise $C^\infty$) PL-disks defined by \\[2.5mm] 
\hspace*{22mm} $\begin{array}[c]{l}
\sigma^\ast = \bigcup \big\{ \langle b(\sigma_k), b(\sigma_{k+1}), \cdots, b(\sigma_n) \rangle \mid 
\sigma = \sigma_k < \sigma_{k+1} < \cdots < \sigma_n \text{ in } {\cal T} \big\} \ \ \text{and} \\[3mm]
H_\sigma = St(b(\sigma), sd^2\,{\mathcal T}) = \bigcup \{ \tau \in sd^2\,{\mathcal T} \mid b(\sigma) \in \tau \}. 
\end{array}$ \\[2.5mm] 
Note that, if $\sigma \subset \partial M$, then $\sigma^\ast$ is a ``half'' disk and $H_\sigma$ is a ``half'' handle cut by $\partial M$ transversely.  
The handle $H_\sigma$ has the core disk $H_\sigma \cap \sigma$ and the cocore $H_\sigma \cap \sigma^\ast$. 
The triangulation ${\cal T}$ induces 
the dual PL cell decomposition ${\cal T}^\ast = \{ \sigma^\ast \mid \sigma \in {\cal T}\}$ and 
the PL handle decomposition ${\cal H}_{\cal T} = \{ H_\sigma \mid \sigma \in {\cal T} \}$ of $M$. 
The $(n-k-1)$-skeleton $({\cal T}^\ast)^{(n-k-1)}$ of ${\cal T}^\ast$ is called 
the dual $(n-k-1)$-skeleton for the $k$-skeleton ${\mathcal T}^{(k)}$ of ${\cal T}$. 
Under our convention in Section 3.1 this handle decomposition does not satisfy our requirement. 
However, in this bounded PL-case we permit this kind of handle decompositions. 

Associated to the triangulation ${\cal T}$, there exists  
a (downward) $C^\infty$ gradient flow $\phi_t$ $(t \in \IR)$ on $M$  
induced from a Riemannian metric of $M$ and 
a Morse function on $M$ which has a critical point of index $k$ at $b(\sigma)$ for each $k$-simplex $\sigma$ of ${\cal T}$. 
The cell decomposition ${\cal T}$ is recovered as the unstable cell complex associated to the flow $\phi_t$ and 
the dual cell decomposition ${\cal T}^\ast$ is approximated by 
the stable cell complex ssociated to $\phi_t$. 
When $\partial M \neq \emptyset$, the flow $\phi_t$ keeps $\partial M$, while the partial gradient flow associated to a handle decomposition in Section 3.1 is transverse to $\partial M$. 

A $C^\infty$ subpolyhedron of $M$ means a subcomplex of some $C^\infty$ triangulation of $M$. 

\begin{lem}\label{lem_reg_nbd_2} 
Suppose $M$ is an $n$-manifold possibly with boundary, ${\mathcal T}$ is a $C^\infty$ triangulation of $M$, 
$O \in {\cal O}({\rm Int}\,M)$ and $K$ is a finite subcomplex of ${\cal T}$ with $K  \subset O$. Then, the following holds. 
\benum
\item[$(1)$] $clb^f(K, O) \leq \dim K+1$. 
\item[$(2)$] If $K \subset K' \in {\cal K}(O)$ and $K'$ is weakly absorbed to $K$ in $O$ keeping $K$ invariant, then \\
\hspp $clb^f(K'; O) \leq clb^f(K;O) \leq \dim K+1$. 
\eenum 
\end{lem} 

For example, the assumption in (2) is satisfied if $K'$ is a finite subcomplex of ${\cal T}$ with $K \subset K' \subset O$ and $K'$ collapses to $K$ in the PL-sense. 

\begin{proof} (1) One may verify this lemma, using the flow $\phi_t$ as in Lemma~\ref{lem_reg_nbd}, based on \cite[Proof of Theorem 2.2, Remark 2.1]{Tsu09}. 
Here, we use a regular neighborhood of $K$ to obtain a simpler discription. 
Take a fine subdividion ${\cal T}'$ of ${\cal T}$ with $St(K, {\cal T}') \subset O$ and 
consider the regular neighborhood $U := St(K, sd^2 {\cal T}')$ of $K$ with respect to ${\cal T}'$. 
Then $U$ has the following property : 

\benum[(1)] 
\item[] \hspace{-4mm} $(\sharp)$ 
For any neighborhoods $W$ of $K$ and $V$ of $U$ in $O$ there exists $H \in {\rm Isot}_c(M, M_{V} \cup K)_0$ such that 
\bit 
\itema $H_1(U) \subset W$, \hsh {\rm (b)} $H_t(\sigma) = \sigma$ $(\sigma \in  {\cal T}', t \in I)$ \ \ and 
\itemc $H$ has a factorization $H = H^{(0)} H^{(1)} \cdots H^{(k)}$ $(k=\dim K)$ \ such that for each $i = 0,1, \cdots, k$ \\ 
there exists $D_i \in {\cal B}_f({\rm Int}_MV)$ for which 
$H^{(i)} \in {\rm Isot}(M, M_{D_i} \cup K)_0$ and $H^{(i)}_1 \in {\rm Diff}(M, M_{D_i})_0^c$ \\
$($in particular, $clb^f H_1 \leq k+1$ in ${\rm Diff}_c(M, M_{V})_0)$. 
\eit  
\eenum 

In fact, we can write as 
$U = \bigcup \{ H_\sigma \mid \sigma \in K' \}$, 
where $K' = {\cal T}'|_K$ (the subdivision of $K$ induced by ${\cal T}'$) and 
$H_\sigma = St(b(\sigma) , sd^2 {\cal T}')$ (the handle associated to $\sigma$ in ${\cal T}'$). 
Then for each $i$-simplex $\sigma$ of $K'$ (with $H_\sigma \not\subset {\rm Int}_MK$) 
we can compress  the handle $H_\sigma$ 
toward the union of $H_\sigma \cap K$ and the attaching region of $H_\sigma$ in a small disk neighborhood of $H_\sigma$. 
The isotopy $H^{(i)}$ is obtained by this procedure and the modification due to \cite[Remark 2.1]{Tsu09}. 

The statement (2) follows from Lemma~\ref{lem_clb^f_basic}\,(1)(ii). 
\end{proof}

\subsubsection{\bf Double mapping cylinder structures between complimentary full subcomplexes} \mbox{} 

We further observe that the triangulation ${\cal T}$ of a manifold $M$ provides $M$ with 
a more rigid structure than the induced handle decomposition ${\cal H}_{\cal T}$ and the gradient flow $\phi_t$. 
Especially, around any subcomplex of ${\cal T}$ we can take a regular neighborhood with a mapping cylinder structure. 
This structure also induces some flow, which has a simple discription and is used to construct absorbing diffeomorphisms. 
This section is devoted to this approach, which yields some basic estimates on $cl$ and $clb^f$ in compact manifolds possibly with boundary.  

\begin{setting}\label{setting_mapping_cylinder} 
Suppose $M$ is an $n$-manifold possibly with boundary, 
${\mathcal T}$ is a $C^\infty$ triangulation of $M$, ${\cal S}$ is a full subcomplex of ${\cal T}$ and ${\cal U} := \{ \sigma \in {\cal T} \mid \sigma \cap |{\cal S}| = \emptyset \}$. 
\end{setting}

Then ${\cal U}$ is also a full subcomplex of ${\cal T}$ and ${\cal S} = \{ \sigma \in {\cal T} \mid \sigma \cap |{\cal U}| = \emptyset \}$.
In this case ${\cal S}$ and ${\cal U}$ are said to be complimentary to each other in ${\cal T}$. 
Also note that any $\sigma \in {\cal T}$ is represented as $\sigma = (\sigma \cap {\cal S}) \ast (\sigma \cap {\cal U})$ (the join of $(\sigma \cap {\cal S})$ and $(\sigma \cap {\cal U})$). 
Consider the standard PL regular neighborhoods 
$St(|{\cal S}|, sd\, {\cal T})$ of $|{\cal S}|$ and 
$St(|{\cal U}|, sd\, {\cal T})$ of $|{\cal U}|$. 
These regular neighborhoods have the structure of PL mapping cylinder with the base 
$|{\cal S}|$ and $|{\cal U}|$ respectively, and 
$M$ has the structure of PL double mapping cylinder as the union of these mapping cylinders. 
This structure respects $\partial M$. 
From these observations we have a $C^\infty$ flow $\psi_t$ on $M$ which fixes $|{\cal S}|$ and $|{\cal U}|$ pointwise and 
keeps each simplex of ${\cal T}$ invariant and under which $M- |{\cal U}|$ is attracted to $|{\cal S}|$. 
The associated vector field $X_\psi$ on $M$ is constructed inductively on a neighborhood of ${\cal T}^{(k)}- ({\cal S} \cup {\cal U})$ for $k=1, \cdots, n$, so that $X_\psi$ is tangent to each simplex in ${\cal T} - ({\cal S} \cup {\cal U})$. 
Therefore, the following holds : 
\bit 
\item[$(\natural)$] For any $C \in {\cal K}(M - |{\cal U}|)$ and any neighborhood $U$ of $|{\cal S}|$ in $M - |{\cal U}|$
there exists $H \in {\rm Isot}_c(M; |{\cal S}| \cup |{\cal U}|)_0$ such that $H_1(C) \subset U$ and $H_t(\sigma) = \sigma$ $(\sigma \in {\cal T}, t \in I)$. 
If $C \in {\cal K}({\rm Int}\,M - |{\cal U}|)$, then we can take $H$ in ${\rm Isot}_c(M; \partial M \cup |{\cal S}| \cup |{\cal U}|)_0$.
\eit 

Note that, if ${\cal P}$ is any subcomplex of ${\cal T}$, 
then $sd\,{\cal P}$ is a full subcomplex of $sd\,{\cal T}$ and we can apply the above argument in $sd\,{\cal T}$. 

\begin{lem}\label{lem_mapping_cylinder}
Suppose $M$ is an $n$-manifold possibly with boundary, $1 \le r \le \infty, \, r \neq n+1$, 
${\mathcal T}$ is a $C^\infty$ triangulation of $M$, 
$P$ is a finite full subcomplex of ${\cal T}$, 
$Q := \bigcup \{ \sigma \in {\cal T} \mid \sigma \cap P = \emptyset \}$, 
$L$ is a subcomplex of ${\cal T}$ with $P \subset L$, 
$O_0 := M - Q$ and $\widetilde{O}_0 := O_0 \cup (\partial O_0 \times [0,1))$. \\
If $clb^f(P, \widetilde{O}_0) \leq \ell < \infty$ and $P$ is displaceable from $L \cup (\partial M \times [0,1))$ in $\widetilde{O}_0$, then the following holds. 

\benum 
\item[$(1)$] 
{\rm (i)} $cld\,{\rm Diff}^r_c(M, \partial M \cup Q)_0 \leq 2$, \hsh
{\rm (ii)} $clb^f\!d\,{\rm Diff}^r_c(M, \partial M \cup Q)_0 \leq 2\ell +1$. 

\item[$(2)$]  
Any $f \in {\rm Diff}^r_c(M, \partial M \cup Q)_0$ has a factorization $f = gh$ such that \\[2mm] 
\hsp 
\btab[c]{lll}
$g \in {\rm Diff}^r_c(M, \partial M \cup Q)_0^c$ & and & $clb^f(g) \leq 2\ell$ in ${\rm Diff}^r_c(M, \partial M \cup Q)_0$, \\[2mm] 
$h \in {\rm Diff}^r_c(M, \partial M \cup Q \cup L)_0^c$ & and & $clb^f(h) \leq 1$ in ${\rm Diff}^r_c(M, \partial M \cup Q \cup L)_0$. 
\etab 
\eenum 
\end{lem} 

\begin{proof} Let $O := {\rm Int}\,M - Q$. Then, $M_O = \partial M \cup Q$ and 
by Setting~\ref{setting_mapping_cylinder}\,$(\natural)$ the following holds; 
\bit 
\item[$(\ast)_1$] (i) $P$ has the weak absorption property (a) in $O$ rel $L$ and (b) in $O_0$ keeping $P$ invariant.
\eit 
Therefore, the conclusion follows from 
Lemma~\ref{basic_lemma_cl-bdry_final}\,[A]\,$(\ast)_1$ (taking $K := P$) and the assumptions. 
\end{proof} 

We list two important examples of complimentary full subcomplexes used in the subsequent sections. 

\begin{example}\label{exp_(S,U)_I}
Suppose $M$ is an $n$-manifold possibly with boundary and ${\mathcal T}$ is a $C^\infty$ triangulation of $M$. 

(1) The pair $(|{\mathcal T}^{(k)}|, |({\cal T}^\ast)^{(n-k-1)}|)$ 
of the $k$-skeleton of ${\cal T}$ and its dual $(n-k-1)$-skeleton 
is a pair of complimentary full subcomplexes in $sd\,{\cal T}$, that is, 
$|({\cal T}^\ast)^{(n-k-1)}| = \bigcup \{ \tau \in sd\,{\mathcal T} \mid \tau \cap |{\mathcal T}^{(k)}| = \emptyset \}.$ 

(2) Let $(P, Q) := (|{\mathcal T}^{(k)}|, |({\cal T}^\ast)^{(n-k-1)}|)$ or $(|({\cal T}^\ast)^{(k)}|, |{\mathcal T}^{(n-k-1)}|)$ \ $(k=0,1, \cdots, n)$ 
and $O_0 := M - Q$. 
Let $L$ be a subcomplex of $sd\,{\cal T}$ with $P \subset L$. 
When $M$ is compact, $P$ is a finite full subcomplex of $sd\,{\cal T}$ and 
the next assertions follow from Lemma~\ref{lem_reg_nbd_2} and the general position argument, respectively. 
\benum 
\itemI $clb^f(P, \widetilde{O}_0) \leq k+1$. 
Moreover, if $K$ is a compact $C^\infty$ subpolyhedron in $M$ with $K \subset P$ and 
$P$ is weakly absorbed to $K$ in $\widetilde{O}_0$ keeping $K$ invariant, then $clb^f(P, \widetilde{O}_0) \leq clb^f(K, \widetilde{O}_0) \leq \dim K +1$. 
\itemII If $2k < n$ and $\dim L \leq k$, then $P$ is displaceable from $L \cup (\partial M \times [0,1))$ in $\widetilde{O}_0$. 
\eenum 
\end{example}

\begin{example}\label{exp_(S,U)_II} (Even-dimensional case) 
 
Suppose $M$ is a compact $2m$-manifold possibly with boundary, ${\cal T}$ is a $C^\infty$ triangulation of $M$, 
and $({\cal C}, {\cal D})$ is a partition of the set of $m$-simplices of ${\cal T}$. 
Let ${\cal T}[m]$ and ${\cal T}^\ast[m]$ denote the sets of $m$-simplices of ${\cal T}$ and $m$-cells of ${\cal T}^\ast$ respectively. 

(1) The sets ${\cal T}[m]$ and ${\cal T}^\ast[m]$ are in the 1-1 correspondence $\sigma \llra \sigma^\ast$, 
and the a partition $({\cal C}, {\cal D})$ of ${\cal T}[m]$ induces 
the partition ${\cal C}^\ast = \{ \sigma^\ast \mid \sigma \in {\cal C} \}$, ${\cal D}^\ast = \{ \sigma^\ast \mid \sigma \in {\cal D} \}$ of ${\cal T}^\ast[m]$. 
Consider the subcomplexes \\
\hspace{25mm}  ${\cal S} := {\cal T}^{(m-1)} \cup {\cal C}$ \ of \ ${\cal T}$ \ \ and \ \ 
${\cal U} := ({\cal T}^\ast)^{(m-1)} \cup {\cal D}^\ast = ({\cal T}^\ast)^{(m)} - {\cal C}^\ast$ \ of \ ${\cal T}^\ast$. \\
Then, $|{\cal S}|$, $|{\cal U}|$ underlie complimentary full subcomplexes of $sd\,{\cal T}$, that is, 
$|{\cal U}| = \bigcup \{ \sigma \in sd\,{\cal T} \mid \sigma \cap |{\cal S}| = \emptyset \}$. 

(2) Suppose ${\cal C} \subset {\cal C}' \subset {\cal T}[m]$ and ${\cal D} \subset {\cal D}' \subset {\cal T}[m]$.  
Consider the following cases: \\[2mm] 
\hsp \hsh 
\btab[c]{c@{ \ }llll}
{[I]} & $(P, Q) = (|{\cal S}|, |{\cal U}|)$, & $C_0 = |{\cal C}|$, & $L = |{\cal T}^{(m-1)}| \cup |{\cal C}'|$, & $L_0 = |{\cal C}'|$ \\[2mm] 
{[II]} & $(P, Q) = (|{\cal U}|, |{\cal S}|)$, & $C_0 = |{\cal D}^\ast|$, & $L = |({\cal T}^\ast)^{(m-1)}| \cup |{\cal D}'{}^\ast|$, & $L_0 = |{\cal D}'{}^\ast|$  
\etab \\[2mm] 
In each case, $P, Q$ are complimentary finite full subcomplexes of $sd\,{\cal T}$, $L$ is a subcomplex of $sd\,{\cal T}$ and \\
\hsh $P \subset L$, \ 
$P = P^{(m-1)} \cup C_0$, $C_0 = Cl_M(P - P^{(m-1)})$ \ and \ 
$L = L^{(m-1)} \cup L_0$, $L_0 = Cl_M(L - L^{(m-1)})$. \\
Let $O_0 := M - Q$. The next assertion (i) follows from Lemma~\ref{lem_reg_nbd_2}. 
\bit 
\itemi $clb^f(P, \widetilde{O}_0) \leq m+1$. 
Moreover, if $K$ is a compact $C^\infty$ subpolyhedron in $M$ with $K \subset P$ and 
$P$ is weakly absorbed to $K$ in $\widetilde{O}_0$ keeping $K$ invariant, then $clb^f(P, \widetilde{O}_0) \leq clb^f(K, \widetilde{O}_0) \leq \dim K +1$.
\itemii
\bit 
\itema If $C_0$ is displaceable from $L_0 \cup (\partial M \times [0,1))$ in $\widetilde{O}_0$, then $P$ is displaceable from $L \cup (\partial M \times [0,1))$ in $\widetilde{O}_0$. 

\itemb If $C_0$ is weakly absorbed to a $(m-1)$-dimensional compact subpolyhedron in $\widetilde{O}_0$, then 
$C_0$ is displaceable from $L_0 \cup (\partial M \times [0,1))$. 
\eit 
\eit 
The assertion (ii)(a) is easily verified 
by the same argument as in Example~\ref{exp_displacement}\,(1)(ii)(a). 
The assertion (ii)(b) follows from Lemma~\ref{lem_absorption_displacement}\,(2)(ii)(a). 
\end{example} 

\section{Diffeomorphism groups of $(2m+1)$-manifolds}

\subsection{Factorization of isotopies on $(2m+1)$-manifolds} \mbox{} 

\subsubsection{\bf Removing crossing points} \mbox{} 

We recall the trick in \cite{BIP, Tsuboi2} to remove crossing points between some $m$-strata and the tracks of some compact $m$-submanifold under an isotopy on a $(2m+1)$-manifold. Consider the following situation. 

\begin{setting}\label{setting_Whitney trick} \mbox{} 
\bit 
\itemI $M$ is an $n$-manifold possibly with boundary $(n = 2m+1, m \geq 1)$ and 
$N \in {\cal S}{\cal M}_c({\rm Int}\,M)$.  
\itemII $E$ is a compact $C^\infty$ $m$-submanifold of ${\rm Int}\,N$, \\
$D \Subset D' \Subset D'' \Subset E$ are compact $C^\infty$ $m$-submanifolds of $E$. 
\itemiii $\Lambda$ is an $m$-dimensional stratified subset in $M - E$. 
\itemiv $H \in {\rm Isot}^r(M, M_N)_0$ $(1 \leq r \leq \infty)$
 and $W$ is an open neighborhood of $\pi_M H(D'' \times I)$ in ${\rm Int}\,N$. \\
Suppose $\Lambda^{(m-1)} \cap \pi_M H(E \times I) = \emptyset$ and $H = \id$ 
on $\text{[a neighborhood of $(E - {\rm Int}\,D)$]} \times I$. 
\eit 
\end{setting}

\noindent {\bf Removing crossing points.} (\cite[Section 3.3]{BIP}, \cite[Lemma 6.5]{Tsuboi2}) \\[1mm] 
{\bf (Claim.)}  There exists $\overline{H} \in {\rm Isot}^r(M, M_N)_0$, $U \in {\cal B}_f(W - E)$ and 
$A \in {\rm Isot}(M, M_U)_0$ such that \\
\hspp 
\btab[t]{c@{ \ }l} 
(i) & $\overline{H}$ is a $C^r$-approximation of $H$ \ and \ $\overline{H} = H$ on $(M - W) \times I$, \\[2mm] 
(ii) & $U \cap \pi_M \overline{H}((E - {\rm Int}\,D) \times I) =\emptyset$,  \\[2mm] 
(iii) & $\pi_M A\overline{H}(E \times I) \cap \Lambda = \emptyset$ \ and \ $A_1 \in {\rm Diff}(M, M_U)_0^c$. 
\etab 
\vskip 2mm 

\noindent {\bf (Construction.)} 
\benum[(1)] 
\item Take an open neighborhood $W_0$ of $D''$ in $W$ such that $\pi_M H(W_0 \times I) \subset W$.  
Then, we can find a $C^r$-approximation $\overline{H} \in {\rm Isot}^r(M, M_N)_0$ of $H$ such that 
$\overline{H} = H$ on $(M - W_0) \times I$ and 
\bit 
\itemI $\pi_M\overline{H}|_{D' \times I} : D' \times I \lra W$ is a $C^\infty$-immersion outside of a finite subset, 
which has no double points on $(D' - {\rm Int}\,D) \times I$ in $D' \times I$. 
\eit 
If $\overline{H}$ is sufficiently close to $H$, then 
\bit
\itemII (a) $\pi_M \overline{H}((E - {\rm Int}\,D) \times I) \cup (E \times \{ 0,1 \})) \cap \Lambda = \emptyset$, \hsh   
$\pi_M \overline{H}(D \times I) \cap \pi_M \overline{H}((E - {\rm Int}\,D') \times I) = \emptyset$, \\[1mm]  
(b) $\pi_M \overline{H}(E \times I) \cap \Lambda^{(m-1)} = \emptyset$. 
\eit  
Let $W' := W - \pi_M \overline{H}((E - {\rm Int}\,D) \times I)$. 
Note that $\pi_M \overline{H}(D'' \times I) \subset W$ and 
\bit 
\itemiii 
$\pi_M \overline{H}({\rm Int}\,D \times I) \cap \pi_M \overline{H}((E - {\rm Int}\,D) \times I) = \emptyset$ \ \ and \ \  
$\pi_M \overline{H}({\rm Int}\,D \times I) \subset W'$.
\eit 

We also choose the immersion $\pi_M\overline{H}|_{D \times I}$ in generic with respect to $\Lambda$ and the self-intersections, 
so that the arguments below work well. 
It follows that 
\begin{itemize}
\itemiv $\pi_M \overline{H}(E \times I)$ and $\Lambda$ intersect transversely 
at finitely many points $\{ \pi_M \overline{H}(v_i, s_i) \}_{i=1, \cdots, p}$, 
where $\{ v_i \}_{i=1, \cdots, p}$ are distinct points in ${\rm Int}\,D$, $s_i \in (0,1)$ $(i=1, \cdots, p)$ and $\pi_M \overline{H}$ is an embedding on 
$\cup_{i=1}^p(\{ v_i \} \times [s_i, 1])$. 

\itemv $\pi_M\overline{H}|_{D \times I}$ has finitely many double point curves, 
which are in general position with respect to the curves $\pi_M\overline{H}(\{ v \} \times I)$ \ $(v \in D)$ 
\eit 

\vskip 1mm 
\item We construct a disjoint union of closed $n$-disks $U = \cup_{i=1}^p U_i$ in 
$W' - E$ and  
$A \in {\rm Isot}(M, M_U)_0$ such that  
$\pi_M A\overline{H}(D \times I) \cap \Lambda = \emptyset$. 
\vskip 1mm 
\item[] {\bf [1]} the case that $m \geq 2$ : 

In this case, generically  
each arc $\pi_M\overline{H}(\{ v_i \} \times [s_i, 1])$ does not contain any double points and dose not intersect $E$. 
For each $i=1, \cdots, p$ take a thin closed $n$-disk neighborhood $U_i$ of $\pi_M\overline{H}(v_i \times [s_i, 1])$ in $W' - E$ so that 
$\{ U_i \}_{i=1, \cdots, p}$ are disjoint. 
The isotopy $A$ is obtained by pushing the track $\pi_M \overline{H}(D \times I)$ 
along the arcs $\pi_M\overline{H}(v_i \times [s_i, 1])$ $(i=1, \cdots, p)$ in $U$. 
To achieve the required condition $\pi_M A\overline{H}(D \times I) \cap \Lambda = \emptyset$, 
we choose a small parameter $t_0 \in (0,1)$ such that 
$\pi_M \overline{H}(D \times [0,t_0]) \cap U = \emptyset$, and take $A$ so that   
$A_t$ ($t \in [t_0, 1]$) pushes a thin neighborhood of each arc $\pi_M\overline{H}(v_i \times [s_i, 1])$ 
ahead of the point $\pi_M \overline{H}(v_i, s_i)$ in $U_i$.  
\vskip 1mm 
\item[] {\bf [2]} the case that $m = 1$ : 

In this case the double point curves may intersect each other and 
possibly there are at most finitely many triple points and cusps. 
Also the arcs $\pi_M\overline{H}(\{ v_i \} \times [s_i, 1])$ may intersect the double point curves. 
To remove the crossing points $\pi_M \overline{H}(E \times I) \cap \Lambda$, 
for each crossing point $\pi_M\overline{H}(v_i, s_i)$ we construct a tree $T_i \subset \pi_M\overline{H}({\rm Int}\,D \times [0,1])$ rooted at $\pi_M\overline{H}(v_i, s_i)$ which serves as a guide to construct the isotopy~$A$. 

In general, for a curve $\pi_M\overline{H}(\{ v \} \times [s,1])$ $(v \in {\rm Int}\,D)$, 
consider the double points on this curve of the form 
$\pi_M\overline{H}(v, t) = \pi_M\overline{H}(v', t')$ with $v' \in {\rm Int}\,D$, $s < t < t' < 1$.  
To each such double point $\pi_M\overline{H}(v, t)$ consider the curve $\pi_M\overline{H}(\{ v' \} \times [t', 1])$, which is called a branch on the curve $\pi_M\overline{H}(\{ v \} \times [s,1])$. 
Generically, $\pi_M\overline{H}(\{ v \} \times [s,1])$ is an arc including no triple points, no non-transverse double points and no cusps, 
and it has only finitely many branches 
$\pi_M\overline{H}(\{ v_j' \} \times [s_j', 1])$ rooted at a double point 
$\pi_M\overline{H}(v, s_j) = \pi_M\overline{H}(v'_j, s_j')$ with $s < s_j < s_j' < 1$ $(j = 1, \cdots,\ell)$. Generically, these branches are arcs with similar  properties. 

If the arc $\pi_M\overline{H}(\{ v \} \times [s,1])$ is 
a branch on an arc $\pi_M\overline{H}(\{ \underline{v} \} \times [\underline{s},1])$ rooted at 
$\pi_M\overline{H}(\underline{v}, \underline{s}') = \pi_M\overline{H}(v,s)$ $(\underline{s}' < s)$ 
and has no branches on itself, 
then we can find a thin closed  $n$-disk neighborhood $O$ of this arc and $A' \in {\rm Isot}(M, M_O)_0$ 
such that $\pi_M A'\overline{H}(\{ \underline{v} \} \times [\underline{s},1]) = \pi_M\overline{H}(\{ \underline{v} \} \times [\underline{s},1])$
and this arc do not have any branch corresponding to $\pi_M\overline{H}(\{ v \} \times [s,1])$ with respect to $A'\overline{H}$
(so the number of branches decreases by one).  
To define $A'$ we have to respect the times $t'_k, t_k$ which appear in the remaining double points 
$\pi_M\overline{H}(v, t_k) = \pi_M\overline{H}(v'_k, t'_k)$ on $\pi_M\overline{H}(\{ v \} \times [s,1])$ 
with $s < t_k' < t_k < 1$ $(k=1, \cdots, q)$. 
Take $\e > 0$ such that $t_k' < t_k -2\e$ $(k=1, \cdots, q)$. 
Then $A'_\tau$ $(\tau \in [s,1])$ pushes $O'[s, \tau]$ ahead of the point $\pi_M\overline{H}(v, s)$ and 
remains $O'[\tau +\e, 1]$ fixed, where $O'$ is a shrink of $O$ and 
$O'[a,b]$ denotes the part of $O'$ corresponding to $\pi_M\overline{H}(\{ v \} \times [a, b])$.
Since $\overline{H}_\tau(D)$ may appear in $O'[\tau+\e, 1]$, if we push this part across $O'[s]$, 
then it may happen that 
$A_{\tau}'\overline{H}_\tau(D)$ touches $\pi_M\overline{H}(\{ \underline{v} \} \times [\underline{s},1])$ and 
generates a new branch instead of $\pi_M\overline{H}(\{ v \} \times [s,1])$. 

The tree $T_i = \cup_{j=0}^{\ell_i} T_i^{(j)}$ is defined as follows : 
Let $T_i^{(0)} = \pi_M\overline{H}(\{ v_i \} \times [s_i, 1])$ and 
inductively $T_i^{(j)}$ is defined as the union of branches on the arcs in $T_i^{(j-1)}$. 
Generically, this procedure ends in a finite step $\ell_i$ and $T_i$ forms a tree and $\{ T_i \}_{i=1}^p$ are disjoint. 
We call $T = \cup_{i=1}^p T_i$ the intersection graph for $\overline{H}$ for short. Let $T^{(j)} = \cup_{i=1}^p T_i^{(j)}$ $(j=0, 1, \cdots, \ell)$, 
where $\ell = \max_i \ell_i$ and $T_i^{(j)}=\emptyset$ for $j > \ell_i$. 
Generically, $T \cap E = \emptyset$ and we may assume that 
the points $w$'s in ${\rm Int}\,D$ which appear in the following way are all distinct : 
$w = v_i$ for some $i = 1, \cdots, p$ or $\pi_M\overline{H}(v,s) = \pi_M\overline{H}(w,t)$, where 
$\pi_M\overline{H}(v,s)$ is a point of a branch arc in $T$ (including the arcs $\pi_M\overline{H}(\{ v_i \} \times [s_i, 1])$) and 
$\pi_M\overline{H}(w,t)$ is a point on a double points curve (including the case $s > t$). 

For each $i=1, \cdots, p$ take a thin closed $n$-disk neighborhood $U_i$ of $T_i$ in $W' - E$ so that 
$\{ U_i \}_{i=1, \cdots, p}$ are disjoint. 
Let $U = \cup_{i=1}^p U_i$ and for each $j=0,1, \cdots, \ell$ let $U^{(j)}$ denote 
a disjoint union of thin closed $n$-disk neighborhoods of the arcs in $T^{(j)}$ in $U$. 
Then, in the backward order we can inductively construct $A^{(j)} \in {\rm Isot}(M, M_{U^{(j)}})_0$ $(j= 0,1, \cdots, \ell)$  such that 
$A^{(k)} A^{(k+1)} \cdots A^{(\ell)}\overline{H}$ has the intersection graph $\cup_{j=0}^{k-1} T^{(j)}$ for each $k=0,1, \cdots, \ell$. 
Finally, let $A := A^{(0)} A^{(1)} \cdots A^{(\ell)} \in {\rm Isot}(M, M_{U})_0$.
Then $ A\overline{H}$ has the empty intersection graph, which means 
$\pi_M A\overline{H}(D \times I) \cap \Lambda = \emptyset$. 

\item Since $U \subset W'$, it follows that 
$\pi_M A\overline{H}((E - {\rm Int}\,D) \times I) = \pi_M \overline{H}((E - {\rm Int}\,D) \times I)$, 
which does not intersect $\Lambda$ by (1)(ii)(a). 
Therefore, we have $\pi_M A\overline{H}(E \times I) \cap \Lambda = \emptyset$. 
Finally, by \cite[Remark 2.1]{Tsu09}, we can modify $A \in {\rm Isot}(M, M_U)_0$ so that 
$A_1 \in {\rm Diff}(M, M_U)_0^c$. 
This completes the construction. 
\eenum

\subsubsection{\bf Factorization of isotopies on $(2m+1)$-manifolds} \mbox{} 

Next we recall the strategy in \cite{BIP}, \cite{Tsuboi2} to factor isotopies on compact $(2m+1)$-manifolds. 

\begin{setting}\label{setting_cpt_odd} Suppose $M$ is a $(2m+1)$-manifold possibly with boundary $(m \geq 1)$, 
$N \in {\cal S}{\cal M}_c({\rm Int}\,M)$, $F \in {\rm Isot}^r(M, M_N)_0$ and 
$P$, $Q$ are disjoint $m$-dimensional stratified subsets of $M$. 
\end{setting}
\noindent {\bf Review of the arguments in \cite{BIP}, \cite{Tsuboi2}.} \mbox{} \\ 
\noindent (Step 1) \ \ First we remove the intersections between low dimensional parts of $Q$ and $\pi_M F(P \times I)$. 
\benum 
\item By the general position argument we obtain an arbitrarily small isotopy $K \in {\rm Isot}(M, M_N \cup P)_0$ 
with support in an arbitrarily small neighborhood  of $Q$ such that $Q_1 := K_1(Q)$ satisfies the following conditions : 
$$\pi_M F(P \times I) \cap Q_1^{(m-1)} = \emptyset \ \ \ \text{and} \ \ \ 
\pi_M F((P^{(m-1)} \times I)\cup (P \times \{ 0,1 \})) \cap Q_1 = \emptyset.$$

\item 
By (1) and Lemma~\ref{Isotopy_ext} there exists a factorization $F = GH$ 
for some $G \in {\rm Isot}^r(M, M_{N} \cup Q_1)_0$ and $H \in {\rm Isot}^r(M, M_{N} \cup P^{(m-1)})_0$. 
From (1)\, it follows that 
\bit 
\item[(i)\ ] $\pi_M H(P \times I) \cap Q_1^{(m-1)} = \emptyset \ \ \ \text{and} \ \ \ 
\pi_M H((P^{(m-1)} \times I)\cup(P \times \{ 0,1 \})) \cap Q_1 = \emptyset.$
\eit 
\eenum

\noindent (Step 2) \ \ Next we remove the intersection $\pi_M H(P \times I) \cap Q_1$. 

\benum 
\item 
Take an open neighborhood $V$ of $M_N \cup P^{(m-1)}$ in $M$ such that $H = \id$ on $V \times I$. 
Since $O := P - (M_N \cup P^{(m-1)})$ is an $m$-manifold without boundary and 
$P - V$ is a compact subset of $O$, 
there exists a compact $m$-submanifold $E$ of $O$ with $P - V \subset {\rm Int}\,E$. 
Let $D \subset D' \subset D'' \subset E$ be shrinks of $E$ with $P - V \subset {\rm Int}\,D$. 
Note that $E - {\rm Int}\,D \subset P - {\rm Int}\,D \subset P \cap V$. 
Since $\pi_M H(D'' \times I) \subset {\rm Int}\,N - (P - {\rm Int}\,E)$, 
there exists an open neighborhood $W$ of $\pi_M H(D'' \times I)$ in $M$ with $Cl_MW \subset {\rm Int}\,N - (P - {\rm Int}\,E)$. 
 
\item By Removing crossing points in Subsection 4.1.1 (applied to $M$, $N$, $E$, $Q_1$, $H$, $W$), 
there exists \\
$\overline{H} \in {\rm Isot}^r(M, M_N)_0$, \ $U \in {\cal B}_f(W - E)$ \ and \ 
$A \in {\rm Isot}(M, M_U)_0$ \ \ such that 
\bit 
\itemI $\overline{H}$ is a $C^r$-approximation of $H$ \ and \ $\overline{H} = H$ on $(M - W) \times I$, 
\itemII $U \cap \pi_M \overline{H}((E - {\rm Int}\,D) \times I) =\emptyset$,
\itemiii $\pi_M A\overline{H}(E \times I) \cap Q_1 = \emptyset$ \ and \ $A_1 \in {\rm Diff}(M, M_U)_0^c$. 
\eit 
Then, it follows that 
\bit 
\itemiv (a) $H' := A\overline{H} \in  {\rm Isot}^r(M, M_N \cup P^{(m-1)})_0$ \ \ and \ \ (b) $\pi_MH'(P \times I) \cap Q_1 = \emptyset$. 
\eit 

In fact, since $U \subset W - E$, we have $U \subset {\rm Int}\,N - P$ and $A \in {\rm Isot}^r(M, M_N \cup P)_0$. 
From (i)(a) it follows that $\overline{H} \in  {\rm Isot}^r(M, M_N \cup P^{(m-1)})_0$. This implies (iv)(a).    
Since $P - E \subset V \cap (M - W)$, it follows that 
$\overline{H}_t = H_t = \id = A_t$ and $H'_t = \id$ on $P - E$, so that 
$\pi_MH'((P - E) \times I) = P - E$.  
This and (iii) means (iv)(b). 

By (iv)(b) and Lemma~\ref{Isotopy_ext} there exists a factorization 
\bit 
\itemv $H' = G'H''$ for some \ $G' \in {\rm Isot}^r(M, M_N \cup Q_1)_0$ \ and \ $H'' \in {\rm Isot}^r(M, M_N \cup P)_0$.
\eit  
\eenum

\noindent (Step 3) \ \ Factorization of $f$ \\
\hspace*{5mm} We denote the 1-levels of the isotopies $G$, $H$, $\overline{H}$, $A$, $H'$, $G'$ and $H''$ 
by the corresponding letters \break 
\hspace*{5mm} $g$, $h$, $\overline{h}$, $a$, $h'$, $g'$ and $h''$ respectively. 
(Step 1) and (Step 2) lead to the following factorization of $f$.
\benum
\item 
Since $\overline{h}^{-1}h \in {\rm Diff}^r(M, M_N)_0$ is sufficiently close to $\id_M$, for the open cover 
$\{ {\rm Int}\,N - Q_1, {\rm Int}\,N - P \}$ of ${\rm Int}\,N$ we have a factorization \\
\hsp $\overline{h}^{-1}h = \hat{h}\hat{g}$ \ \ for some 
$\hat{g} \in {\rm Diff}^r(M, M_N \cup Q_1)_0$ \ and \ $\hat{h} \in {\rm Diff}^r(M, M_N \cup P)_0$. 

\item Since $\overline{h} = a^{-1}g'h''$, it follows that \\[1mm] 
\hsh $f = gh = g \,\overline{h}\, (\overline{h}^{-1}h) = g \,(a^{-1}g'h'') (\hat{h}\hat{g}) 
= \hat{g}^{-1}\big[ \big(\hat{g}g a^{-1}(\hat{g}g)^{-1}\big) (\hat{g}g g')(h''\hat{h})\big] \hat{g} 
= \hat{g}^{-1}\big[ \tilde{a} \tilde{g} \tilde{h} \big] \hat{g},$ \\[1mm] 
\hsh 
where 
\btab[t]{c@{ \ }l}
(a) & $\tilde{g} := \hat{g}g g' \in {\rm Diff}^r(M, M_N \cup Q_1)_0$, \ \ \ 
(b) \ $\tilde{h}:=h''\hat{h} \in {\rm Diff}^r(M, M_N \cup P)_0$, \\[2mm] 
(c) & $\tilde{a}:= \hat{g}g a^{-1}(\hat{g}g)^{-1} \in {\rm Diff}^r(M, M_{\widetilde{U}})_0^c$ \ 
for \ $\widetilde{U} := \hat{g}g(U) \in {\cal B}_f({\rm Int}\,N)$.
\etab 
\eenum
\vskip 3mm 

\subsection{Compact manifold case} \mbox{} 

\begin{thm}\label{thm_cpt_bdry_odd} Suppose $M$ is a compact $(2m+1)$-manifold possibly with boundary $(m \geq 0)$ and $1 \leq r \leq \infty$, $r \neq 2m+2$.  Then 
$cld\,{\rm Diff}^r(M, \partial)_0 \leq 4$ and $clb^fd\,{\rm Diff}^r(M, \partial)_0 \leq 4m+6$.
\end{thm}

\begin{compl}\label{compl_cpt_bdry_odd} In Theorem~\ref{thm_cpt_bdry_odd},  
if ${\cal T}$ is a $C^\infty$ triangulation of $M$, $(P, Q) \equiv (|{\mathcal T}^{(m)}|, |({\cal T}^\ast)^{(m)}|)$ and 
$\alpha := clb^f\,(P, (M-Q)^\sim)$ and $\beta := clb^f\,(Q, (M-P)^\sim)$, then 
$clb^fd\,{\rm Diff}^r(M, \partial)_0 \leq 2(\alpha + \beta + 1)$. 
\end{compl}

Recall that $O^\sim \equiv \widetilde{O} := O \cup (\partial O \times [0,1)) \in {\cal O}(\widetilde{M})$ for $O \in {\cal O}(M)$.  

In the proof of Theorem~\ref{thm_cpt_bdry_odd} 
we use a triangulation (Example~\ref{exp_(S,U)_I}) instead of a handle decomposition and 
apply Lemma~\ref{lem_mapping_cylinder} to obtain the estimates on $cl$ and $clb^f$. 
This is our refinement for the existence of $\partial M$. 

\begin{proof}[\bf Proof of Theorem~\ref{thm_cpt_bdry_odd} and Compliment~\ref{compl_cpt_bdry_odd}] \mbox{} 
\benum[(1)] 
\item If $m=0$, then $M$ is a finite disjoint union of closed intervals and circles. 
Since \\
\hsp $clb^fd\,{\rm Diff}^r([0,1], \partial)_0 \leq 2$ \ \ and \ \ $clb^fd\,{\rm Diff}^r({\Bbb S}^1)_0 \leq 3$ \ \ (cf.~\cite[Remark 3.1]{Tsu09}), \\
it follows that \hsp $cld\,{\rm Diff}^r(M, \partial)_0 \leq clb^fd\,{\rm Diff}^r(M, \partial)_0 \leq 3$. \\
Below we assume that $m \geq 1$. 

\item 
Take any $C^\infty$ triangulation ${\cal T}$ of $M$ and 
let $(P, Q) = (|{\cal T}^{(m)}|, |{{\cal T}^\ast}^{(m)}|)$ (the $m$-skeleton of ${\cal T}$ and its dual $m$-skeleton). 
Given any $f \in {\rm Diff}^r(M, \partial)_0$. There exists $F \in {\rm Isot}^r(M, \partial)_0$ with $F_1 = f$. \\
For notational simplicity, we set ${\cal D}_A := {\rm Diff}^r(M, \partial M \cup A)_0$ for any subset $A$ of $M$. \\
We apply the argument in Subsection 4.1.2 to the data $\widetilde{M}$, $M$, $F$, $P$, $Q$. 
Here, we can identify as ${\rm Isot}^r(M, \partial)_0 = {\rm Isot}^r(\widetilde{M}, \widetilde{M}_M)_0$ 
and ${\rm Diff}^r(M, \partial)_0 = {\rm DIff}^r(\widetilde{M}, \widetilde{M}_M)_0$ canonically. 
\bit 
\itemI In (Step 1) we obtain an isotopy $K \in {\rm Isot}(M, \partial M \cup P)_0$ and $Q_1 := K_1(Q)$. 
\itemII (Step 3) yields a factorization of $f$ in the following form : \ \ 
$f = \hat{g}^{-1}(agh)\hat{g},$ \ \ 
where 
\bit
\itema $g, \hat{g} \in {\cal D}_{Q_1}$, \hsh 
(b) $h \in {\cal D}_{P}$, \hsh 
(c) $a\in {\rm Diff}^r(M, M_U)_0^c$ \ for some $U \in {\cal B}_f({\rm Int}\,M)$.
\eit 

\itemiii Let $k := K_1 \in {\cal D}_P$. Since $Q_1 = k(Q)$, we obtain $g' := g^{k^{-1}} \equiv k^{-1}g k\in {\cal D}_{Q}$.  
\eit 

\item We apply Lemma~\ref{lem_mapping_cylinder} to the factors of the factorization of $f$ in (2)(ii). \\
From Example~\ref{exp_(S,U)_I} it folllows that 
\bit 
\itemI $(P, Q) = (|{\cal T}^{(m)}|, |({\cal T}^\ast)^{(m)}|)$ is a pair of complimentary finite full subcomplexes of $sd\,{\cal T}$, 
\itemII $P$ is displaceable from $P \cup (\partial M \times [0,1))$ in $(M-Q)^\sim$ \ and \ $\alpha := clb^f\,(P, (M-Q)^\sim) \leq m+1$, 
\itemiii $Q$ is displaceable from $Q \cup (\partial M \times [0,1))$ in $(M-P)^\sim$ \ and \ $\beta := clb^f\,(Q, (M-P)^\sim) \leq m+1$. 
\eit 
Hence, we can apply Lemma~\ref{lem_mapping_cylinder} to $sd\,{\cal T}$, $(P, Q)$, $L = P$  
and $g' \in {\cal D}_Q$ to obtain 
a factorization \\[0.5mm] 
\hspace*{7mm} $g' = g_1g_2$ \ with \ 
$g_1 \in ({\cal D}_Q)^c$, \ $clb^f(g_1) \leq 2\alpha$ \ in ${\cal D}_Q$ \ and \ 
$g_2 \in ({\cal D}_{Q \cup P})^c$, \ $clb^f(g_2) \leq 1$ \ in ${\cal D}_{Q \cup P}$. \\[0.5mm] 
This induces the factorization \ $g = g'^{k} = g_1^{k}g_2^{k}$. \ \ Since $k(P, Q) = (P, Q_1)$, it follows that 
\vskip 1mm 
\bit 
\itemiv $g_1^{k} \in ({\cal D}_{Q_1})^c$, \ $clb^f(g_1^{k}) \leq 2\alpha$ \ in ${\cal D}_{Q_1}$ \ and \ 
$g_2^{k} \in ({\cal D}_{Q_1 \cup P})^c$, \ $clb^f(g_2^{k}) \leq 1$ \ in ${\cal D}_{Q_1 \cup P}$. 
\eit 
\vskip 1mm 
Next, we apply Lemma~\ref{lem_mapping_cylinder} to $sd\,{\cal T}$, $(Q, P)$, $L = Q$ and 
$g_2^k h \in {\cal D}_P$ to obtain a factorization 
\vskip 1mm 
\bit 
\itemv $g_2^k h = h_1h_2$ \ with \ 
$h_1 \in ({\cal D}_P)^c$, \ $clb^f(h_1) \leq 2\beta$ \ in ${\cal D}_P$, \ 
$h_2 \in ({\cal D}_{P \cup Q})^c$, \ $clb^f(h_2) \leq 1$ \ in ${\cal D}_{P \cup Q}$. 
\eit 
\vskip 1mm 
Therefore, in ${\rm Diff}^r(M, \partial M)_0$ it follows that 
\bit 
\item[(vi)] 
$f = \hat{g}^{-1}(agh) \hat{g}
= \hat{g}^{-1}(a g_1^k h_1h_2 ) \hat{g}$, 
\vskip 2mm 
\item[] $cl\,f \leq 4$ \ \ and \ \ 
$\bary[t]{l@{ \ }l}
clb^ff 
& = clb^f(a g_1^k h_1h_2) 
= clb^f\,a +  clb^f\,g_1^k +  clb^f\, h_1 +  clb^f\, h_2 \\[2.5mm] 
& \leq 1 + 2\alpha + 2\beta + 1 
\leq 4m+6. 
\eary$
\eit 
\eenum
\vskip -7mm 
\end{proof}

\begin{remark}\label{rmk_incorporation}
In the last paragraph of Proof of Theorem~\ref{thm_cpt_bdry_odd} we decomposed the composition $\tilde{g}_2\tilde{h}$ instead of $\tilde{h}$. 
(The same argument has already appeared in the proof of  Proposition~\ref{prop_2m_cpt_no-m-h}.)
We call this device an incorporation of a factor to the next factor. 
\end{remark}

For a $(2m+1)$-manifold with a handle decomposition we have 
the following relative estimate on $clb^f$.  

\begin{prop}\label{cl_odd_cpt_bd_handle} 
Suppose $M$ is a $(2m+1)$-manifold without boundary $(m \geq 0)$, $1 \leq r \leq \infty$, $r \neq 2m+2$ and 
${\cal H}$ is a handle decomposition of $M$. 
\benum
\item[{\rm [I]}\,] $clb^fd\,({\rm Diff}^r(M, M_N)_0, {\rm Diff}^r(M, M_{N_1})_0) \leq 2c({\cal H}|_{N_1})+2 \leq 2c({\cal H})+2$ \\
\hspace*{10mm} for any $N \in {\cal S}{\cal M}_c(M, {\cal H})$ and $N_1 \in {\cal S}{\cal M}_c(M, {\cal H}^\ast)$ with $N \subset N_1$. 

\item[{\rm [II]}] If $M$ is closed, then $clb^fd\,{\rm Diff}^r(M)_0 \leq 2c({\cal H})+2$.  
\eenum 
\end{prop}

The estimate in [II] is compared with that in \cite{Tsu09}, that is, \\
\hspp $clb^f\,{\rm Diff}^r(M)_0 \leq 4c({\cal H}) + 3$ \ (\cite[Proof of Theorem 1.5]{Tsu09} (p.54)). 

\begin{proof} \mbox{} [I] The proof is similar to that of Theorem~\ref{thm_cpt_bdry_odd}. 
We apply Lemma~\ref{factorization_hdle_cpt} instead of Lemma~\ref{lem_mapping_cylinder}. 
\benum[(1)] 
\item If $m=0$, then 
as in Proof of Theorem~\ref{thm_cpt_bdry_odd}, 
$N$ is a finite disjoint union of closed intervals and circles and 
$clb^fd\,{\rm Diff}^r(M, M_N)_0 \leq 3$, while $c({\cal H}) = 2$.
Below we assume that $m \geq 1$. 

\item Let $(P, Q) = (P_{\cal H}^{(m)}, P_{\cal H^\ast}^{(m)})$. For notational simplicity, we set \\  
\hspp ${\cal D}_A := {\rm Diff}^r(M, M_N \cup A)_0$ \ and \ ${\cal D}_{1,A} := {\rm Diff}^r(M, M_{N_1} \cup A)_0$ \ for \ $A \subset M$. \\
Given any $f \in {\rm Diff}^r(M, M_N)_0$. There exists $F \in {\rm Isot}^r(M, M_N)_0$ with $F_1 = f$. \\
We apply the argument in Subsection 4.1.2 to the data $M$, $N$, $F$, $P$, $Q$. 
\bit 
\itemI In (Step 1) we obtain an isotopy $K \in {\rm Isot}(M, M_N \cup P)_0$ and $Q_1 := K_1(Q)$. 
\itemII (Step 3) yields a factorization of $f$ in the following form : \ \ 
$f = \hat{g}^{-1}(agh)\hat{g},$ \ \ 
where 
\bit
\itema $g,  \hat{g}\in {\cal D}_{Q_1}$, \hsh 
(b) $h \in {\cal D}_P$, \hsh 
(c) $a\in {\rm Diff}^r(M, M_U)_0^c$ \ 
for some $U \in {\cal B}_f({\rm Int}\,N)$.
\eit 

\itemiii Let $k := K_1 \in {\cal D}_P$. Since $k(M_N \cup Q) = M_N \cup Q_1$, we obtain $g' := g^{k^{-1}} \equiv k^{-1}g k \in {\cal D}_Q$.   
\eit 

\item We apply Lemma~\ref{factorization_hdle_cpt} to the factors of the factorization of $f$ in (2)(ii). We put 
\bit 
\itemI $\alpha := c({\cal H}|_N^{(m)}) = c\big((P_{{\cal H}|_N})^{(m)}\big) \leq m+1$ \ and \ 
$\beta := c({\cal H}^\ast|_{N_1}^{(m)}) = c\big((P_{{\cal H}^\ast|_{N_1}})^{(m)}\big) \leq m+1$. 
\eit 
\vskip 1mm 
Since ${\cal H}|_{N}^{(m)} \subset {\cal H}|_{N_1}^{(m)}$ and 
${\cal H}^\ast|_{N_1}^{(m)} = \{ h \in {\cal H}|_{N_1} \mid \text{[the index of $h$ in ${\cal H}] \geq m + 1$} \}$,   
 it follows that 
\vskip 1mm 
\bit 
\itemII $\alpha + \beta \leq c({\cal H}|_{N_1})$. 
\eit 
First we apply Lemma~\ref{factorization_hdle_cpt} to $M$, $N$, ${\cal H}$, $(P, Q)$ and $g' \in {\cal D}_Q$ to obtain 
a factorization \\[0.5mm] 
\hspace*{7mm} $g' = g_1g_2$ \ with \ 
$g_1 \in ({\cal D}_Q)^c$, \ $clb^f(g_1) \leq 2\alpha$ \ in ${\cal D}_Q$ \ and \ 
$g_2 \in ({\cal D}_{P \cup Q})^c$, \ $clb^f(g_2) \leq 1$ \ in ${\cal D}_{P \cup Q}$. \\[0.5mm] 
This induces the factorization \ $g = g'^{k} = g_1^{k}g_2^{k}$. \ \ Since $k(P, Q) = (P, Q_1)$, it follows that 
\vskip 1mm 
\bit 
\itemiii $g_1^{k} \in ({\cal D}_{Q_1})^c$, \ $clb^f(g_1^{k}) \leq 2\alpha$ \ in ${\cal D}_{Q_1}$ \ and \ 
$g_2^{k} \in ({\cal D}_{Q_1 \cup P})^c$, \ $clb^f(g_2^{k}) \leq 1$ \ in ${\cal D}_{Q_1 \cup P}$. 
\eit 
\vskip 1mm 
Next we apply Lemma~\ref{factorization_hdle_cpt} to $M$, $N_1$, ${\cal H}^\ast$, $(Q, P)$ and 
$g_2^k h \in {\cal D}_{1, P}$ to obtain a factorization \ $g_2^k h = h_1h_2$, 
\bit 
\itemiv 
$h_1 \in ({\cal D}_{1, P})^c$, \ $clb^f(h_1) \leq 2\beta$ \ in ${\cal D}_{1, P}$ \ and \ 
$h_2 \in ({\cal D}_{1, P \cup Q})^c$, \ $clb^f(h_2) \leq 1$ \ in ${\cal D}_{1, P \cup Q}$. 
\eit 
\vskip 1mm 
Therefore, in ${\rm Diff}^r(M, M_{N_1})_0$ it follows that 
\bit 
\itemv 
$f = \hat{g}^{-1}(agh) \hat{g}
= \hat{g}^{-1}(a g_1^k h_1h_2 ) \hat{g}$, 
\vskip 2mm 
\item[] $cl\,f \leq 4$ \ \ and \ \ 
$\bary[t]{l@{ \ }l}
clb^ff 
& = clb^f(a g_1^k h_1h_2) 
= clb^f\,a +  clb^f\,g_1^k +  clb^f\, h_1 +  clb^f\, h_2 \\[2.5mm] 
& \leq 1 + 2\alpha + 2\beta + 1 \leq 2c({\cal H}|_{N_1}) + 2.
\eary$
\eit 
\vskip 2mm 
Finally, the assertion [II] follows from [I] by taking $N = N_1 = M$. 
\eenum
\vskip -4mm 
\end{proof}

\subsection{Open manifold case} 

\begin{thm}\label{thm_open_odd} Suppose $M$ is an open $(2m+1)$-manifold $(m \geq 0)$, $1 \leq r \leq \infty$, $r \neq 2m+2$ and 
${\cal H}$ is any handle decomposition of $M$. Then, 
\benum
\item[{\rm (1)}] {\rm (i)} $cld\,{\rm Diff}^r(M)_0 \leq 8$ \ \ and \ \ 
{\rm (ii)} $clb^dd\,{\rm Diff}^r(M)_0 \leq 4c({\cal H})+4 \leq 8m +12$, 
\item[{\rm (2)}] {\rm (i)} $cld\,{\rm Diff}_c^r(M)_0 \leq 4$ \ \ and \ \ 
{\rm (ii)} $clb^fd\,{\rm Diff}_c^r(M)_0 \leq 2c({\cal H})+2 \leq 4m +6$. 
\eenum 
\end{thm}

\begin{proof} 
(1) Take any $f \in {\rm Diff}^r(M)_0$. 
There exists $F \in {\rm Isot}^r(M)_0$ with $F_1 = f$. 
\benum 
\itemi 
By Lemma~\ref{lem_factorization_open} 
there exists an exhausting sequence $\{ M_k \}_{k\geq 1}$ in $M$ such that 
\bit 
\itema $M_k \in {\mathcal S}{\mathcal M}_c(M)$ $(k \geq 1)$ and 
$F(M_{4k, 4k+1} \times I) \Subset M_{4k-1, 4k+2} \times I$ $(k \geq 0)$.
\eit  
Set $M' := \bigcup_{k \geq 0} M_{4k+2, 4k+3}$ and $M'' := \bigcup_{k \geq 0} M_{4k, 4k+1}$. 
Then there exists a factorization $F = GH$ for some $G \in {\rm Isot}^r(M, M')_0$ and $H \in {\rm Isot}^r(M, M'')_0$. 
This induces the factorization $f = gh$, where $g = G_1 \in {\rm Diff}^r(M, M')_0$ and $h = H_1 \in {\rm Diff}^r(M, M'')_0$. 

For each $k \geq 0$ we have $G|_{M_{4k-1,4k+2} \times I} \in {\rm Isot}^r(M_{4k-1,4k+2}, \partial)_0$ 
and $g|_{M_{4k-1,4k+2}} \in {\rm Diff}^r(M_{4k-1,4k+2}, \partial)_0$. 
By Theorem~\ref{thm_cpt_bdry_odd} 
it follows that $cl\,g|_{M_{4k-1,4k+2}} \leq 4$ in ${\rm Diff}^r(M_{4k-1,4k+2}, \partial)_0$. 
This implies that $cl\,g \leq 4$ in ${\rm Diff}^r(M, M')_0$. 
Similarly $cl\,h \leq 4$ in ${\rm Diff}^r(M, M'')_0$. Therefore $cl\,f \leq 8$ in ${\rm Diff}^r(M)_0$. 

\itemii We trace the argument in (i) with necessary modifications. 
For notational simplicity we set $\alpha := 2c({\cal H})+2$ and ${\cal D}_A := {\rm Diff}^r(M, M_A)_0$ for $A \subset M$. 
For $N \in {\cal S}{\cal M}_c(M)$ and $\eta \in {\rm Diff}^r(N, \partial)_0$ let $\widetilde{\eta} \in {\rm Diff}^r(M, M_N)_0$ denote the 
canonical extension of $\eta$ by $\id$. 

By Lemma~\ref{lem_factorization_open-2} we may assume that 
the exhausting sequence $\{ M_k \}_{k\geq 1}$ satisfies the following additional conditions with respect to the handle decomposition ${\cal H}$: 
\benum 
\itemb There exist $N_k'  \in {\cal S}{\cal M}_c(M; {\cal H})$ and $N_k'' \in {\cal S}{\cal M}_c(M; {\cal H}^\ast)$ $(k \geq 0)$ 
such that for each $k \geq 0$ 
\bit 
\item[] $M_{4k-1, 4k+2} \Subset N_k' \Subset N_k'' \ \Subset \ M_{4k-2, 4k+3}$  
\ \ and \ \ $N_k'' \cap N_{k+1}'' = \emptyset$. 
\eit 

\itemc There exist $L_k'  \in {\cal S}{\cal M}_c(M; {\cal H})$ and $L_k'' \in {\cal S}{\cal M}_c(M; {\cal H}^\ast)$ $(k \geq 1)$ such that for each $k \geq 1$ 
\bit 
\item[] $M_{4k-3, 4k} \Subset L_k' \Subset L_k'' \ \Subset \ M_{4k-4, 4k+1}$  
\ \ and \ \ $L_k'' \cap L_{k+1}'' = \emptyset$. 
\eit 
\eenum 
By Proposition~\ref{cl_odd_cpt_bd_handle} we have \ 
$clb^fd\,({\cal D}_{N_k'}, {\cal D}_{N_k''}) \leq \alpha$ $(k \geq 0)$ \ and \ $clb^fd\,({\cal D}_{L_k'}, {\cal D}_{L_k''}) \leq \alpha$ $(k \geq 1)$. 
Then, it follows that \ $clb^f (g|_{M_{4k-1,4k+2}})^\sim \leq \alpha$ \ in ${\cal D}_{N_k''}$ \ and \ 
$clb^f (h|_{M_{4k-3,4k}})^\sim \leq \alpha$ \ in ${\cal D}_{L_k''}$. \\
This implies that $clb^d g \leq \alpha$ and $clb^d h \leq \alpha$ in ${\rm Diff}^r(M)_0$. 
Therefore $clb^d f \leq 2\alpha$ in ${\rm Diff}^r(M)_0$. 
\eenum

(2) The assertions follow directly from Theorem~\ref{thm_cpt_bdry_odd} and Proposition~\ref{cl_odd_cpt_bd_handle}. 
\end{proof} 

\section{Diffeomorphism groups of $2m$-manifolds} 

\subsection{Factorization of isotopies on $2m$-manifolds} \mbox{} 

\subsubsection{\bf Whitney trick} \mbox{} 

First we recall the Whitney trick used in \cite{Tsuboi3} to remove 
inessential intersections between some $m$-strata and the tracks of some $m$-disks under an isotopy on a $2m$-manifold. 
Below we assume that $m \geq 3$. 

\begin{setting}\label{setting_Whitney trick}
Consider the following situation : 
\bit 
\itemI $M$ is a $2m$-manifold possibly with boundary $(m \geq 3)$ and 
$N \in {\cal S}{\cal M}_c({\rm Int}\,M)$. 
\itemII $E = \bigcup_{i=1}^s E_i$ is a disjoint union of closed $m$-disks in ${\rm Int}\,N$, \\
$D = \bigcup_{i=1}^s D_{i}$, $D' = \bigcup_{i=1}^s D_{i}'$, $D'' = \bigcup_{i=1}^s D''_{i}$, where 
$D_i \Subset D'_{i}\Subset D''_{i} \Subset E_i$ are closed $m$-subdisks of $E_i$ $(i=1, \cdots, s)$. 
\hfill (The case for $s = 1$ is described in \cite{Tsuboi3}.) 
\itemiii $\Lambda$ is an $m$-dimensional stratified subset in $M - E$. 
\itemiv $\widetilde{H} \in {\rm Isot}(M, M_N)_0$ 
 and $W$ is an open neighborhood of $\pi_M \widetilde{H}(D'' \times I)$ in ${\rm Int}\,N$. \\
Suppose $\Lambda^{(m-1)} \cap \pi_M\widetilde{H}(E \times I) = \emptyset$ and $\widetilde{H} = \id$ 
on $\text{[a neighborhood of $(E - {\rm Int}\,D)$]} \times I$. 
\eit 
\end{setting}
\noindent {\bf Basic Whitney trick.} (\cite[p161 - 165, Lemma 3.6]{Tsuboi3})
\begin{enumerate}[(1)]  
\item Take a $C^\infty$ approximation $\widetilde{H}' \in {\rm Isot}(M, M_N)_0$ of $\widetilde{H}$ such that 
(a) $\widetilde{H}' = \widetilde{H}$ on $(M - W) \times I$,
\bit 
\itemb $\pi_M\widetilde{H}'|_{D' \times I}$ is a $C^\infty$ immersion outside a 1-dimensional subset and it is generic with respect to $D$, $\Lambda$ and self-intersections, and $\pi_M\widetilde{H}'((D' - {\rm Int}\,D) \times I)$ has no double point, 
\itemc $\pi_M\widetilde{H}'((E - {\rm Int}\,D) \times I) \cap \pi_M\widetilde{H}'(({\rm Int}\,D) \times I) = \emptyset$, 
$\pi_M\widetilde{H}'((E - {\rm Int}\,D) \times I) \cap \Lambda = \emptyset$ and \\
$\widetilde{H}' = \id$ on $\text{[a neighborhood of $(E - {\rm Int}\,D'')$]} \times~I$. 
\eit 

\item In this situation the intersection $\Gamma := \Lambda \cap \pi_M\widetilde{H}'(D \times I)$ is a compact 1-manifold 
as a trace of finite points $\Lambda \cap \widetilde{H}_t'(D)$ $(t \in I)$. 
Since $\Lambda \cap D = \emptyset$, the algebraic intersection number 
between $\Lambda$ and each level $\widetilde{H}_t'(D)$ is zero 
and there are finitely many generations and cancellations of intersection points. 
\item A pre-Whitney disk is constructed by tracing a pair of intersection points starting from each generation point of intersection 
and connecting them in each level by the image of a geodesic segment in $D$, where each $D_i$ is regarded as a round $m$-disk with a flat metric. 
If two such pairs of intersection points meet at a cancellation point of intersection in a level $\widetilde{H}_t'(D)$, 
the remaining two intersection points from those pairs are coupled and their traces are connected by a geodesic segment in each level. 
The gap among three geodesic segments in the level $\widetilde{H}_t'(D)$ is fullfilled by a geodesic triangle in this level. 
\bit 
\itemI If the trace of a pair of intersection points reaches the 1-level $\widetilde{H}_1'(D)$, we obtain an arc component of $\Gamma$ and a disk bounded by this arc and an arc in $\widetilde{H}_1'(D)$. 
The Whitney disk for this arc component is obtained by smoothing this disk near each geodesic triangle. 
\itemII If the trace of a pair of intersection points self-intersects at a cancellation point $\widetilde{H}_t'(v)$ in a level $\widetilde{H}_t'(D)$, then we obtain a circle component of $\Gamma$ and a disk bounded by this circle. 
The Whitney disk associated to this circle component is obtained by smoothing this disk near each geodesic triangle and 
add a thin band along the arc $\pi_M\widetilde{H}'(\{ v \} \times [t,1])$ in $\pi_M\widetilde{H}'(D \times I)$. 
\eit 
Since $\widetilde{H}'$ is generic, it follows that 
\bit 
\itemiii 
these Whitney disks are pairwise disjoint and embedded in $\pi_M\widetilde{H}(D \times I) - D$,  
and in $\pi_M\widetilde{H}(D \times I)$
they have (a) no double points if $m \geq 4$ and 
(b) finitely many double points if $m = 3$. 
In the case (b), at each double point we take the branch arc of the form 
$\pi_M\widetilde{H}'(\{ v_i \} \times [t_i, 1])$, which has no double points except $\pi_M\widetilde{H}'(v_i , t_i)$. 
\eit 
 
\item Let $\widetilde{U}$ be the disjoint union of small $2m$-disk neighborhoods of these Whitney disks 
(together with the branch arcs in the case $m=3$) in $W - \big(E \cup \pi_M\widetilde{H}'((E - {\rm Int}\,D) \times I)\big)$. 
Then there exists $\widetilde{A} \in {\rm Isot}(M, M_{\widetilde{U}})_0$ such that   
$\pi_M \widetilde{A}\widetilde{H}'(D \times I) \cap \Lambda = \emptyset$ and $\widetilde{A}_1 \in {\rm Diff}(M, M_{\widetilde{U}})_0^c$. 
Note that $\pi_M \widetilde{A}\widetilde{H}'(E \times I) \cap \Lambda = \emptyset$.
\eenum 

\subsubsection{\bf Factorization of isotopies on $2m$-manifolds} \mbox{} 

\begin{setting}\label{setting_cpt_even} Suppose $M$ is an $2m$-manifold possibly with boundary $(m \geq 1)$, 
$N$ is a compact $2m$-submanifold of ${\rm Int}\,M$ and $F \in {\rm Isot}^r(M, M_N)_0$. 
Suppose $P$ and $Q$ are $m$-dimensional stratified subsets of $M$ such that 
$P^{(m-1)} \cap Q = P \cap Q^{(m-1)} = \emptyset$. 
\end{setting}
\noindent {\bf Review of the arguments in \cite{Tsuboi3} and some refinements in our setting.} \mbox{} \\ 
\noindent (Step 1) \ \ First we remove the intersections between low dimensional skeletons of $Q$ and $\pi_M F(P \times I)$. 
\begin{enumerate}
\item There exists an arbitrarily small isotopy $K \in {\rm Isot}(M, M_N \cup P^{(m-1)})_0$ with support in an arbitrarily small neighborhood of $Q$ such that $Q_1 := K_1(Q)$ satisfies the following conditions : 

\begin{itemize}
\item[(0)\,] $Q_1^{(m-1)} \cap P = \emptyset$ \hspace{10mm} (i) \ $Q_1 \cap \pi_M F((P^{(m-2)} \times I) \cup (P^{(m-1)} \times \{ 0, 1\})) = \emptyset$
\item[(ii)\,] $Q_1^{(m-1)} \cap \pi_M F(P^{(m-1)} \times I) = \emptyset$ 
\hspace{10mm} (iii) \ $Q_1^{(m-2)} \cap \pi_M F(P \times I) = \emptyset$. 
\end{itemize}
\item 
By (1)(i) and Lemma~\ref{Isotopy_ext} there exists a factorization $F = GH$ 
for some $G \in {\rm Isot}^r(M, M_{N} \cup Q_1)_0$ and $H \in {\rm Isot}^r(M, M_{N} \cup P^{(m-2)})_0$. 
From (1)\,(ii), (iii) it follows that 
\bit 
\item[(i)\ ] $Q_1^{(m-1)} \cap \pi_M H(P^{(m-1)} \times I) = \emptyset$ 
\hspace{10mm} (ii) \ $Q_1^{(m-2)} \cap \pi_M H(P \times I) = \emptyset$. 
\eit 
\end{enumerate} 

\noindent $\ast$ Below assume that $m \geq 2$. 
\vskip 2mm 

\noindent (Step 2) \ \ Next we remove the intersection $Q_1 \cap \pi_M H(P^{(m-1)} \times I)$.  
(cf.~Proof of Theorem~\ref{thm_cpt_bdry_odd}\,Case [I]) 
\begin{enumerate} 
\item Choose a $C^\infty$ approximation $\overline{H} \in {\rm Isot}(M, M_N \cup P^{(m-2)})_0$ of $H$ which is generic with respect to $P$ and $Q_1$ (so that (2) holds). If $\overline{H}$ is sufficiently close to $H$, then 
\bit 
\item[(i)\ ] $Q_1^{(m-1)} \cap \pi_M \overline{H}(P^{(m-1)} \times I) = \emptyset$ 
\hspace{10mm} (ii) \ $Q_1^{(m-2)} \cap \pi_M \overline{H}(P \times I) = \emptyset$. 
\eit 

\item There exists $U \in {\cal B}_f({\rm Int}\,N - (P^{(m-1)} \cup Q_1^{(m-1)}))$  
and $A \in {\rm Isot}(M, M_U)_0$ \ such that \\
$Q_1 \cap \pi_M A \overline{H}(P^{(m-1)} \times I) = \emptyset$ and $A_1 \in {\rm Diff}(M, M_U)_0^c$. 
Let $H' : = A \overline{H} \in {\rm Isot}(M, M_N \cup P^{(m-2)})_0$. 

\item There exists a factorization $H' = G'H''$ for some \\
\hspace*{30mm} $G' \in {\rm Isot}(M, M_{N} \cup Q_1)_0$ \ and \ $H'' \in {\rm Isot}(M, M_{N} \cup P^{(m-1)})_0$. \\
We have \ $Q_1^{(m-2)} \cap \pi_M H''(P \times I) = \emptyset$ \ since \ $Q_1^{(m-2)} \cap \pi_M H'(P \times I) = \emptyset$. 
\end{enumerate} 

\vskip 2mm 
\noindent $\ast$ From here we assume that $m \geq 3$.  
\vskip 2mm 
\noindent (Step 3) \ \ We remove the intersection $Q_1^{(m-1)} \cap \pi_M H''(P \times I)$. (cf.~Proof of Theorem~\ref{thm_cpt_bdry_odd}\,Case [I])
\begin{enumerate} 
\item Take a $C^\infty$ approximation $\overline{H}'' \in {\rm Isot}(M, M_N \cup P^{(m-1)})_0$ of $H''$ which is generic with respect to $P$ and $Q_1$. We have $Q_1^{(m-2)} \cap \pi_M\overline{H}''(P \times I) = \emptyset$. 

\item There exists $U'' \in {\cal B}_f({\rm Int}\,N - (P \cup Q_1^{(m-2)}))$ and  
$A'' \in {\rm Isot}(M, M_{U''})_0$ such that \\
$Q_1^{(m-1)} \cap \pi_M A''\overline{H}''(P \times I) = \emptyset$ and $A''_1 \in {\rm Diff}(M, M_{U''})_0^c$. 
Let $H''' := A''\overline{H}'' \in {\rm Isot}(M, M_N \cup P^{(m-1)})_0$.
\end{enumerate} 
\hspace*{5mm} In \cite{Tsuboi3} Step 3 is incorporated to the inductive argument in Step 4\,(2) below 
and repeated in each inductive step. 

Until now we have obtained $H^{(0)} := H''' \in {\rm Isot}(M, M_N \cup P^{(m-1)})_0$ with  
$Q_1^{(m-1)} \cap \pi_M H^{(0)}(P \times I) = \emptyset$. 

\begin{setting_5.2+}
Let ${\cal S}$ denote the set of all $m$-strata of $P$ and we put \\
\hspace*{10mm} ${\cal S}_0 := \{ \sigma \in {\cal S} \mid H^{(0)} \equiv \id 
\text{ on $[$a neighborhood of $\sigma] \times I$}\}$ \hspace{3mm} and \\
\hspace*{10mm} ${\cal S}_c := \{ \sigma \in {\cal S} - {\cal S}_0  \mid \sigma \text{ is an open $m$-disk}, \ \sigma \subset {\rm Int}\,N\}$. \\[0.5mm] 
Suppose ${\cal E}$ is a finite subset of ${\cal S}_c$ and ${\cal E} = \bigcup_{j=1}^k {\cal E}_j$ is a partition of ${\cal E}$ $(k \geq 1)$. Let \\
\hspace*{10mm} $\Sigma_0 := P^{(m-1)} \cup |{\cal S}_0|$ \ \ and \ \ 
$\Sigma_j := P^{(m-1)} \cup |{\cal S}_0| \cup (\bigcup_{i=1}^j |{\cal E}_i|)$ \ $(j=1, \cdots, k)$. \\ 
Let ${\cal U}$ denote the set of all $m$-strata of $Q_1$ and let \\
\hspace*{10mm}   
$\Lambda_j := Q_1^{(m-1)} \cup \bigcup \{ \tau \in {\cal U} \mid \tau \cap |{\cal E}_j| = \emptyset \}$ \ $(j=1, \cdots, k)$. \\
Note that 
\bit 
\itemI $\Sigma_j = \Sigma_{j-1} \cup |{\cal E}_j|$ (a disjoint union), 
\itemII $\Lambda_j^{(m-1)} = Q_1^{(m-1)}$, 
$\Lambda_j \cap |{\cal E}_j| = \emptyset$ \ and \  
$\Sigma_j \cap (\Sigma_{j-1} \cup \Lambda_j) = \Sigma_{j-1}$. 
\itemiii $H^{(0)} \in {\rm Isot}(M, M_{N} \cup \Sigma_0)_0$ \  
and \ $\pi_M H^{(0)}(P \times I) \cap Q_1^{(m-1)} = \emptyset$.  
\eit 
\end{setting_5.2+} 

\noindent (Step 4) \ This is the main step in which the Whitney trick is used to remove 
inessential intersections 
\hspace*{5mm} between $Q_1$ and $\pi_M H^{(0)}(P \times I)$ over $m$-strata.  

Starting from $H^{(0)}$, inductively we construct isotopies  $H'{}^{(j)}$, $A^{(j)}$, $H''{}^{(j)}$, $G^{(j)}$, $H^{(j)}$ $(j = 1, \cdots, k)$ as follows. \hsh 
Suppose $H^{(j-1)} \in {\rm Isot}(M, M_{N} \cup \Sigma_{j-1})_0$ is obtained so that  
 $\pi_M H^{(j-1)}(P \times I) \cap Q_1^{(m-1)} =~\emptyset$.  
\benum[(1)]  
\item[{[1]}] We apply Basic Whitney trick to the isotopy $H^{(j-1)}$, the $m$-dimensional stratified subset $\Lambda_j$,  
 disjoint unions of closed $m$-disks $D \Subset D' \Subset D'' \Subset E \subset |{\cal E}_j|$ and 
 an open neighborhood $W$ of $\pi_M H^{(j-1)}(D'' \times I)$ in ${\rm Int}\,N$. 
 Here, $D, D', D'', E$ and $W$ are chosen as follows : 
 \bit 
 \itemI 
 There exists an open neighborhood $V$ of $M_N \cup \Sigma_{j-1}$ in $M$ such that 
 $H^{(j-1)}|_{V \times I} = \id$. 
 Let ${\cal E}_j = \{ \sigma_i \}_{i=1}^s$. (If ${\cal E}_j = \emptyset$, then the argument below is formally skipped.) 
 Then each $\sigma_i$ is an open $m$-disk in ${\rm Int}\,N$  and 
 $\sigma_i - V$ is a compact subset of $\sigma_i$. Hence 
 we can take closed $m$-disks $D_i \Subset D'_{i} \Subset D''_{i} \Subset E_i$ in $\sigma_i$ with $\sigma_i - V \Subset D_i$. 
 Let $D := \bigcup_{i=1}^s D_{i}$, $D' := \bigcup_{i=1}^s D'_{i}$, $D'' := \bigcup_{i=1}^s D''_{i}$ and $E := \bigcup_{i=1}^s E_i$. 
These satisfy the condition (vi) in Setting~\ref{setting_Whitney trick} since 
$\pi_M H^{(j-1)}(P \times I) \cap Q_1^{(m-1)} = \emptyset$ and $E - {\rm Int}\,D \subset V$. 
\itemII From the assumption and (i) it follows that \\
\hspace*{15mm} $\pi_M H^{(j-1)}(D'' \times I) \subset O := {\rm Int}\,N - 
 (\Sigma_{j-1}\cup (|{\cal E}_j| - {\rm Int}\,E) \cup Q_1^{(m-1)}) \in {\cal O}(M)$. \\
 We take the neighborhood $W$ so that $Cl_M W \subset O$. 
\eit 
\item[{[2]}] Basic Whitney trick yields the following data : 
 
\bit 
\itema $H'{}^{(j)} \in {\rm Isot}(M, M_{N})_0$ : a $C^\infty$ approximation of $H^{(j-1)}$ with 
$H'{}^{(j)} = H^{(j-1)}$ on $(M - W) \times I$, 
\itemb $U_j \in {\cal B}_f(W - E)$,  
\itemc $A^{(j)} \in {\rm Isot}(M, M_{U_j})_0$ such that $\pi_M A^{(j)}H'^{(j)}(E \times I) \cap \Lambda_j = \emptyset$ and 
$A^{(j)}_1 \in {\rm Diff}(M, M_{U_j})_0^c$. 
\eit 
The isotopies $H'{}^{(j)}$, $A^{(j)}$ and $H''^{(j)} := A^{(j)}H'^{(j)}$ satisfy the following conditions : 
\bit 
\itemI  $M_{U_j} \supset M - W \supset M - Cl_M W 
\supset M_N \cup \Sigma_{j-1}\cup (|{\cal E}_j| - {\rm Int}\,E) \cup Q_1^{(m-1)}$ \\
\hspace*{10mm} since $U_j \subset W \subset Cl_M W \subset O$. \\
$V \supset |{\cal E}_j| - {\rm Int}\,E$ \ since ${\rm Int}\,E_i \supset D_i \supset \sigma_i - V$ $(i = 1, \cdots, s)$. \\
Hence, $V - Cl_M W \supset M_N \cup \Sigma_{j-1}\cup (|{\cal E}_j| - {\rm Int}\,E)$. 
\itemII On $(V - W) \times I$, we see that $H'^{(j)} = H^{(j-1)} = \id$, 
$A^{(j)} = \id$ by (i), and so $H''^{(j)} = \id$. \\
Hence $H'^{(j)}, A^{(j)}, H''^{(j)} \in {\rm Isot}(M, M_{N} \cup \Sigma_{j-1})_0$ by (i). 
\itemiii $\pi_M H''^{(j)}(|{\cal E}_j| \times I) \cap \Lambda_j = \emptyset$.  
This follows from (c) and the next observation : \\
$\pi_M H''^{(j)}((|{\cal E}_j| - E) \times I) = |{\cal E}_j| - E \subset M - \Lambda_j$ \\
\hspace*{10mm} since $H''^{(j)} = \id$ on $(V - W) \times I$ by (ii) and $\Lambda_j \cap |{\cal E}_j| = \emptyset$ 
by Setting~\ref{setting_cpt_even}$^+$\,(ii). 

\itemiv $\pi_M H''^{(j)}(P \times I) \cap Q_1^{(m-1)} = \emptyset$. In fact, 
by the assumption $H^{(j-1)}$ satisfies the same condition and 
so does $H'^{(j)}$ as a fine $C^\infty$ approximation of $H^{(j-1)}$.
Since $A^{(j)} = \id$ on $Q_1^{(m-1)} \times I$ by (i), the conclusion holds. 
\eit 

\item[{[3]}] To obtain a factorization of $H''{}^{(j)}$ based on [2](iii) we apply Lemma~\ref{Isotopy_ext} to $N$, 
$K = \Sigma_{j} \cap N$, $L = (\Sigma_{j-1} \cup \Lambda_j)\cap N$ and $H''{}^{(j)}|_{N \times I}$. 
Note that 
\bit 
\itemI $K \cap L = \Sigma_{j-1} \cap N$ and $K - L = |{\cal E}_j| \cap N$ by Setting~\ref{setting_cpt_even}$^+$\,(i),(ii), 
\itemII $H''^{(j)} \in {\rm Isot}(M, M_{N} \cup \Sigma_{j-1})_0$ and 
$\pi_M H''^{(j)}(|{\cal E}_j| \times I) \cap (\Sigma_{j-1} \cup \Lambda_j) = \emptyset$ by [2](ii), (iii). 
\eit 
Lemma~\ref{Isotopy_ext} induces a factorization $H''{}^{(j)} = G^{(j)}H^{(j)}$ for some \\ 
\hspace*{10mm} $G{}^{(j)} \in {\rm Isot}(M, M_{N} \cup \Sigma_{j-1} \cup \Lambda_j)_0$ \ \ and \ \ 
$H^{(j)} \in {\rm Isot}(M, M_{N} \cup \Sigma_{j})_0$. 

From [2](iv) and the facts that $G^{(j)} = \id$ on $\Lambda_j \times I$ and $Q_1^{(m-1)} \subset \Lambda_j$, 
it follows that $H^{(j)} = (G^{(j)})^{-1}H''{}^{(j)}$ also satisfies the next condition : 
\bit 
\itemiii $\pi_M H^{(j)}(P \times I) \cap Q_1^{(m-1)} = \emptyset$. 
\eit 
\hspace*{-6mm} This completes the inductive step. 
\eenum

In \cite{Tsuboi3} the inductive argument in Setp 4 is applied to each open $m$-cell in ${\cal S}_c$.  
The grouping of $m$-cells in Setting~\ref{setting_cpt_even}$^+$ may decrease the number $k$ of inductive steps in Setp 4 effectively. 
This advantage is essentially used to treat open $2m$-manifolds in Section 5.3. 
\vskip 2mm 
\noindent (Step 5) \ \ Factorization of $f$ \\
\hspace*{5mm} We denote the 1-levels of the isotopies appeared in Steps 1 - 4 by the corresponding small letters. \\
\hspace*{5mm} As a summary of the previous steps, we have the following factorizations : 
\benum
\item Steps 1 - 3 : \hspace*{5mm} $f = gh$, \hspace*{5mm} $\overline{h} \Doteq h$, \hspace*{5mm} $a\overline{h} = h' = g'h''$, 
\hspace*{5mm} $\overline{h}'' \Doteq h''$, \hspace*{5mm} $h''' = a''\overline{h}''$ \\[1mm] 
\hspace*{10mm} 
$\bary[c]{@{}l@{ \ }l}
\therefore \ \ f & = \, gh \, = \,g \, \overline{h}\, (\overline{h}^{-1}h) \, = \,g \,(a^{-1}h') (\overline{h}^{-1}h) \, = \,g \,a^{-1}(g'h'') (\overline{h}^{-1}h) \\[2mm] 
& = \, g \,a^{-1}g' \,\overline{h}'' {(\overline{h}''}^{-1} h'') (\overline{h}^{-1}h) \, = \,g \,a^{-1}g' {(a''}^{-1}h''') {(\overline{h}''}^{-1} h'') (\overline{h}^{-1}h). 
\eary$
\vskip 4mm  

\item Step 4 : \hspace*{5mm} $h^{(0)} = h'''$, \hspace*{5mm} $h'{}^{(j)} \Doteq h^{(j-1)}$, \hspace*{5mm} $a^{(j)}h{'}^{(j)} = h''{}^{(j)} = g^{(j)}h^{(j)}$ 
\hspace{5mm} $(j=1, \cdots, k)$. 
\bit 
\itemI $h^{(j-1)} = h{'}^{(j)}({h'{}^{(j)}}^{-1}h^{(j-1)}) = ({a^{(j)}}^{-1} g^{(j)})h^{(j)}({h{'}^{(j)}}^{-1}h^{(j-1)})$ 

\vskip 2mm 
\itemII $\bary[t]{@{}l@{ \ }l} 
h''' & = \, ({a^{(1)}}^{-1} g^{(1)})h^{(1)}({h{'}^{(1)}}^{-1}h^{(0)}) \\[2mm]
& = \, ({a^{(1)}}^{-1} g^{(1)})({a^{(2)}}^{-1} g^{(2)})h^{(2)}({h{'}^{(2)}}^{-1}h^{(1)})({h{'}^{(1)}}^{-1}h^{(0)}) \\[2mm] 
& = \, ({a^{(1)}}^{-1} g^{(1)})({a^{(2)}}^{-1} g^{(2)})({a^{(3)}}^{-1} g^{(3)})h^{(3)}({h{'}^{(3)}}^{-1}h^{(2)})({h{'}^{(2)}}^{-1}h^{(1)})({h{'}^{(1)}}^{-1}h^{(0)}) \\[2mm]
& = \, \cdots \, = \, \big[ \prod_{j=1}^k ({a^{(j)}}^{-1} g^{(j)})\big] h^{(k)}\big[\prod_{j=k}^1 ({h{'}^{(j)}}^{-1}h^{(j-1)})\big] 
\eary$ 
\vskip 3mm

\itemiii  $f \, = \, g \,a^{-1}g' {a''}^{-1} \big[ \prod_{j=1}^k ({a^{(j)}}^{-1} g^{(j)})\big] h^{(k)} 
\big[\prod_{j=k}^1 ({h{'}^{(j)}}^{-1}h^{(j-1)})\big] {(\overline{h}''}^{-1} h'') (\overline{h}^{-1}h)$
\eit 
\vskip 3mm 

\item $f_1 := g \,a^{-1}g' {a''}^{-1} \prod_{j=1}^k ({a^{(j)}}^{-1} g^{(j)}) \in {\rm Diff}^r(M, M_N)_0$ :  
\bit 
\itemI $f_1 = [g \,a^{-1}g^{-1}] \cdot [gg' {a''}^{-1} (gg')^{-1}] (gg') \prod_{j=1}^k ({a^{(j)}}^{-1} g^{(j)})$ 
\itemII Let $b_{-1} = g \,a^{-1}g^{-1}$, \ $b_0 = gg' {a''}^{-1} (gg')^{-1}$, \ $g_0 = gg'$ \ and \\
\hspace*{25mm} $g_j = g^{(j)}$, \ $a_j = {a^{(j)}}^{-1}$, \ $b_j = a_j^{g_0g_1 \cdots g_{j-1}}$ \ \ $(j=1, \cdots, k)$. \\
From Fact~\ref{fact_conjugation} it follows that \\
\hspace*{10mm} $\bary[t]{@{}c@{ \ }l}
(gg') \prod_{j=1}^k ({a^{(j)}}^{-1} g^{(j)})
& = g_0\prod_{j=1}^k (a_j g_{j}) 
= \big[\prod_{j=1}^k (g_{j-1}a_j)\big] g_{k} \\[2mm] 
& = \big[\prod_{j=1}^k b_j \big] g_0g_1 \cdots g_k 
= \big[\prod_{j=1}^k b_j\big] (gg') \big[\prod_{j=1}^k g^{(j)}\big] \ \ \ \text{and} 
\eary$ \\[2mm]
\hspace*{10mm} $f_1 = \big[\prod_{j=-1}^k b_j\big] (gg') \big[\prod_{j=1}^k g^{(j)}\big]$
\vskip 2mm 

\itemiii Note that \ $gg' \in {\rm Diff}^r(M, M_{N} \cup Q_1)_0$, \ 
$g{}^{(j)} \in {\rm Diff}(M, M_{N} \cup \Sigma_{j-1} \cup \Lambda_j)_0$ $(j=1, \cdots, k)$ \ and \\ 
$b_j \in {\rm Diff}^r(M, M_{V_j})_0^c$ 
for some $V_j \in {\cal B}_f({\rm Int}\,N)$ $(j=-1,0,1, \cdots,k)$. 
\eit 
\vskip 1mm 
\item $f_2 := \big[\prod_{j=k}^1 ({h{'}^{(j)}}^{-1}h^{(j-1)})\big] {(\overline{h}''}^{-1} h'') (\overline{h}^{-1}h) \in {\rm Diff}^r(M, M_N)_0$ : 
\vskip 1mm 
\bit 
\item[(i)\ ] $f_2$ is chosen arbitrarily close to $\id_M$. 
\item[(ii)\,] If $P \cap Q_1 \cap {\rm Int}\,N$ is a finite set, then there is $\widehat{U} \in {\cal B}_f({\rm Int}\,N)$ with 
$P \cap Q_1 \cap {\rm Int}\,N \subset {\rm Int}\,\widehat{U}$. 
By (i) for the open cover $\{ {\rm Int}\,N - P, {\rm Int}\,\widehat{U}, {\rm Int}\,N - Q_1\}$ of ${\rm Int}\,N$, 
there is a factorization \break 
$f_2 = \hat{h}\,\hat{a}\,\hat{g}$ \ such that \ 
$\hat{h} \in {\rm Diff}^r(M, M_N \cup P)_0$, \ \ 
$\hat{a} \in {\rm Diff}^r(M, M_{\widehat{U}})_0$, \ \ 
$\hat{g} \in {\rm Diff}^r(M, M_N \cup Q_1)_0$. 
By \cite[Remark 2.1]{Tsu09} we can modify $\hat{a}$ and $\hat{g}$, so that 
$\hat{a} \in {\rm Diff}^r(M, M_{\widehat{U}})_0^c$. 
\eit 

\item The case where $P \cap Q_1 \cap {\rm Int}\,N$ is a finite set :  
\bit 
\itemI There is a factorization $f_2 = \hat{h}\,\hat{a}\,\hat{g}$ as in (4)(ii). Let \\
\hspace*{10mm} $b'_{-2} := \hat{a}$, \ \ $b'_{j} := b_j^{\hat{g}}$ \ $(j=-1,0, \cdots, k$) \ \ and \\[1mm] 
\hspace*{10mm} $g'{}^{(1)} := \hat{g} (gg')g^{(1)}$, \ \ 
$g'{}^{(j)} := g^{(j)}$ $(j=2, \cdots, k)$, \ \ 
$h'{}^{(k)} := h^{(k)} \hat{h}$. \\[1mm] 
Then it follows that \hspace{3mm} 
$\hat{g} \big[\prod_{j=-1}^k b_j\big]
= \hat{g} \big[\prod_{j=-1}^k b_j\big]\hat{g}^{-1} \hat{g}
= \big[\prod_{j=-1}^k b_j^{\hat{g}} \big] \hat{g}$ \hspace{3mm} and \\[2mm] 
\hspace*{10mm} 
$\bary[t]{@{}l@{ \ }l}
f = f_1 h^{(k)} f_2 
& = \, \big[\prod_{j=-1}^k b_j\big] (gg') \big[\prod_{j=1}^k g^{(j)}\big] h^{(k)} \hat{h}\,\hat{a}\,\hat{g} \\[2.5mm] 
& = \, (\hat{a}\hspace{0.2mm}\hat{g})^{-1}\Big[\hat{a}\hspace{0.2mm}\hat{g} \big[\prod_{j=-1}^k b_j\big] (gg') \big[\prod_{j=1}^k g^{(j)}\big] h^{(k)} \hat{h}\Big] \hat{a}\hspace{0.2mm}\hat{g} \\[2.5mm] 
& = \, (\hat{a}\hspace{0.2mm}\hat{g})^{-1}\Big[\hat{a} \big[\prod_{j=-1}^k b_j^{\hat{g}} \big] \hat{g} (gg') \big[\prod_{j=1}^k g^{(j)}\big] h^{(k)} \hat{h}\Big]\hat{a}\hspace{0.2mm}\hat{g} \\[2.5mm] 
& = \, (\hat{a}\hspace{0.2mm}\hat{g})^{-1}\Big[\big[\prod_{j=-2}^k b'_j\big] \big[\prod_{j=1}^k g'{}^{(j)}\big] h'{}^{(k)}\Big]\hat{a}\hspace{0.2mm}\hat{g}
\eary$ 
\vskip 2mm 

\itemII Note that \ 
${g'}^{(j)} \in {\rm Diff}^r(M, M_{N} \cup \Lambda_j)_0$ $(j=1, \cdots, k)$, \ 
${h'}^{(k)} \in {\rm Diff}^r(M, M_{N} \cup \Sigma_{k})_0$ \ and \\ 
$b'_j \in {\rm Diff}^r(M, M_{V_j'})_0^c$ for some $V_j' \in {\cal B}_f({\rm Int}\,N)$ $(j=-2, -1,0,1, \cdots,k)$
\eit 
\eenum 

In the subsequent subsections we obtain some estimates of $cl\,f$ and $clb^f f$  from the factorization of $f$ in Step 5\,(5). 

\subsection{Compact manifold case} 
\subsubsection{\bf Compact manifold case I --- Triangulations} \mbox{} 

\begin{setting}\label{setting_even_cpt_triang}
Suppose $N$ is a compact $2m$-manifold possibly with boundary ($m \geq 3$), 
$1 \leq r \leq \infty$, $r \neq 2m+1$, 
${\mathcal T}$ is a $C^\infty$ triangulation of $N$, $(P, Q) := (|{\cal T}^{(m)}, |{{\cal T}^\ast}^{(m)}|)$,  
${\cal S}$ is the set of $m$-simplices of ${\cal T}$, 
${\cal S}' := \{ \sigma \in {\cal S} \mid \sigma \not \subset \partial N \}$  
 and $\widetilde{N} := N \cup_{\partial N} (\partial N \times [0,1))$.
\end{setting}

\begin{setting_5.3+} 
Suppose ${\cal S} \supset {\cal F} \supset {\cal S}'$,  
$\{ {\cal F}_j \}_{j=1}^k$ is a finite cover of ${\cal F}$ (as a set), 
${\cal G}_j := {\cal S} - {\cal F}_j$, 
$(P_{{\cal F}_j}, Q_{{\cal G}_j}) := (P^{(m-1)} \cup |{\cal F}_j|, Q^{(m-1)} \cup |{\cal G}_j^{\ast}|)$, 
$O_j : = N-Q_{{\cal G}_j}$ and $\ell_j := clb^f\,(P_{{\cal F}_j}, \widetilde{O}_j)$ $(j=1, \cdots, k)$.
\end{setting_5.3+}

\begin{thm}\label{thm_cpt_bdry_even} 
In Setting~{\rm \ref{setting_even_cpt_triang}, ~\ref{setting_even_cpt_triang}$^+$}: 
\benum 
\item[{\rm [I]}\,] $cld\,{\rm Diff}^r(N, \partial N)_0 \leq 3k+5$ \ and \ 
$clb^f\!d\,{\rm Diff}^r(N, \partial N)_0 \leq 2\Big(\mbox{\small $\ds \sum_{j=1}^k$}\, \ell_j + k + m + 2\Big) \leq 2(m+2)(k+1)$ \\ 
\hspace*{5mm} if each $|{\cal F}_j|$ $(j=1, \dots, k)$ is strongly displaceable from $|{\cal F}_j| \cup (\partial N \times [0,1))$ in $\widetilde{N}$.  
\item[{\rm [II]}] $cld\,{\rm Diff}^r(N, \partial N)_0 \leq 2k+7$ \ and \ 
$clb^f\!d\,{\rm Diff}^r(N, \partial N)_0 \leq 2\,\mbox{\small $\ds \sum_{j=1}^k$}\, \ell_j + k + 2m + 6 \leq (2m+3)(k+1) + 3$ \\ 
\hspace*{5mm} if each $|{\cal F}_j|$ $(j=1, \dots, k)$ is strongly displaceable from $|{\cal S}| \cup (\partial N \times [0,1))$ in $\widetilde{N}$. 
\eenum 
\end{thm} 

\begin{compl}\label{compl_ell_j} \mbox{}
\benum
\item \bit 
\itemI $\ell_j \leq m+1$. 
\itemII If $K_j$ is a compact $C^\infty$ subpolyhedron in $N$ with $K_j \subset P_{{\cal F}_j}$ and 
$P_{{\cal F}_j}$ is weakly absorbed to $K_j$ in $\widetilde{O}_j$ keeping $K_j$ invariant, then 
$\ell_j \leq clb^f(K_j, \widetilde{O}_j) \leq \dim K_j +1$.
\eit 
\item[(2)] The condition in the case [II] holds, for example, if each $|{\cal F}_j|$ has an arbitrarily small $2m$-disk neighborhood in $\widetilde{N}$. 
\eenum 
\end{compl}

\begin{cor}\label{cor_cpt_bdry_even} In Setting~$\ref{setting_even_cpt_triang}$: 
 Let $\ell := \# {\cal S}'$ $($i.e., the number of $m$-simplices of ${\cal T}$ not in $\partial N$$)$. 
Then 
\hspace*{30mm} $cld\,{\rm Diff}^r(N, \partial)_0 \leq 2\ell+7$ \ and \ 
$clb^f\!d\,{\rm Diff}^r(N, \partial N)_0 \leq (2m+1)(\ell +1) + 5$. 
\end{cor} 

When $N$ is a closed $2m$-manifold, the estimates in Corollary~\ref{cor_cpt_bdry_even} are compared with those in \cite{Tsuboi3}, \\
\hsp that is, 
\btab[t]{ll}
$cl\,{\rm Diff}^r(N)_0 \leq 4\ell+11$  & (the last sentense in \cite[Proof of Theorem 1.2]{Tsuboi3} (p.166)), \\[2mm] 
$clb^f\,{\rm Diff}^r(N)_0 \leq 4(\ell+4)m + 3\ell + 7$ & (\cite[Proof of Corollary 1.3]{Tsuboi3} (p.173)). 
\etab 
\vskip 4mm 

\begin{proof}[{\bf Proof of Theorem~\ref{thm_cpt_bdry_even}.}] \mbox{}  

We apply the argument in Subsection 5.1.2 in $M \equiv \widetilde{N}$. 
Take any $f \in {\rm Diff}^r(N, \partial N)_0 \cong {\rm Diff}^r(M, M_N)_0$ 
and $F \in {\rm Isot}^r(M, M_N)_0$ with $F_1 = f$. 
Our task is to estimate $cl\,f$ and $clb^ff$ in ${\rm Diff}^r(M, M_N)_0$. 

The cell complex $Q$ underlies a subcomplex of $sd\,{\cal T}$. We denote this subcomplex by $\underline{Q}{}$. 
To apply (Step~4) we prefer $\underline{Q}$ consisting of smooth closed simplices rather than $Q$ itself consisting of piecewise smooth cells. 
For each $m$-simplex $\sigma$ of $P$ 
its dual $m$-cell $\sigma^\ast$ underlies a subcomplex of $\underline{Q}$, 
which we denote by $\underline{\sigma}^\ast$. Note that $Q^{(m-1)} \subsetneqq \ \underline{Q}^{(m-1)}$ and 
$P \cap \underline{Q}^{(m-1)} \neq \emptyset$

\benum[(1)] 
\item We define $m$-dimensional stratified subsets $P_0$ and $\underline{Q}_0$ of $N$ as follows : 
\bit 
\itemI Let $P_0 : = |{\cal T}^{(m-1)}| \cup |{\cal S}'|$. This is a subcomplex of $P$ and $M_{N} \cup P_0 = M_{N} \cup P$. 

\itemII Take a small neighborhood $U$ of the finite set 
$P_0 \cap Q = \bigcup \{ \sigma \cap \sigma^\ast \mid \sigma \in {\cal S}' \}$
in ${\rm Int}\,N - (P^{(m-1)} \cup Q^{\,(m-1)})$ which is a finite disjoint union of closed $2m$-disks. 
Then, there exists $K' \in {\rm Isot}(M, M_U)_0$ such that 
each $\sigma \in {\cal S}'$ intersects a unique open $m$-simplex in $K'_1(\underline{\sigma}^\ast)$  transversely at one point. 
Let $\underline{Q}_0 := K'_1(\underline{Q})$. 
\eit 
Note that $P_0^{(m-1)} \cap \underline{Q}_0 = P_0 \cap \underline{Q}_0^{(m-1)} = \emptyset$. 

\item We apply 
(Step 1) $\sim$ (Step 3) in Subsection 5.1.2 to $(M, N, F, P_0, \underline{Q}_0)$ to 
obtain the corresponding factorization of the isotopy $F$. 
\bit 
\itemI In (Step 1)\,(1) we obtain the isotopy $K$ and $\underline{Q}_1 = K_1(\underline{Q}_0)$.
\eit 
 In Setting~\ref{setting_cpt_even}$^+$ (replacing open simplices by closed simplices) we have 
\bit 
\itemII $\text{[the set of $m$-simplices of $P_0$]} = {\cal S}'$, \\  
${\cal S}_0 = \{ \sigma \in {\cal S}' \mid H^{(0)} \equiv \id \text{ on $[$a neighborhood of $\stackrel{\circ}{\sigma}\,] \times I$}\}$ \ \ and \ \ 
${\cal S}_c = {\cal S}' - {\cal S}_0$. 
\eit 
Let ${\cal E} = {\cal S}_c$ and take a partition ${\cal E} = \bigcup_{j=1}^k {\cal E}_j$ such that 
${\cal E}_j \subset {\cal F}_j$ $(j=1, \cdots, k)$. 
(For example, let ${\cal E}_j := {\cal E} \cap {\cal F}_j - \bigcup_{i=1}^{j-1} {\cal F}_i$.)
Then, it follows that 
\bit 
\itemiii (a) $\Sigma_k = P_0$, \hspace{3mm} (b) \ ${\cal U} = $ the set of $m$-simplices of $\underline{Q}_1$ \ \ and \\
(c) $\Lambda_j \equiv \underline{Q}_1^{(m-1)} \cup \bigcup \{ \tau \in {\cal U} \mid \ \stackrel{\circ}{\tau} \cap \, |{\cal E}_j| = \emptyset \}$.  
\eit 
\vskip 1mm 
We apply (Step 4) and (Step 5) to $\{ {\cal E}_j \}_{j=1}^k$ to obtain a factorization of $f$ of the following form :   
\vskip 1mm 
\bit 
\itemiv $f = \, (\hat{a}\hspace{0.2mm}\hat{g})^{-1}\Big[\big[\prod_{j=-2}^k b'_j\big] 
\big[\prod_{j=1}^k g'{}^{(j)}\big] h'{}^{(k)}\Big]\hat{a}\hspace{0.2mm}\hat{g}$, where  
\vskip 2mm 
\bit 
\itema ${g'}^{(j)} \in {\rm Diff}^r(M, M_{N} \cup \Lambda_j)_0$ $(j=1, \cdots, k)$, 
\itemb ${h'}^{(k)} \in {\rm Diff}^r(M, M_{N} \cup P_0)_0$, 
\itemc $b'_j \in {\rm Diff}^r(M, M_{V_j'})_0^c$ for some $V_j' \in {\cal B}_f({\rm Int}\,N)$ $(j=-2, -1,0,1, \cdots,k)$. 
\eit 
\eit 
\vskip 1mm

\item Next we apply Examples~\ref{exp_(S,U)_I}, ~\ref{exp_(S,U)_II} and Lemma~\ref{lem_mapping_cylinder} to this situation. \\
Let $L_j := P_{{\cal F}_j}$ in [I] and $L_j = P$ in [II] $(j=1, \cdots, k)$. It follows that 
\bit 
\itemI $cld\,{\rm Diff}^r(N, \partial N \cup P)_0 \leq 2$ \ \ and \ \ $clb^fd\,{\rm Diff}^r(N, \partial N \cup P)_0 \leq 2m+1$,  

\itemII 
any $\underline{g}\in {\rm Diff}^r(N, \partial N \cup Q_{{\cal G}_j})_0$ has a factorization $\underline{g}= \underline{g}_1\underline{g}_2$ such that \\[2mm] 
\hsp 
\btab[c]{lll}
$\underline{g}_1 \in {\rm Diff}^r(N, \partial N \cup Q_{{\cal G}_j})_0^c$ & and & $clb^f(\underline{g}_1) \leq 2\ell_j$ in ${\rm Diff}^r(N, \partial N \cup Q_{{\cal G}_j})_0$, \\[2mm] 
$\underline{g}_2 \in {\rm Diff}^r(N, \partial N \cup Q_{{\cal G}_j} \cup L_j)_0^c$ & and & $clb^f(\underline{g}_2) \leq 1$ in ${\rm Diff}^r(N, \partial N \cup Q_{{\cal G}_j} \cup L_j)_0$. 
\etab 
\eit 
\vskip 2mm 
The assertion (i) follows from Example~\ref{exp_(S,U)_I} and Lemma~\ref{lem_mapping_cylinder} for $(Q^{(m-1)}, P)$. 
The assertion (ii) follows from Example~\ref{exp_(S,U)_II}, 
Lemma~\ref{lem_mapping_cylinder} for $(P_{{\cal F}_j}, Q_{{\cal G}_j})$ and the assumption on ${\cal F}_j$. 
We also note that Complement~\ref{compl_ell_j}\,(1) follows from Example~\ref{exp_(S,U)_II} for $(P_{{\cal F}_j}, Q_{{\cal G}_j})$. 

\item In order to apply the estimates on $cl$ and $clb^f$ in (3) to (2)(iv), 
we need to transfrom them to the corresponding estimates for $\underline{Q}_1$, 
using the diffeomorphism $\phi := K_1K_1' \in {\rm Diff}(M, M_{N})_0$. \\
We use the following notations : 
\bit 
\itemI 
$(P_1, Q_1,  Q_{1,{\cal G}_j}) := \phi(P_0, Q, Q_{{\cal G}_j})$ (as a cell complex), \ 
$L_{1,j} : = \phi(L_j)$ $(j=1, \cdots, k)$, \\
$\sigma^\flat := \phi(\sigma^\ast)$ for $\sigma \in {\cal S}$ \ and \ 
${\cal C}^\flat := \{ \sigma^\flat \mid \sigma \in {\cal C} \}$ for ${\cal C} \subset {\cal S}$. 

\itemII Note that 
(a) $\underline{Q}_1$ is a subdivision of the cell complex $Q_1$ into a simplicial complex, \\
(b) $Q_{1,{\cal G}_j} = Q_1^{(m-1)} \cup |{\cal G}_j^{\,\flat}| \subset \Lambda_j$ and (c) $M_N \cup L_{1, j} = M_N \cup P_1$ in [II]. 

\itemiii The choice of $K'$ and $K$ implies that \\
(a) $\sigma^\flat$ is sufficiently close to $\sigma^\ast$ so that 
$(P - \stackrel{\circ}{\sigma}) \cap \sigma^\flat = \emptyset$, and \\
(b) $P_1$ differs from $P_0$ only in $U$ and 
there exists $\eta \in {\rm Diff}(M, M_U)_0$ with $P_1 = \eta(P_0)$. \\
By \cite[Remark 2.1]{Tsu09} we may assume that $\eta \in {\rm Diff}(M, M_U)_0^c$. 
\eit 
\item[] 
Since the tuple $(N, P_0, Q_{{\cal G}_j}, L_j)$ corresponds to $(N, P_1, Q_{1, {\cal G}_j}, L_{1,j})$ under the diffeomorphism $\phi$ 
and $\partial N \cup P =\partial N \cup P_0$, 
from (3) we have the following conclusions :  
\bit 
\item[(i)$'$\,] $cld\,{\rm Diff}^r(N, \partial N \cup P_1)_0 \leq 2$ \ \ and \ \ $clb^fd\,{\rm Diff}^r(N, \partial N \cup P_1)_0 \leq 2m+1$,  

\item[(ii)$'$] 
any $\underline{g}\in {\rm Diff}^r(N, \partial N \cup Q_{1,{\cal G}_j})_0$ has a factorization $\underline{g}= \underline{g}_1\underline{g}_2$ such that \\[2mm] 
\hsp 
\btab[c]{lll}
$\underline{g}_1 \in {\rm Diff}^r(N, \partial N \cup Q_{1, {\cal G}_j})_0^c$ & and & $clb^f(\underline{g}_1) \leq 2\ell_j$ in ${\rm Diff}^r(N, \partial N \cup Q_{1, {\cal G}_j})_0$, \\[2mm] 
$\underline{g}_2 \in {\rm Diff}^r(N, \partial N \cup Q_{1,{\cal G}_j} \cup L_{1,j})_0^c$ & and & $clb^f(\underline{g}_2) \leq 1$ in ${\rm Diff}^r(N, \partial N \cup Q_{1,{\cal G}_j} \cup L_{1,j})_0$. 
\etab 
\eit 
\vskip 2mm 

\item Finally we deduce the estimates of $cl\,f$ and $clb^f f$ in ${\rm Diff}^r(M, M_{N})_0$. 
First we note that 
\bit 
\itemi (a) \,${g'}^{(j)} \in {\rm Diff}^r(M, M_{N} \cup \Lambda_j)_0 \subset {\rm Diff}^r(M, M_{N} \cup Q_{1,{\cal G}_j})_0$  
by (4)(ii)(b), and \\ 
\hspace*{5.5mm} $cl\,{g'}^{(j)} \leq 2$, \ $clb^f {g'}^{(j)} \leq 2\ell_j+1$ by (4)(ii)$'$. 
\item[] (b) \,$cl\,{h'}^{(k)} \leq 2$, \ $clb^f{h'}^{(k)} \leq 2m+1$ by (3)(i). 
\item[] (c) \,$cl\,b'_j \leq 1$, \ $clb^f b'_j \leq 1$, \ so that \ $cl\,\big(\prod_{j=-2}^k b'_j\big) \leq k+3$, \ 
$clb^f\big(\prod_{j=-2}^k b'_j \big)\leq k+3$.
\eit 
\vskip 2mm 
{Case [I] :} From (i) it follows that \ $cl\,f \leq (k+3) + 2k + 2 = 3k+5$ \ and \\[1mm] 
\hspace*{20mm} 
$\bary[t]{l@{ \ }l}
clb^f f & \leq (k+3) + \mbox{\small $\ds \sum_{j=1}^k$}\,(2\ell_j+1) + 2m+1
= 2\Big(\mbox{\small $\ds \sum_{j=1}^k$}\, \ell_j + k + m + 2\Big) \\[5mm]
& \leq 2(m+2)(k+1) = (n+4)(k+1).
\eary$ \\[2mm] 
{Case [II] :}  In this case we can further refine the factorization of $f$ in (2)(iv). 
\bit 
\itemii By (4)(ii)$'$ and (4)(ii)(c)
each ${g'}^{(j)} \in {\rm Diff}^r(M, M_{N} \cup Q_{1,{\cal G}_j})_0$ has a factorization \ ${g'}^{(j)} = g_jh_j$ \\ such that 
\btab[t]{lll}
$g_j \in {\rm Diff}^r(M, M_{N} \cup Q_{1, {\cal G}_j})_0^c$ & and & $clb^f(g_j) \leq 2\ell_j$ in ${\rm Diff}^r(M, M_{N} \cup Q_{1, {\cal G}_j})_0$, \\[2mm] 
$h_j \in {\rm Diff}^r(M, M_{N} \cup Q_{1,{\cal G}_j} \cup P_1)_0^c$ & and & $clb^f(h_j) \leq 1$ in ${\rm Diff}^r(M, M_{N} \cup Q_{1,{\cal G}_j} \cup P_1)_0$. 
\etab 
\vskip 2mm 
Let $g_1' = g_1$, $g_j' := g_j^{h_1 \cdots h_{j-1}}$ $(j=2, \cdots, k)$ and $h' := h_1 \cdots h_k$. 
Then, it follows that 
\vskip 1mm 
\bit 
\itema $g_j' \in {\rm Diff}^r(M, M_{N})_0^c$ \ $(j=1, \cdots, k)$ \ and \ $h' \in {\rm Diff}^r(M, M_{N} \cup P_1)_0$, 
\vskip 1mm 
\itemb $\prod_{j=1}^k g'{}^{(j)} = \prod_{j=1}^k (g_jh_j) = g_1 g_2^{h_1} g_3^{h_1h_2} \cdots g_k^{h_1 \cdots h_{k-1}} (h_1 \cdots h_k)
= \big(\prod_{j=1}^k g_j'\big) h'$. 
\vskip 1mm 
\itemc $cl\,g_j' \leq 1$, \ $clb^f g_j' \leq 2\ell_j$ \ \ and \ \ $cl\,h' \leq 2$, \ $clb^f h' \leq 2m+1$ \ by (4)(i)$'$
\itemd $cl\,(h'{h'}^{(k)}) \leq cl\,h' + cl\,{h'}^{(k)} \leq 2+2 = 4$. 
\eit 
\vskip 1mm 
\itemiii To obtain a finer estimate on $clb^f (h'{h'}^{(k)})$
we decompose the composition $h'{h'}^{(k)}$ using $\eta \in {\rm Diff}(M, M_U)_0^c$ given in (4)(iii)(b). 
First note that 
\bit 
\itema $\eta' :=(\eta^{-1})^{h'} \in {\rm Diff}(M, M_{U'})_0^c$ \ for $U' := h'(U) \in {\cal B}_f({\rm Int}\,N)$. 
\eit 
Since $\eta(P_0) = P_1$, we have $({h'}^{(k)})^\eta \in {\rm Diff}^r(M, M_{N} \cup P_1)_0$, which implies that 
\bit 
\itemb $h : = h'({h'}^{(k)})^\eta \in {\rm Diff}^r(M, M_{N} \cup P_1)_0$ \ and \ $cl\,h \leq 2$, \ $clb^f h \leq 2m+1$ by (4)(i)$'$. 
\eit 
Since $(h')^\eta = \eta\eta' h'$, 
it follows that $(h'{h'}^{(k)})^\eta = (h')^\eta({h'}^{(k)})^\eta = \eta \eta' h$ and  
\bit 
\itemc $clb^f (h'{h'}^{(k)}) = clb^f (\eta \eta' h) \leq 1 + 1 + 2m+1 = 2m+3$. 
\eit 

\itemiv From (ii) and (iii) it follows that 
$cl\,f \leq (k+3) + k + 4 = 2k+7$ \ and \\[1mm] 
$clb^f f \leq (k+3) + \mbox{\small $\ds \sum_{j=1}^k$}\,(2\ell_j) + 2m+3
= 2\,\mbox{\small $\ds \sum_{j=1}^k$}\, \ell_j + k + 2m + 6 \leq (2m+3)(k+1) + 3 = (n+3)(k+1)+3.$  
\eit 
\vskip 2mm 
This completes the proof. 
\eenum 
\vspace{-7mm} 
\end{proof} 

\begin{proof}[{\bf Proof of Corollary~\ref{cor_cpt_bdry_even}.}] \mbox{}
\benum 
\item Let ${\cal S'} = \{ \sigma_1, \cdots, \sigma_\ell \}$ and 
take ${\cal F} := {\cal S}'$ and its cover ${\cal F}_j = \{ \sigma_j \}$ $(j=1, \cdots, \ell)$.  
Since each $|{\cal F}_j| = \sigma_j$ has an arbitrarily small $2m$-disk neighborhood in $\widetilde{N}$, 
by Theorem~\ref{thm_cpt_bdry_even}\,[II] it follows that 
\hspace*{10mm} $cld\,{\rm Diff}^r(N, \partial N)_0 \leq 2\ell+7$ \ and \ 
$clb^f\!d\,{\rm Diff}^r(N, \partial N)_0 \leq 2\,\mbox{\small $\ds \sum_{j=1}^\ell$}\, \ell_j + \ell + 2m + 6$.

\item Next we show that $\ell_j \leq m$ $(j=1, \cdots, \ell)$.
Recall that $P_{{\cal F}_j} = P^{(m-1)} \cup \sigma_j$. 
Take any face $\tau$ of $\sigma_j$ and let $K_j :=  P^{(m-1)} - \stackrel{\circ}{\tau}$. 
Then $P_{{\cal F}_j}$ elementarily collapses to the subcomplex $K_j$ in the PL-sense. 
Hence, $P_{{\cal F}_j}$ is weakly absorbed to $K_j$ in $\widetilde{O}_j$ keeping $K_j$ invariant.
Thus, $\ell_j \leq clb^f(K_j, \widetilde{O}_j) \leq m$ by Compliment~\ref{compl_ell_j}\,(1)(ii).   
\item From (1), (2) we have $clb^f\!d\,{\rm Diff}^r(N, \partial N)_0 \leq (2m+1)(\ell +1) + 5$.
\eenum 
\vskip -7mm 
\end{proof} 

\subsubsection{\bf Compact manifold case II --- Handle decompositions} \mbox{} 

\begin{setting}\label{setting_even_cpt} 
Suppose $M$ is a $2m$-manifold without boundary $(m \geq 3)$, $1 \leq r \leq \infty$, $r \neq 2m+1$, 
${\cal H}$ is a handle decomposition of $M$, 
$(P, Q) = (P_{\cal H}^{(m)}, P_{\cal H^\ast}^{(m)})$ (the $m$-skeletons of the core complexes of ${\cal H}$ and ${\cal H}^\ast$) 
and ${\cal S}$ is the set of all open $m$-cells of $P$. 
Suppose $N \in {\cal S}{\cal M}_c(M)$, $N_1 \in {\cal S}{\cal M}_c(M, {\cal H})$, 
$N_2 \in {\cal S}{\cal M}_c(M, {\cal H}^\ast)$, $N \Subset N_1 \Subset N_2$ and 
${\rm Int}\,N_1$ includes any $\sigma \in {\cal S}$ with $\sigma \cap {\rm Int}\,N \neq \emptyset$. 
\end{setting}

\begin{setting_5.4+}
Suppose ${\cal C}$ is a subset of ${\cal S}$ such that 
$ \{ \sigma \in {\cal S} \mid \sigma \cap {\rm Int}\,N \neq \emptyset \} \subset {\cal C} \subset \{ \sigma \in {\cal S} \mid \sigma \subset {\rm Int}\,N_1 \}$ 
and $\{ {\cal C}_j \}_{j=1}^k$ is a finite cover of ${\cal C}$ (as a set). 
\end{setting_5.4+}

\begin{prop}\label{cl_even_cpt_bd_handle} 
In Setting~{\rm \ref{setting_even_cpt}, ~ \ref{setting_even_cpt}$^+$}: 
\benum 
\item[{\rm [I]}\,] $cld({\rm Diff}^r(M, M_{N})_0, {\rm Diff}^r(M, M_{N_2})_0) \leq 3k+5$ \ \ and \\[1mm] 
$clb^f\!d\,({\rm Diff}^r(M, M_{N})_0, {\rm Diff}^r(M, M_{N_2})_0) \leq 
2\Big(c({\cal H}) + (k-1)c({\cal H}^{(m)}) + k + 2\Big) \leq 2(m+2)(k+1)$ \\[1mm] 
\hspace*{5mm} if each $Cl_M|{\cal C}_j|$ $(j=1, \dots, k)$ is strongly displaceable from itself in $M$. 
\vskip 2mm 

\item[{\rm [II]}] $cld({\rm Diff}^r(M, M_{N})_0, {\rm Diff}^r(M, M_{N_2})_0) \leq 2k+7$ \ \ and \\ 
$clb^f\!d\,({\rm Diff}^r(M, M_{N})_0, {\rm Diff}^r(M, M_{N_2})_0) 
\leq 2\,\Big(c({\cal H}) + (k-1)c({\cal H}^{(m)})\Big)  + k + 6 \leq (2m+3)(k+1) + 3$ \\[1mm] 
\hspace*{5mm} if each $Cl_M|{\cal C}_j|$ $(j=1, \dots, k)$ is strongly displaceable from $Cl_M|{\cal S}|$ in $M$. 
\eenum 
\end{prop}

\begin{proof} 
The proof is similar to that of Theorem~\ref{thm_cpt_bdry_even} 
and is based on Lemmas~\ref{factorization_hdle_cpt}, ~\ref{lem_even_cpt_handle}. 

Take any $f \in {\rm Diff}^r(M, M_{N})_0$ and $F \in {\rm Isot}^r(M, M_{N})_0$ with $f = F_1$. 
Our goal is to estimate $cl\,f$ and $clb^f f$ in ${\rm Diff}^r(M, M_{N_2})_0$. 

\benum[(1)] 
\item 
First we apply (Step 1) $\sim$ (Step 3) in Subsection 5.1.2 to $(M, N)$, $P$, $Q$ and $F$. 
Then we obtain 
\bit 
\itemI $K \in {\rm Isot}(M, M_N \cup P^{(m-1)})_0$, $Q_1 = K_1(Q)$ and 
$H^{(0)} \in {\rm Isot}^r(M, M_N \cup P^{(m-1)})_0$. 
\eit 
To apply (Step 4) we replace $N$ by $N_1$.  
In Setting~\ref{setting_cpt_even}$^+$ we have 
\bit 
\itemII [the set of open $m$-cells of $P$] $= {\cal S}$, \\
${\cal S}_0 := \{ \sigma \in {\cal S} \mid H^{(0)}\equiv \id \text{ on $[$a neighborhood of $\sigma] \times I$}\}$, \ \  
${\cal S}_c = {\cal S} - {\cal S}_0 \subset {\cal C}$. 
\eit 
Let ${\cal E} = {\cal S}_c$ and take a partition ${\cal E} = \bigcup_{j=1}^k {\cal E}_j$ such that 
${\cal E}_j \subset {\cal C}_j$ $(j=1, \cdots, k)$. 
Then we have  
\bit 
\itemiii $\Sigma_j := P^{(m-1)} \cup |{\cal S}_0| \cup (\bigcup_{i=1}^j |{\cal E}_i|)$ \ $(j=1, \cdots, k)$, \ \ $\Sigma_k = P$, 

${\cal U} := $ the set of open $m$-cells of $Q_1$ \ \ and \\
$\Lambda_j := Q_1^{(m-1)} \cup \bigcup \{ \tau \in {\cal U} \mid \tau \cap |{\cal E}_j| = \emptyset \}$ \ $(j=1, \cdots, k)$. 
\eit 
\vskip 1mm 
We apply (Step 4) to $(M, N_1)$, $P$, $Q_1$, $H^{(0)}$ and $\{ {\cal E}_j \}_{j=1}^k$. Then (Step 5) 
 yields the factorization   
\vskip 1mm 
\bit 
\itemiv $f = \, (\hat{a}\hspace{0.2mm}\hat{g})^{-1}\Big[\big[\prod_{j=-2}^k b'_j\big] 
\big[\prod_{j=1}^k g'{}^{(j)}\big] h'{}^{(k)}\Big]\hat{a}\hspace{0.2mm}\hat{g}$, \ where  
\vskip 2mm 
\bit 
\itema ${g'}^{(j)} \in {\rm Diff}^r(M, M_{N_1} \cup \Lambda_j)_0$ \ $(j=1, \cdots, k)$, \hsh 
(b) ${h'}^{(k)} \in {\rm Diff}^r(M, M_{N_1} \cup P)_0$,  
\itemc $b'_j \in {\rm Diff}^r(M, M_{V_j'})_0^c$ \ for some \ $V_j' \in {\cal B}_f({\rm Int}\,N_1)$ \ $(j=-2, -1,0,1, \cdots,k)$. 
\eit 
\eit 
\vskip 1mm 
\item Next we apply Lemmas~\ref{factorization_hdle_cpt}, ~\ref{lem_even_cpt_handle}. 
Let $(K_j, Q_{{\cal C}_j}) := (P_{N_1}^{(m-1)}\cup |{\cal C}_j|, Q - |{\cal C}_j^\ast|)$, 
$L_j := P_{{\cal C}_j}$ in [I] and $L_j := P$ in [I\hspace{-0.2mm}I] 
and $\ell_j := c(K_j)$, $\ell := c((P_{{\cal H}^\ast|_{N_2}})^{(m-1)})$ $(j=1, \cdots, k)$. 
It follows that 
\bit 
\itemI any $\underline{g} \in {\rm Diff}^r(M, M_{N_1} \cup Q_{{\cal C}_j})_0$ $(j = 1, \cdots, k)$ \ 
has a factorization \ $\underline{g} = \underline{g}_1\underline{g}_2$ \ such that \\[1mm] 
\hsp \btab[t]{@{}l@{ \ }l@{ \ }l}
$\underline{g}_1 \in {\rm Diff}^r(M, M_{N_1} \cup Q_{{\cal C}_j})_0^c$ & and & 
$clb^f(\underline{g}_1) \leq 2\ell_j$ in ${\rm Diff}^r(M, M_{N_1} \cup Q_{{\cal C}_j})_0$, \\[2mm] 
$\underline{g}_2 \in {\rm Diff}^r(M, M_{N_1} \cup Q_{{\cal C}_j} \cup L_j)_0^c$ & and & 
$clb^f(\underline{g}_2) \leq 1$ in ${\rm Diff}^r(M, M_{N_1} \cup Q_{{\cal C}_j} \cup L_j)_0$, 
\etab 
\vskip 2mm 
\itemII $cld\,{\rm Diff}^r(M, M_{N_2} \cup P)_0 \leq 2$ \ \ and \ \ 
$clb^f\!d\,{\rm Diff}^r(M, M_{N_2} \cup P)_0 \leq 2\ell+1$. 
\eit 
The assertion (i) follows from Lemma~\ref{lem_even_cpt_handle} 
for the data : \\
\hsh $M$, ${\cal H}$, $(P, Q)$, $N_1$, ${\cal C}_j$, ${\cal C}_j' := {\cal C}_j$ in [I] and ${\cal S}$ in [I\hspace{-0.2mm}I], 
$(K_j, Q_{{\cal C}_j})$, $L_j$, $O_j := {\rm Int}\,N_1 - Q_{{\cal C}_j}$ and $\ell_j$. \\
By the assumption, $Cl_M|{\cal C}_j|$ is displaceable from $Cl_M|{\cal C}_j'|$ in $O_j$. 
The assertion (ii) follows from Lemma~\ref{factorization_hdle_cpt} for the data $M$, ${\cal H}^\ast$, $(Q^{(m-1)}, P)$, $N_2$ and $\ell$. 
It is easily seen that  
\bit 
\itemiii  $\ell_j \leq c(P_{\cal H}^{(m)}) = c({\cal H}^{(m)}) \leq m+1$, \ \ $\ell \leq c(P_{{\cal H}^\ast}^{(m-1)}) = c(({\cal H}^\ast)^{(m-1)}) \leq m$
\ \ and \\
$\ell_j + \ell \leq c({\cal H}^{(m)}) + c(({\cal H}^\ast)^{(m-1)}) = c({\cal H}) \leq 2m+1$. 
\eit 

\item We transform the estimates in (2) to those associated to $Q_1$, using $K_1 \in {\rm Diff}(M, M_{N})_0$. We set 
\bit 
\itemI $(P_1, Q_{1,{\cal C}_j}, L_{1,j}) := K_1(P, Q_{{\cal C}_j}, L_j)$, $\sigma^\flat := K_1(\sigma^\ast)$ for $\sigma \in {\cal S}$ \ and \ 
${\cal D}^\flat := \{ \sigma^\flat \mid \sigma \in {\cal D} \}$ for ${\cal D} \subset {\cal S}$. 
\eit 
The choice of $K$ in (Step 1) means that 
\bit 
\itemII 
\bit 
\itema $\sigma^\flat \cap \tau \neq \emptyset$ iff $\sigma = \tau$ \ ($\sigma, \tau \in {\cal S}$), hence 
$\Lambda_j = Q_1 - |{\cal E}_j^\flat| \supset Q_1 - |{\cal C}_j^\flat| = K_1(Q - |{\cal C}_j^\ast|) \equiv Q_{1,{\cal C}_j}$.
\itemb $P_1$ differs from $P$ only in some $U \in {\cal B}_f({\rm Int}\,N)$, which is a small neighborhood of $P \cap Q \cap {\rm Int}\,N$ in ${\rm Int}\,N$, 
and there exists $\eta \in {\rm Diff}(M, M_U)_0$ with $P_1 = \eta(P)$. 
\eit 
By \cite[Remark 2.1]{Tsu09} we may assume that $\eta \in {\rm Diff}(M, M_U)_0^c$. 
\eit 
Since the tuple $(N_2, N_1, P, Q_{{\cal C}_j}, L_j)$ corresponds to $(N_2, N_1, P_1, Q_{1, {\cal C}_j}, L_{1,j})$ under the diffeomorphism $K_1$, 
from (2) we have the following conclusions :  

\bit 
\item[(i)$'$\,] any $\underline{g} \in {\rm Diff}^r(M, M_{N_1} \cup Q_{1, {\cal C}_j})_0$ $(j = 1, \cdots, k)$ \ 
has a factorization \ $\underline{g} = \underline{g}_1\underline{g}_2$ \ such that \\[1mm] 
\hsp \btab[t]{@{}l@{ \ }l@{ \ }l}
$\underline{g}_1 \in {\rm Diff}^r(M, M_{N_1} \cup Q_{1,{\cal C}_j})_0^c$ & and & 
$clb^f(\underline{g}_1) \leq 2\ell_j$ in ${\rm Diff}^r(M, M_{N_1} \cup Q_{1,{\cal C}_j})_0$, \\[2mm] 
$\underline{g}_2 \in {\rm Diff}^r(M, M_{N_1} \cup Q_{1,{\cal C}_j} \cup L_{1,j})_0^c$ & and & 
$clb^f(\underline{g}_2) \leq 1$ in ${\rm Diff}^r(M, M_{N_1} \cup Q_{1,{\cal C}_j} \cup L_{1,j})_0$, 
\etab 
\vskip 2mm 
\item[(ii)$'$] $cld\,{\rm Diff}^r(M, M_{N_2} \cup P_1)_0 \leq 2$ \ \ and \ \ 
$clb^f\!d\,{\rm Diff}^r(M, M_{N_2} \cup P_1)_0 \leq 2\ell+1$. 
\eit 
\vskip 1mm 

\item It remains to estimate the contribution of ${g'}^{(j)}$, ${h'}^{(k)}$ and $b'_j$ to $cl\,f$ and $clb^f f$ in ${\rm Diff}^r(M, M_{N_2})_0$.
\bit 
\itemi (a) \,${g'}^{(j)} \in {\rm Diff}^r(M, M_{N_1} \cup \Lambda_j)_0 \subset {\rm Diff}^r(M, M_{N_1} \cup Q_{1,{\cal C}_j})_0$ 
by (3)(ii)(a), and \\ 
\hspace*{5.5mm} $cl\,{g'}^{(j)} \leq 2$, \ $clb^f {g'}^{(j)} \leq 2\ell_j+1$ by (3)(i)$'$. 
\item[] (b) \,$cl\,{h'}^{(k)} \leq 2$, \ $clb^f{h'}^{(k)} \leq 2\ell+1$ by (2)(ii). 
\item[] (c) \,$cl\,b'_j \leq 1$, \ $clb^f b'_j \leq 1$, \ so that \ $cl\,\big(\prod_{j=-2}^k b'_j\big) \leq k+3$, \ 
$clb^f\big(\prod_{j=-2}^k b'_j \big)\leq k+3$.
\eit 
\vskip 2mm 
{Case [I] :} From (i) it follows that \ $cl\,f \leq (k+3) + 2k + 2 = 3k+5$ \ and \\[1mm] 
\hspace*{20mm} 
$\bary[t]{l@{ \ }l}
clb^f f & \leq (k+3) + \mbox{\small $\ds \sum_{j=1}^k$}\,(2\ell_j+1) + 2\ell+1
= 2\Big(\mbox{\small $\ds \sum_{j=1}^k$}\, \ell_j + \ell + k + 2\Big) \\[5mm]
& \leq 2\Big(c({\cal H}) + (k-1)c({\cal H}^{(m)}) + k + 2\Big) \leq 2(m+2)(k+1) = (n+4)(k+1).
\eary$ \\[2mm] 
{Case [I\hspace{-0.2mm}I] :}  We refine the factorization of $f$ in (1)(iv). 
\bit 
\itemii By (3)(i)$'$ 
each ${g'}^{(j)} \in {\rm Diff}^r(M, M_{N_1} \cup Q_{1,{\cal C}_j})_0$ has a factorization \ ${g'}^{(j)} = g_jh_j$ \\ such that 
\btab[t]{lll}
$g_j \in {\rm Diff}^r(M, M_{N_1} \cup Q_{1,{\cal C}_j})_0^c$ & and & 
$clb^f(g_j) \leq 2\ell_j$ in ${\rm Diff}^r(M, M_{N_1} \cup Q_{1,{\cal C}_j})_0$, \\[2mm] 
$h_j \in {\rm Diff}^r(M, M_{N_1} \cup Q_{1,{\cal C}_j} \cup P_1)_0^c$ & and & 
$clb^f(h_j) \leq 1$ in ${\rm Diff}^r(M, M_{N_1} \cup Q_{1,{\cal C}_j} \cup P_1)_0$. 
\etab 
\vskip 2mm 
Let $g_1' = g_1$, $g_j' := g_j^{h_1 \cdots h_{j-1}}$ $(j=2, \cdots, k)$ and $h' := h_1 \cdots h_k$. 
Then, it follows that 
\vskip 1mm 
\bit 
\itema $g_j' \in {\rm Diff}^r(M, M_{N_1})_0^c$ \ $(j=1, \cdots, k)$ \ and \ $h' \in {\rm Diff}^r(M, M_{N_1} \cup P_1)_0 \subset {\rm Diff}^r(M, M_{N_2} \cup P_1)_0$, 
\vskip 1mm 
\itemb $\prod_{j=1}^k g'{}^{(j)} = \prod_{j=1}^k (g_jh_j) = g_1 g_2^{h_1} g_3^{h_1h_2} \cdots g_k^{h_1 \cdots h_{k-1}} (h_1 \cdots h_k)
= \big(\prod_{j=1}^k g_j'\big) h'$. 
\vskip 1mm 
\itemc $cl\,g_j' \leq 1$, \ $clb^f g_j' \leq 2\ell_j$ \ \ and \ \ $cl\,h' \leq 2$, \ $clb^f h' \leq 2\ell+1$ \ by (3)(ii)$'$
\itemd $cl\,(h'{h'}^{(k)}) \leq cl\,h' + cl\,{h'}^{(k)} \leq 2+2 = 4$ \ by (i)(b) and (c). 
\eit 
\vskip 1mm 
\itemiii To obtain a finer estimate on $clb^f (h'{h'}^{(k)})$
we decompose the composition $h'{h'}^{(k)}$ using $\eta \in {\rm Diff}(M, M_U)_0^c$ given in (3)(ii)(b). 
We note that 
\bit 
\itema $\eta' :=(\eta^{-1})^{h'} \in {\rm Diff}(M, M_{U'})_0^c$ \ for $U' := h'(U) \in {\cal B}_f({\rm Int}\,N_1)$. 
\eit 
Since $\eta(P) = P_1$, we have $({h'}^{(k)})^\eta \in {\rm Diff}^r(M, M_{N_1} \cup P_1)_0$, which implies that 
\bit 
\itemb $h : = h'({h'}^{(k)})^\eta \in {\rm Diff}^r(M, M_{N_1} \cup P_1)_0$ \ and \ $cl\,h \leq 2$, \ $clb^f h \leq 2\ell+1$ by (3)(ii)$'$. 
\eit 
Since $(h')^\eta = \eta\eta' h'$, 
it follows that $(h'{h'}^{(k)})^\eta = (h')^\eta({h'}^{(k)})^\eta = \eta \eta' h$ and  
\bit 
\itemc $clb^f (h'{h'}^{(k)}) = clb^f (\eta \eta' h) \leq 1 + 1 + 2\ell+1 = 2\ell+3$. 
\eit 

\itemiv From (ii) and (iii) it follows that 
$cl\,f \leq (k+3) + k + 4 = 2k+7$ \ and \\[1mm] 
$clb^f f \leq (k+3) + \mbox{\small $\ds \sum_{j=1}^k$}\,(2\ell_j) + 2\ell+3
= 2\,\mbox{\small $\ds \sum_{j=1}^k$}\, \ell_j  + 2\ell + k + 6 \\
\leq 2\,\Big(c({\cal H}) + (k-1)c({\cal H}^{(m)})\Big)  + k + 6
\leq (2m+3)(k+1) + 3 = (n+3)(k+1)+3.$  
\eit 
\eenum 

This completes the proof. 
\end{proof} 

When $M$ is a closed $2m$-manifold, we can select $N = N_1 = N_2 = M$. 
Then, Proposition~\ref{cl_even_cpt_bd_handle} reduces to the following form. 

\begin{cor}\label{cor_closed} 
Suppose $M$ is a closed $2m$-manifold $(m \geq 3)$, $1 \leq r \leq \infty$, $r \neq 2m+1$, 
${\cal H}$ is a handle decomposition of $M$, 
${\cal C}$ is the set of all open $m$-cells of $P_{\cal H}$ and $\{ {\cal C}_j \}_{j=1}^k$ is a finite cover of ${\cal C}$. 

\benum 
\item[{\rm [I]}\,] $cld\,{\rm Diff}^r(M)_0 \leq 3k+5$ \ \ and \ \  \\[1mm] 
$clb^f\!d\,{\rm Diff}^r(M)_0 \leq 
2\Big(c({\cal H}) + (k-1)c({\cal H}^{(m)}) + k + 2\Big) \leq 2(m+2)(k+1)$ \\[1mm] 
\hspace*{5mm} if each $Cl_M|{\cal C}_j|$ $(j=1, \dots, k)$ is strongly displaceable from itself in $M$. 
\vskip 2mm 

\item[{\rm [II]}] $cld\,{\rm Diff}^r(M)_0 \leq 2k+7$ \ \ and \\[1mm] 
$clb^f\!d\,{\rm Diff}^r(M)_0  
\leq 2\,\Big(c({\cal H}) + (k-1)c({\cal H}^{(m)})\Big)  + k + 6 \leq (2m+3)(k+1) + 3$ \\[1mm] 
\hspace*{5mm} if each $Cl_M|{\cal C}_j|$ $(j=1, \dots, k)$ is strongly displaceable from $Cl_M|{\cal C}|$ in $M$. 
\eenum 
\end{cor} 

\begin{example}\label{exp_even_covering} 
Consider the product of two $m$-spheres $M = S^m \times S^m$. 
Each factor $S^m_k$ $(k=1,2)$ has a handle decompostion with one $0$-handle and one $m$-handle. 
By taking the products of these handles and smoothing them, 
we obtain a handle decompostion ${\cal H}$ of $M$ with one 0-handle, two $m$-handles and one $2m$-handle. 
The core complex $P_{\cal H}$ of ${\cal H}$ has two open $m$-cells $\sigma_j$ $(j=1,2)$, 
for which we may assume that $Cl_M \sigma_j$ $(j = 1, 2)$ are smooth $m$-spheres with a trivial normal bundle and 
these intersect at the unique 0-cell (i.e., the center point of the 0-handle) transversely. 
Hence, each $Cl_M\sigma_j$ $(j = 1, 2)$ is strongly displaceable from itself in $M$ (cf. Example~\ref{exp_displacement}\,(3)), 
but it is not displaceable from $Cl_M(\sigma_1 \cup \sigma_2)$. 
Therefore, if $m \geq 3$ and $1 \leq r \leq \infty$, $r \neq 2m+1$, from Corollary~\ref{cor_closed}\,[I] it follows that \\
\hspace*{45mm} $cld\,{\rm Diff}^r(M)_0 \leq 11$ \ \ and \ \ $clb^f\!d\,{\rm Diff}^r(M)_0 \leq 18$. 
\end{example}

\subsection{Open manifold case} \mbox{} 

This section is devoted to the estimates of commutator length diameter of diffeomorphism groups for some classes of open manifolds. 

\subsubsection{\bf Grouping of $m$-cells in open $2m$-manifolds} \mbox{} 

In this subsection we obtain basic results related to grouping of infinitely many $m$-cells in open $2m$-manifolds 
based on Theorem~\ref{thm_cpt_bdry_even} and Proposition~\ref{cl_even_cpt_bd_handle} in the compact $2m$-manifold case. 
These results are used in the next subsection to treat some important classes of open $2m$-manifolds 
(for example, covering spaces, infinite connected sums, etc.).  

First we consider the case related to exhausting sequences and triangulations. 

\begin{setting}\label{setting_exhausting_seq}
Suppose $M$ is an open $2m$-manifold, 
$\{ M_k \}_{k \geq 0}$ is an exhausting sequence of $M$ $(M_k = \emptyset$ $(k < 0))$ and  
${\cal T}$ is a $C^\infty$ triangulation of $M$ such that each $M_k$ underlies a subcomplex of ${\cal T}$.  
Let ${\cal S}$ and ${\cal S}_{k}$ $(k \geq 0)$ denote 
the sets of $m$-simplices of ${\cal T}$ and ${\cal T}|_{M_{k-1,k}}$ respectively. 
Suppose $p \geq 1$, $q \geq 0$ and 
$\{{\cal S}_{0,j} \}_{j=1}^{p+q}$ is a finite cover of ${\cal S}_0$ and 
$\{{\cal S}_{k,j} \}_{j=1}^{p}$ is a finite cover of ${\cal S}_k$ for each $k \geq 1$ such that 
\bit 
\item[$(\natural)$]
$\{ |{\cal S}_{k,j}| \}_{k \geq 0}$ is a disjoint family for each $j=1, \cdots, p$. 
\eit 
\end{setting}

Recall that $\widetilde{N} := N \cup_{\partial N} (\partial N \times [0,1))$ for any manifold $N$. 

\begin{prop}\label{prop_exhausting seq} 
In Setting~{\rm \ref{setting_exhausting_seq}} :  Suppose $m \geq 3$ and $1 \leq r \leq \infty$, $r \neq 2m+1$. 
\benum[\ (1)] 
\item[{\rm [I]}\,] 
\btab[t]{@{}ll}
$cld\,{\rm Diff}^r(M)_0 \leq 3(2p+q) + 10$, \hsf & $clb^dd\,{\rm Diff}^r(M)_0 \leq 2(m+2)(2p+q+2)$, \\[2mm] 
$cld\,{\rm Diff}^r_c(M)_0 \leq 3(p+q) + 5$, & $clb^fd\,{\rm Diff}_c^r(M)_0 \leq 2(m+2)(p+q+1)$,
\etab 
\vskip 1mm 
\bit 
\item[if \ {\rm (a)}] each $|{\cal S}_{0,j}|$ $(j=1, \cdots, p+q)$ is strongly displaceable from $|{\cal S}_{0,j}| \cup (\partial M_0 \times [0,1))$ in $\widetilde{M}_0$ 
and 
\itemb each $|{\cal S}_{k,j}|$ $(k \geq 1, j=1, \cdots, p)$ is strongly displaceable from $|{\cal S}_{k,j}| \cup (\partial M_{k-1, k} \times [0,1))$ in $\widetilde{M}_{k-1, k}$. 
\eit   

\item[{\rm [II]}] 
\btab[t]{@{}ll}
$cld\,{\rm Diff}^r(M)_0 \leq 2(2p+q) + 14$, \hsf &  
$clb^dd\,{\rm Diff}^r(M)_0 \leq (2m+3)(2p+q+2) + 6$, \\[2mm] 
$cld\,{\rm Diff}_c^r(M)_0 \leq 2(p+q) +7$, & 
$clb^fd\,{\rm Diff}_c^r(M)_0 \leq (2m+3)(p+q+1) + 3$, 
\etab 
\vskip 1mm 
\bit 
\item[if \ {\rm (a)}] each $|{\cal S}_{0,j}|$ $(j=1, \cdots, p+q)$ is strongly displaceable from 
$|{\cal S}_{0}| \cup (\partial M_0 \times [0,1))$ in $\widetilde{M}_0$ 
and 
\itemb each $|{\cal S}_{k,j}|$ $(k \geq 1, j=1, \cdots, p)$ is strongly displaceable from $|{\cal S}_{k}| \cup (\partial M_{k-1, k} \times [0,1))$ in $\widetilde{M}_{k-1, k}$. 
\eit   
\eenum 
\end{prop} 

\begin{proof} 
We consider the two cases \ $(\dagger)$ $q = q' = cl$ \ and \ $(\ddagger)$ $q = clb^f$, $q'=clb^d$ \ simultaneously. 
We set $a, b \in \IZ_{\geq 1}$ as follows; \\[2mm] 
\hsp 
\btab[t]{c@{ \ }ll}
(i) & in the case $(\dagger)$ : \hsf & 
$(a, b) := 
\left\{ 
\bary[c]{@{\,}l@{\ \,}ll}
(3(p+q) + 5, & 3p + 5) & \text{in \,[I]} \\[2mm] 
(2(p+q) + 7, & 2p + 7) & \text{in [II]}
\eary \right.$ \\[6mm] 
(ii) & in the case $(\ddagger)$ : \hsf & 
$(a,b) := 
\left\{ 
\bary[c]{@{\,}l@{}l@{\ \,}ll}
(2(m+2)(p+q+1) &, & 2(m+2)(p+1)) & \text{in \,[I]} \\[2mm] 
((2m+3)(p+q+1) + 3&, & (2m+3)(p+1) + 3) & \text{in [II]}
\eary \right.$
\etab \\[3mm] 
Our goal is to show that \hsh $(\flat)$\, $q\hspace{0.2mm}d\,{\rm Diff}_c^r(M)_0 \leq a$ \ \ and \ \ 
$q'\hspace{0.2mm}d\,{\rm Diff}^r(M)_0 \leq a+b$ \hsh in each case. 

Below, we show the following claim :  
\vskip 1mm 
\hsp $(\ast)$ \ 
$q\hspace{0.2mm}d\,{\rm Diff}^r(M, M_{M_k})_0 \leq a$ 
\hsh 
$(\ast\ast)$ \ $q\hspace{0.2mm}d\,{\rm Diff}^r(M, M_{M_{k, \ell}})_0 \leq b$ \hsh for any $0 \leq k < \ell < \infty$. \\
This claim and Lemma~\ref{lemma_subseq} imply the assertion $(\flat)$. 

The estimates $(\ast)$ and $(\ast\ast)$ 
are deduced by Theorem~\ref{thm_cpt_bdry_even}. 
\benum[(1)] 
\item 
Let ${\cal S}_{M_k}$ and ${\cal S}_{M_{k,\ell}}$ denote 
the sets of $m$-simplices of ${\cal T}|_{M_k}$ and ${\cal T}|_{M_{k, \ell}}$ respectively. 
For each $k \geq 0$ the set 
${\cal S}_{M_k} = \bigcup_{i=0}^k {\cal S}_i$ has the cover $\{ {\cal F}_j^k \}_{j=1}^{p+q}$ defined by 
\bit 
\itemI ${\cal F}_j^k := \bigcup_{i=0}^k {\cal S}_{i, j}$ $(j=1, \cdots, p)$, \hsf 
${\cal F}_j^k := {\cal S}_{0,j}$ $(j=p+1, \cdots, p+q)$. 
\eit 
For each $0 \leq k < \ell < \infty$ 
the set ${\cal S}_{M_{k, \ell}} = \bigcup_{i=k+1}^\ell {\cal S}_i$ 
has the covering $\{ {\cal G}_j^{k, \ell} \}_{j=1}^{p}$ defined by
\bit 
\itemII ${\cal G}_j^{k, \ell} := \bigcup_{i=k+1}^\ell {\cal S}_{i, j}$ \ $(j=1, \cdots, p)$. 
\eit 

\item The estimate $(\ast)$ is obtained 
by applying Theorem~\ref{thm_cpt_bdry_even} to 
$M_k$, ${\cal T}|_{M_k}$, ${\cal S}_{M_k}$ and $\{ {\cal F}_j^k \}_{j=1}^{p+q}$. \\
In each case of [I] and [II], by the assumption $(\natural)$ it is easily seen that \\
\hsp each $|{\cal F}_j^k|$ $(j=1, \dots, p+q)$ is strongly displaceable from \\
\hspp $|{\cal F}_j^k| \cup (\partial M_k \times [0,1))$ \ in [I] \ \ and \ \ 
$|{\cal S}_{M_k}| \cup (\partial M_k \times [0,1))$ \ in [II] \ \ in $\widetilde{M}_k$.

\item Similarly, $(\ast\ast)$ is obtained 
by applying Theorem~\ref{thm_cpt_bdry_even} to 
$M_{k, \ell}$, ${\cal T}|_{M_{k, \ell}}$, ${\cal S}_{M_{k, \ell}}$ and $\{ {\cal G}_i^{k, \ell} \}_{i=1}^{p}$. 
\eenum 
\vskip -7mm 
\end{proof}

Next we consider the case concerned with handle decompositions. 

\begin{setting}\label{setting_even_open_handle} 
Suppose $M$ is a $2m$-manifold without boundary $(m \geq 3)$, $1 \leq r \leq \infty$, $r \neq 2m+1$, 
${\cal H}$ is a handle decomposition of $M$, 
$P = P_{\cal H}^{(m)}$ (the $m$-skeleton of the core complex of ${\cal H}$) and 
$Q = P_{\cal H^\ast}^{(m)}$ (the $m$-skeleton of the core complex of ${\cal H}^\ast$). 
\end{setting}

\begin{setting_5.6+}
Suppose ${\cal S}$ is the set of all open $m$-cells of $P$, $\{ {\cal S}_j \}_{j=1}^k$ $(k \geq 1)$ is a finite cover of ${\cal S}$ and 
for each $j = 1, \cdots, k$, 
$\{ {\cal S}_{j, \lambda_j} \}_{\lambda_j \in \Lambda_j}$ is a covering of ${\cal S}_j$ such that 
each ${\cal S}_{j, \lambda_j}$ is a finite set and 
$\{ |{\cal S}_{j, \lambda_j}| \}_{\lambda_j \in \Lambda_j}$ is a locally finite family in $M$. 
For any subset $\Delta \subset \Lambda_j$ let ${\cal S}_{j, \Delta} := \bigcup_{ \lambda_j \in \Delta} {\cal S}_{j, \lambda_j} \subset {\cal S}_j$. 
\end{setting_5.6+}

\begin{thm}\label{thm_even_open_handle} 
In Setting~{\rm \ref{setting_even_open_handle}, ~ \ref{setting_even_open_handle}$^+$} : 
\benum 
\item[{\rm [I]}\,] If $Cl_M|{\cal S}_{j, \Delta}|$ is strongly displaceable from itself in $M$   
for each $j=1, \dots, k$ and any finite subset $\Delta \subset \Lambda_j$, then 
\bit 
\itemI $cld\,{\rm Diff}_c^r(M)_0 \leq 3k+5$, \\[1mm] 
$clb^fd\,{\rm Diff}_c^r(M)_0 \leq 2\big(c({\cal H}) + (k-1)c({\cal H}^{(m)}) + k + 2\big) \leq 2(m+2)(k+1)$, 
\itemII 
$cld\,{\rm Diff}^r(M)_0 \leq 6k+10$, \\[1mm] 
$clb^dd\,{\rm Diff}^r(M)_0 \leq 4\big(c({\cal H}) + (k-1)c({\cal H}^{(m)}) + k + 2\big) \leq 4(m+2)(k+1)$. 
\eit 

\vskip 2mm 
\item[{\rm [II]}] If $Cl_M|{\cal S}_{j, \Delta}|$ is strongly displaceable from $Cl_M|{\cal S}|$ in $M$ 
 for each $j=1, \dots, k$ and any finite subset $\Delta \subset \Lambda_j$, then 
\bit 
\itemI 
$cld\,{\rm Diff}^r_c(M)_0 \leq 2k+7$, \\[1mm] 
$clb^fd\,{\rm Diff}^r_c(M)_0 \leq 2\,\big(c({\cal H}) + (k-1)c({\cal H}^{(m)})\big)  + k + 6 \leq (2m+3)(k+1) + 3$, 
\itemII 
$cld\,{\rm Diff}^r(M)_0 \leq 4k+14$, \\[1mm] 
$clb^dd\,{\rm Diff}^r(M)_0 \leq 4\big(c({\cal H}) + (k-1)c({\cal H}^{(m)}) + 2k + 12 \leq 2(2m+3)(k+1) + 6$. 
\eit 
\eenum 
\end{thm}

\begin{proof} 
To treat various cases simultaneously, we use the following notations: 

Let $(\dagger)$ $q = q' = cl$ \ or \ $(\ddagger)$ $q = clb^f$, $q' = clb^d$. \ We set $a \in \IZ_{\geq 1}$ as follows;  \\[2mm] 
\hsp 
\btab[t]{c@{ \ }ll}
(i) & in the case $(\dagger)$ : & 
$a := 
\left\{ 
\bary[c]{@{\,}l@{\ \ }l}
3k+5 & \text{in \,[I]}, \\[2mm] 
2k+7 \hsf  & \text{in [II]}. 
\eary \right.$ \\[6mm] 
(ii) & in the case $(\ddagger)$ : \hsf & 
$a := 
\left\{ 
\bary[c]{@{\,}l@{\ \ }l}
2\big(c({\cal H}) + (k-1)c({\cal H}^{(m)}) + k + 2\big) & \text{in \,[I]}, \\[2mm] 
2\,\big(c({\cal H}) + (k-1)c({\cal H}^{(m)})\big)  + k + 6 \hsf & \text{in [II]}.
\eary \right.$ 
\etab 
\vskip 2mm 

\benum[(1)]
\item First we show the following claim based on Proposition~\ref{cl_even_cpt_bd_handle}. \\
{\bf Claim.} If $N \in {\cal S}{\cal M}_c(M)$, $N_1 \in {\cal S}{\cal M}_c(M, {\cal H})$, 
$N_2 \in {\cal S}{\cal M}_c(M, {\cal H}^\ast)$, $N \Subset N_1 \Subset N_2$ and \\ 
\hspace*{13mm} ${\rm Int}\,N_1$ includes any $|{\cal S}_{j, \lambda_j}|$ $(j=1, \cdots, k, \lambda_j \in \Lambda_j)$ 
with $|{\cal S}_{j, \lambda_j}| \cap {\rm Int}\,N \neq \emptyset$, then \\ 
\hspace*{25mm} 
$q\hspace{0.2mm}d({\rm Diff}^r(M, M_{N})_0, {\rm Diff}^r(M, M_{N_2})_0) \leq a.$ 

\begin{proof}[Proof of Claim] 
We apply Setting 5.4, 5.4$^+$ and Proposition~\ref{cl_even_cpt_bd_handle}.
\bit 
\itemI For Setting 5.4 :  If $\sigma \in {\cal S}$ and $\sigma \cap {\rm Int}\,N \neq \emptyset$, then $\sigma \subset {\rm Int}\,N_1$. 
In fact, $\sigma \in {\cal S}_{j,\lambda_j}$ for some $j =1, \cdots, k$ and $\lambda_j \in \Lambda_j$ and 
$|{\cal S}_{j, \lambda_j}| \cap {\rm Int}\,N \neq \emptyset$. 
Hence, the assumption on $N_1$ implies ${\rm Int}\,N_1 \supset |{\cal S}_{j, \lambda_j}| \supset \sigma$. 

\itemII For Setting 5.4$^+$ : For each $j = 1, \cdots, k$ \ we set \\
\hspace*{10mm} $\Delta_j := \{ \lambda_j \in \Lambda_j \mid |{\cal S}_{j, \lambda_j}| \cap {\rm Int}\,N \neq \emptyset \}$ \ \ and \ \ 
${\cal C}_j := {\cal S}_{j, \Delta_j} \equiv \bigcup_{ \lambda_j \in \Delta_j} {\cal S}_{j, \lambda_j} \subset {\cal S}_j$. \\
Then, (a) $\Delta_j$ is a finite set since $\{ |{\cal S}_{j, \lambda_j}| \}_{\lambda_j \in \Lambda_j}$ is a locally finite family, 
(b) $|{\cal C}_j| \subset {\rm Int}\,N_1$ by the assumption on $N_1$ 
and (c) ${\cal C} := \bigcup_{j=1}^k {\cal C}_j$ satisfies the condition : \\
\hspace*{20mm} $\{ \sigma \in {\cal S} \mid \sigma \cap {\rm Int}\,N \neq \emptyset \} \subset {\cal C} 
\subset \{ \sigma \in {\cal S} \mid \sigma \subset {\rm Int}\,N_1 \}.$ 

\itemiii In the case [I] each $Cl_M|{\cal C}_j| = Cl_M|{\cal S}_{j, \Delta_j}|$ $(j=1, \dots, k)$ is strongly displaceable from itself in $M$. 
Hence, the conclusion follows from Proposition~\ref{cl_even_cpt_bd_handle}\,[I]. 

In the case [II] each $Cl_M|{\cal C}_j| = Cl_M|{\cal S}_{j, \Delta_j}|$ $(j=1, \dots, k)$ is strongly displaceable from $Cl_M|{\cal S}|$ in $M$. 
Since $Cl_M|{\cal C}| \subset Cl_M|{\cal S}|$, the conclusion follows from Proposition~\ref{cl_even_cpt_bd_handle}\,[II]. 
\eit 
\vskip -6.5mm 
\end{proof}
In Claim we can replace $q$ by $q'$ since $N_2$ is compact.  

\item[(2)] The estimates for ${\rm Diff}_c^r(M)_0$ in [I](i) and [II](i) are obtained as follows. 

We have to show that $q\hspace{0.2mm}d\,{\rm Diff}^r_c(M)_0 \leq a$. 
Take any $f \in {\rm Diff}^r_c(M)_0$ and $F \in {\rm Isot}^r_c(M)_0$ with $f = F_1$. 
Since ${\rm supp}\,F$ is compact and 
${\cal F} = \{ Cl_M |{\cal S}_{j, \lambda_j}| \mid j=1, \cdots, k, \lambda_j \in \Lambda_j \}$ is a locally finite family of compact subsets of $M$, 
we can find a triple $N, N_1, N_2$ as in Claim  with ${\rm supp}\,F \Subset N$. 
From Claim it follows that \hsp $q\hspace{0.2mm}d({\rm Diff}^r(M, M_{N})_0, {\rm Diff}^r(M, M_{N_2})_0) \leq a.$ \\
Since $f \in {\rm Diff}^r(M, M_{N})_0$, it follows that $q(f) \leq a$ in ${\rm Diff}^r(M, M_{N_2})_0$ and 
hence $q(f) \leq a$ in ${\rm Diff}_c^r(M)_0$. 

\item[(3)] For the estimates for ${\rm Diff}^r(M)_0$ in [I](ii) and [II](ii), we have to show that \\ 
\hspp $q'\hspace{0.2mm}d\,{\rm Diff}^r(M)_0 \leq 2a$, \ \ when $M$ is noncompact. 
\vskip 1mm 

Take any $f \in {\rm Diff}^r(M)_0$ and $F \in {\rm Isot}^r(M)_0$ with $f = F_1$. 
By Lemma~\ref{lem_factorization_open-2}\,[I] 
there exists an exhausting sequence $\{ M_k \}_{k \geq 1}$ of $M$ which satisfies the following conditions : 

\bit 
\itemI $M_k \in {\mathcal S}_c(M)$ $(k \in \IZ)$  and  
$F(M_{4k,4k+1} \times I) \Subset M_{4k-1, 4k+2} \times I$ $(k \geq 0)$.  

\itemII There exist $N_k^1  \in {\cal S}{\cal M}_c(M; {\cal H})$ and $N_k^2 \in {\cal S}{\cal M}_c(M; {\cal H}^\ast)$ $(k \geq 0)$ 
such that for each $k \geq 0$ 
\bit 
\itema $M_{4k-1, 4k+2} \Subset N_k^1 \Subset N_k^2 \ \Subset \ M_{4k-2, 4k+3}$, \ \ (b) $N_k^2 \cap N_{k+1}^2 = \emptyset$ \ \ and 
\itemc ${\rm Int}\,N_k^1$ includes any $|{\cal S}_{j, \lambda_j}|$ $(j=1, \cdots, k, \lambda_j \in \Lambda_j)$ with 
$|{\cal S}_{j, \lambda_j}| \cap {\rm Int}\,M_{4k-1, 4k+2} \neq \emptyset$. 
\eit 

\itemiii There exist $L_k^1  \in {\cal S}{\cal M}_c(M; {\cal H})$ and $L_k^2 \in {\cal S}{\cal M}_c(M; {\cal H}^\ast)$ $(k \geq 1)$ such that for each $k \geq 1$ 
\bit 
\itema $M_{4k-3, 4k} \Subset L_k^1 \Subset L_k^2 \ \Subset \ M_{4k-4, 4k+1}$, \ \ (b) $L_k^2 \cap L_{k+1}^2 = \emptyset$ \ \ and
\itemc ${\rm Int}\,L_k^1$ includes any $|{\cal S}_{j, \lambda_j}|$ $(j=1, \cdots, k, \lambda_j \in \Lambda_j)$ with 
$|{\cal S}_{j, \lambda_j}| \cap {\rm Int}\,M_{4k-3, 4k} \neq \emptyset$
\eit 
\eit 
\vskip 1mm 
The conditions (ii)(c) and (iii)(c) are achieved by taking the locally finite family 
${\cal F} = \{ Cl_M |{\cal S}_{j, \lambda_j}| \mid j=1, \cdots, k, \lambda_j \in \Lambda_j \}$. 
From Claim and (ii), (iii) it follows that 
\bit 
\itemiv $q'\hspace{0.2mm}d({\rm Diff}^r(M, M_{M_{4k-1, 4k+2}})_0, {\rm Diff}^r(M, M_{N_k^2})_0) \leq a$ \ \ $(k \geq 0)$ and \\ 
$q'\hspace{0.2mm}d({\rm Diff}^r(M, M_{M_{4k-3, 4k}})_0, {\rm Diff}^r(M, M_{L_k^2})_0) \leq a$ \ \ $(k \geq 1)$. 
\eit 
\vskip 2mm 
Lemma~\ref{lem_factorization_open-2}\,[II]  we have a factorization 
\bit 
\itemv $F = GH$ \ \ for some $G \in {\rm Isot}^r(M, M')_0$ \ and \ $H \in {\rm Isot}^r(M, M'')_0$, \\
\hsh where \ $M' := \bigcup_{k \geq 0} M_{4k+2, 4k+3}$ \ and \ $M'' := \bigcup_{k \geq 0} M_{4k, 4k+1}$. 
\eit 
By (iv) it follows that \ in ${\rm Diff}^r(M)_0$
\bit 
\item[(vi)] $q'(G_1) \leq a$, $q'(H_1) \leq a$ \ and \ hence $q'(f) \leq 2a$ \ as required. 
\eit 
\eenum 
 
This completes the proof. 
\end{proof}

\subsubsection{\bf Some classes of open $2m$-manifolds $M$ with $cld\,{\rm Diff}^r(M)_0 < \infty$} \mbox{} 

In this subsection we show the uniform perfectness of diffeomorphism groups for some important classes of open $2m$-manifolds, 
including covering spaces and infinite connected sums. 
\vskip 2mm 

\noindent {\bf [1] Open $2m$-manifolds with finitely many $m$-handles} \mbox{} 

\begin{prop}\label{prop_2m_open_no-m-h}
Suppose $M$ is a $2m$-manifold without boundary and $1 \leq r \leq \infty$, $r \neq 2m+1$. 
If $M$ admits a handle decomposition ${\cal H}$ without $m$-handles, then 
\bit 
\itemI $cld\,{\rm Diff}^r(M)_0 \leq 6$, \hsh $clb^dd\,{\rm Diff}^r(M)_0 \leq 4c({\cal H})+2 \leq 8m+2$, 
\itemII $cld\,{\rm Diff}_c^r(M)_0 \leq 3$, \hsh $clb^fd\,{\rm Diff}_c^r(M)_0 \leq 2c({\cal H})+1 \leq 4m+1$. 
\eit 
\end{prop}

\begin{prop}\label{prop_2m_open_finite-m-h}
Suppose $M$ is a $2m$-manifold without boundary $(m \geq 3)$, $1 \leq r \leq \infty$, $r \neq 2m+1$, 
${\cal H}$ is a handle decomposition of $M$ with only finitely many $m$-handles, 
${\cal S}$ is the set of all open $m$-cells of $P_{\cal H}$ and $\{ {\cal S}_j \}_{j=1}^k$ $(k \geq 1)$ is a finite cover of ${\cal S}$. 

\benum 
\item[{\rm [I]}\,] If $Cl_M|{\cal S}_{j}|$ is strongly displaceable from itself in $M$ for each $j=1, \dots, k$, then 
\bit 
\itemI 
$cld\,{\rm Diff}^r(M)_0 \leq 3k+8$, \\
$clb^dd\,{\rm Diff}^r(M)_0 \leq 2\,\Big(2c({\cal H}) + (k-1)c({\cal H}^{(m)})\Big)  + 2k + 3 \leq 2(m+2)(k+3) - 7$,  
\itemII 
$cld\,{\rm Diff}^r_c(M)_0 \leq 3k+5$, \\
$clb^fd\,{\rm Diff}^r_c(M)_0 \leq 2\,\Big(c({\cal H}) + (k-1)c({\cal H}^{(m)})\Big)  + 2k + 4 \leq 2(m+2)(k+1)$. 
\eit 

\item[{\rm [I\hspace{-0.1mm}I]}] 
If $Cl_M|{\cal S}_{j}|$ is strongly displaceable from $Cl_M|{\cal S}|$ in $M$ for each $j=1, \dots, k$, then 
\bit 
\itemI $cld\,{\rm Diff}^r(M)_0 \leq 2k+10$, \\
$clb^dd\,{\rm Diff}^r(M)_0 \leq 2\,\Big(2c({\cal H}) + (k-1)c({\cal H}^{(m)})\Big)  + k + 5 \leq (2m+3)(k+3)-2$, 
\itemII
$cld\,{\rm Diff}^r_c(M)_0 \leq 2k+7$, \\
$clb^fd\,{\rm Diff}^r_c(M)_0 \leq 2\,\Big(c({\cal H}) + (k-1)c({\cal H}^{(m)})\Big)  + k + 6 \leq (2m+3)(k+1) +3$. 
\eit 
\eenum 
\end{prop}

\begin{proof}[\bf Proof of Propositions~\ref{prop_2m_open_no-m-h}, ~\ref{prop_2m_open_finite-m-h}] \mbox{} 

The closed manifold cases are included in Proposition~\ref{prop_2m_cpt_no-m-h} and Corollary~\ref{cor_closed}. 
Below we assume that $M$ is an open manifold. 
We only verify the case of ${\rm Diff}^r(M)_0$. The case of ${\rm Diff}_c^r(M)_0$ follows from 
simpler versions of the arguments (2) and (3) below (cf. Proof (2) of Theorem~\ref{thm_even_open_handle}). 

Take any $f \in {\rm Diff}^r(M)_0$ and $F \in {\rm Isot}^r(M)_0$ with $F_1 = f$. 

\benum 
\item Let $K :=$ the union of all $m$-handles in ${\cal H}$. 
By Lemma~\ref{lem_factorization_open-2}\,[I] there exists an exhausting sequence 
$\{ M_i \}_{i \geq 1}$ of $M$ which satisfies the following conditions : 

\bit 
\itemI (a) $K \Subset M_1$, \hsh (b) 
$F(M_{4i,4i+1} \times I) \Subset M_{4i-1, 4i+2} \times I$ $(i \geq 0)$. 

\itemII There exist $N_i'  \in {\cal S}{\cal M}_c(M; {\cal H})$ and $N_i'' \in {\cal S}{\cal M}_c(M; {\cal H}^\ast)$ $(i \geq 0)$ 
such that for each $i \geq 0$
\bit 
\itema $M_{4i-1, 4i+2} \Subset N_i' \Subset N_i'' \Subset M_{4i-2, 4i+3}$ 
\ \ and \ \ {\rm (b)} $N_i'' \cap N_{i+1}'' = \emptyset$. 
\eit 

\itemiii There exist $L_i'  \in {\cal S}{\cal M}_c(M; {\cal H})$ and 
$L_i'' \in {\cal S}{\cal M}_c(M; {\cal H}^\ast)$ $(i \geq 1)$ such that for each $i \geq 1$ 
\bit 
\itema $M_{4i-3, 4i} \Subset L_i' \Subset L_i'' \Subset M_{4i-4, 4i+1}$ 
\ \ and \ \ {\rm (b)} $K \cap L_1'' = \emptyset$, \ $L_i'' \cap L_{i+1}'' = \emptyset$. 
\eit 
\eit 
By Lemma~\ref{lem_factorization_open-2}\,[II] 
there exists a factorization $F = GH$ for some $G \in {\rm Isot}^r(M, M')_0$ and $H \in {\rm Isot}^r(M, M'')_0$, 
where \ $M' := \bigcup_{i \geq 0} M_{4i+2, 4i+3}$ \ and \ $M'' := \bigcup_{i \geq 0} M_{4i, 4i+1}$. 

\item in Proposition~\ref{prop_2m_open_no-m-h} : \\
Let $q = cl$ or $clb^d$ and set $\ell := 3$ for $q = cl$ and $\ell := 2c({\cal H})+1$ for $q = clb^d$. 
Since ${\cal H}$ has no $m$-handles, 
by Proposition~\ref{prop_2m_cpt_no-m-h}, if 
$N \in {\cal S}{\cal M}_c(M; {\cal H})$, $L \in {\cal S}{\cal M}_c(M; {\cal H}^\ast)$ and $N \subset L$, then \\ 
\hspp $q\hspace{0.3mm}d\,({\rm Diff}^r(M, M_N)_0,  {\rm Diff}^r(M, M_L)_0) \leq \ell$. \\ 
Thus, by (1)(ii),\hspace{0.4mm}(iii) we have  

\bit 
\itemI $q\hspace{0.3mm}d\,({\rm Diff}^r(M, M_{N_i'})_0,  {\rm Diff}^r(M, M_{N_i''})_0) \leq \ell$ \ $(i \geq 0)$ \ \ and 
\itemII $q\hspace{0.3mm}d\,({\rm Diff}^r(M, M_{L_i'})_0,  {\rm Diff}^r(M, M_{L_i''})_0) \leq \ell$ \ $(i \geq 1)$.  
\eit 
Therefore, from Lemma~\ref{lem_factorization_open-2}\,[II] 
it follows that \ $q(G_1), q(H_1) \leq \ell$ \ and \ $q(f) \leq 2 \ell$ \ in ${\rm Diff}^r(M)_0$. 

\item in Proposition~\ref{prop_2m_open_finite-m-h} : \\
Let $q = cl$ or $clb^d$ and set $\ell$ and $\ell'$ as follows;  \\
\hsh for $q = cl$ : \hsh $\ell := 3k+5$ \ in [I], \hsh $\ell := 2k+7$ \ in [I\hspace{-0.1mm}I], \hsh $\ell' := 3$ \\[2mm] 
\hsh for $q = clb^d$ : \hsh 
$\bary[t]{@{}l}
\ell := 
\left\{ \bary[c]{@{\,}l@{ \ }l@{ \ }l} 
2\big(c({\cal H}) + (k-1)c({\cal H}^{(m)})\big) + 2k + 4 & \leq 2(m+2)(k+1) & \text{in \,[I]} \\[2.5mm]
2\,\big(c({\cal H}) + (k-1)c({\cal H}^{(m)})\big)  + k + 6 & \leq (2m+3)(k+1) + 3 \hsh & \text{in [I\hspace{-0.1mm}I]}
\eary\right. \\[6mm] 
\ell' := 2c({\cal H})-1 \leq 4m+1. 
\eary$ \\[2mm] 
Since $K \subset N_0' \in {\cal S}{\cal M}_c(M; {\cal H})$, we have $|{\cal S}| \Subset N_0'$. Hence, we can apply 
Proposition~\ref{cl_even_cpt_bd_handle} to 
$M_2 \Subset N_0' \Subset N_0''$, ${\cal S}$ and $\{ {\cal S}_j \}_{j=1}^k$ 
 so to obtain the following estimate.  
\bit 
\itemI $q\hspace{0.3mm}d({\rm Diff}^r(M, M_{M_2})_0, {\rm Diff}^r(M, M_{N_0''})_0) \leq \ell$ 
\eit 
Since $N_i''$ and $L_i''$ $(i \geq 1)$ do not intersect $K$, 
by Proposition~\ref{prop_2m_cpt_no-m-h} and (1)(ii), (iii) it follows that 

\bit 
\itemII $q\hspace{0.3mm}d\,({\rm Diff}^r(M, M_{N_i'})_0,  {\rm Diff}^r(M, M_{N_i''})_0) \leq \ell' \leq \ell$ \ $(i \geq 1)$, \\ 
$q\hspace{0.3mm}d\,({\rm Diff}^r(M, M_{L_i'})_0,  {\rm Diff}^r(M, M_{L_i''})_0) \leq \ell'$ \ $(i \geq 1)$. 
\eit 
Therefore, from Lemma~\ref{lem_factorization_open-2}\,[I\hspace{-0.1mm}I] 
it follows that 

\bit
\itemiii $q(G_1) \leq \ell$, $q(H_1) \leq \ell'$ \ and \ $q(f) \leq \ell + \ell'$ \ in ${\rm Diff}^r(M)_0$. 
\eit 
\eenum 
\vskip -7mm 
\end{proof} 

\noindent {\bf [2] Covering spaces of closed $2m$-manifolds} 
\begin{setting}\label{setting_even_covering}
Suppose $\pi : \widetilde{M} \to M$ is a $C^\infty$ covering space 
over a closed $2m$-manifold $M$ $(m \geq 3)$ and $1 \leq r \leq \infty$, $r \neq 2m+1$. 
\end{setting}

\begin{setting_5.7+}
Suppose ${\cal H}$ is a handle decomposition of $M$, 
${\cal S}$ is the set of open $m$-cells of $P_{\cal H}$ and 
$\{ {\cal S}_j \}_{j=1}^k$ $(k \geq 1)$ is a finite cover of ${\cal S}$ which satisfies the following condition : 
\bit 
\item[$(\dagger)$] Each $Cl_M |{\cal S}_j|$ $(j = 1, \cdots, k)$ has an open neighborhood $U_j$ in $M$ which is evenly covered by $\pi$.   
\eit 
\end{setting_5.7+}

\begin{thm}\label{thm_even_covering}
In Setting~{\rm \ref{setting_even_covering}, ~ \ref{setting_even_covering}$^+$} : 
\benum 
\item[{\rm [I]}\,] Suppose $Cl_M|{\cal S}_{j}|$ is strongly displaceable from itself in $M$ for each $j=1, \dots, k$.  
\vskip 1mm 
\bit 
\itemI $cld\,{\rm Diff}^r_c(\widetilde{M})_0 \leq 3k+5$, \\[1mm] 
$clb^fd\,{\rm Diff}^r_c(\widetilde{M})_0 \leq 2\big(c({\cal H}) + (k-1)c({\cal H}^{(m)}) + k + 2\big) \leq 2(m+2)(k+1)$. 
\vskip 1mm 
\itemII  
$cld\,{\rm Diff}^r(\widetilde{M})_0 \leq 6k+10$, \\[1mm] 
$clb^dd\,{\rm Diff}^r(\widetilde{M})_0 \leq 4\big(c({\cal H}) + (k-1)c({\cal H}^{(m)}) + k + 2\big) \leq 4(m+2)(k+1)$. 
\eit 
\vskip 1mm 
\item[{\rm [II]}] Suppose $Cl_M|{\cal S}_{j}|$ is strongly displaceable from $Cl_M|{\cal S}|$ in $M$ for each $j=1, \dots, k$. 
\vskip 1mm 
\bit 
\itemI $cld\,{\rm Diff}^r_c(\widetilde{M})_0 \leq 2k+7$, \\[1mm] 
$clb^fd\,{\rm Diff}^r_c(\widetilde{M})_0 \leq 2\,\big(c({\cal H}) + (k-1)c({\cal H}^{(m)})\big)  + k + 6 \leq (2m+3)(k+1) + 3$. 
\vskip 1mm 
\itemII 
$cld\,{\rm Diff}^r(\widetilde{M})_0 \leq 4k+14$, \\[1mm] 
$clb^dd\,{\rm Diff}^r(\widetilde{M})_0 \leq 4\big(c({\cal H}) + (k-1)c({\cal H}^{(m)}) + 2k + 12 \leq 2(2m+3)(k+1) + 6$. 
\eit 
\eenum 
\end{thm} 

\begin{cor}\label{cor_even_covering} 
In Setting~{\rm \ref{setting_even_covering}} : Suppose $M$ has a $C^\infty$ triangulation with at most $\ell$ $m$-simplices.  
\bit 
\itemI $cld\,{\rm Diff}^r_c(\widetilde{M})_0 \leq 2\ell+7$ \ \ and \ \  
$clb^fd\,{\rm Diff}^r_c(\widetilde{M})_0 \leq (2m+3)(\ell+1) + 3$. 
\vskip 1mm 
\itemII  
$cld\,{\rm Diff}^r(\widetilde{M})_0 \leq 4\ell+14$ \ \ and \ \  
$clb^dd\,{\rm Diff}^r(\widetilde{M})_0 \leq 2(2m+3)(\ell +1) + 6$. 
\eit 
\end{cor}

\begin{proof}[\bf Proof of Theorem~\ref{thm_even_covering}] \mbox{} 
We apply  Setting~\ref{setting_even_open_handle}, ~ \ref{setting_even_open_handle}$^+$ and 
Theorem~\ref{thm_even_open_handle} in $\widetilde{M}$. 
\benum[(1)] 
\item For Setting~\ref{setting_even_open_handle} : \\
The lifts and inverse images in $\widetilde{M}$ of various objects on $M$  are denoted by the symbol \ $\widetilde{}$\ . 
Then we have the following data on $\widetilde{M}$ : \\[1mm] 
\hspace{25mm} $\widetilde{\cal H}$, $\widetilde{\cal H^\ast}$, $\widetilde{\cal S}$, $\widetilde{\cal S}_j$ \ \ and \ \ 
$\widetilde{P} := \pi^{-1}(P)$, $\widetilde{Q} := \pi^{-1}(Q)$, $\widetilde{U}_j := \pi^{-1}(U_j)$. \\[1mm] 
Here, $\widetilde{\cal H}$ and $\widetilde{\cal H^\ast}$ are the handle decompositions of $\widetilde{M}$ consisting of the lifts of handles of ${\cal H}$ and ${\cal H}^\ast$ respectively, 
$\widetilde{\cal S}$ and $\widetilde{\cal S}_j$ are the sets of lifts of open $m$-cells in ${\cal S}$ and ${\cal S}_j$ respectively.  
It follows that 
\bit 
\itemI $\widetilde{P}$ is the $m$-skeleton of the core complex of $\widetilde{\cal H}$, \\ 
$\widetilde{\cal S}$ is the set of all open $m$-cells of $\widetilde{P}$, \\
$\widetilde{Q}$ is the $m$-skeleton of the core complex of $\widetilde{\cal H}^\ast$ and $\widetilde{\cal H^\ast} = \big(\widetilde{\cal H}\big)^\ast$,
\vskip 1mm 
\itemII $\widetilde{\cal S} = \bigcup_{j=1}^k \widetilde{\cal S}_j$, \ \ $\big|\widetilde{\cal S}\big| = \widetilde{|{\cal S}|}$, \ \ 
$\widetilde{|{\cal S}_j|} = |\widetilde{\cal S}_j|$, \ \ 
$\widetilde{{\cal S}_j^\ast} = (\widetilde{\cal S}_j)^\ast$, \ \ 
$\widetilde{{|{\cal S}_j^\ast}|} = \big|\widetilde{{\cal S}_j^\ast}\big| = \big|\big(\widetilde{\cal S}_j\big)^\ast\big|$. 
\eit  
\vskip 2mm 
\item For Setting~\ref{setting_even_open_handle}$^+$ : \\
We can take $U_j$ so that there exists an open neighborhood $U_j'$ of $Cl_M U_j$ in $M$ which is evenly covered by $\pi$. 
Then $\widetilde{U}_j' = \pi^{-1}(U_j')$ has 
a (at most countable) disjoint open cover $\{ \widetilde{U}_{j, \lambda_j}' \}_{\lambda_j \in \Lambda_j}$ 
such that the restrictions $\pi : \widetilde{U}_{j, \lambda_j}' \to U_j'$ $(\lambda_j \in \Lambda_j)$ are diffeomorphisms. 
Let $\widetilde{U}_{j, \lambda_j} := \widetilde{U}_j \cap \widetilde{U}_{j, \lambda_j}'$ $(\lambda_j \in \Lambda_j)$ 
and $\widetilde{\cal S}_{j, \lambda_j} := \{ \sigma \in \widetilde{S}_j \mid \sigma \subset \widetilde{U}_{j, \lambda_j} \}$ 
(the lift of ${\cal S}_j$ on $\widetilde{U}_{j, \lambda_j}$). Then, the following holds for each $j = 1, \cdots, k$. 
\vskip 1mm 
\bit 
\itemI  
\bit 
\itema $\{ \widetilde{U}_{j, \lambda_j} \}_{\lambda_j \in \Lambda_j}$ is a discrete family of 
relatively compact open subsets of $\widetilde{M}$. 
\itemb $\pi : \widetilde{U}_{j, \lambda_j} \to U_j$ is a diffeomorphism. \hsh (c)\ $Cl_M |{\cal S}_j| \subset U_j$. 
\itemd $U_j \cap Cl_M|{\cal S}| = Cl_{U_j} (|{\cal S}| \cap U_j)$ \ and \ 
$\widetilde{U}_{j, \lambda_j} \cap Cl_{\widetilde{M}}|\widetilde{\cal S}| 
=  Cl_{\widetilde{U}_{j, \lambda_j}}(|\widetilde{\cal S}| \cap \widetilde{U}_{j, \lambda_j})$. 
\eit 
\vskip 1mm  
\itemII 
\bit 
\itema $\widetilde{\cal S}_{j, \lambda_j}$ is a finite subset of $\widetilde{\cal S}_j$ and 
$\{ \widetilde{\cal S}_{j, \lambda_j} \}_{\lambda_j \in \Lambda_j}$ is a disjoint covering of $\widetilde{\cal S}_j$. 
\itemb $Cl_{\widetilde{M}}|\widetilde{\cal S}_{j, \lambda_j}| \subset \widetilde{U}_{j, \lambda_j}$ and 
$\{ Cl_{\widetilde{M}}\,|\widetilde{\cal S}_{j, \lambda_j}| \}_{\lambda_i \in \Lambda_j}$ is a discrete family of compact subsets in $M$. 
\itemc $\pi(Cl_{\widetilde{M}}|\widetilde{\cal S}_{j, \lambda_j}|) = Cl_M|{\cal S}_{j}|$ \ and \ 
$\pi(\widetilde{U}_{j, \lambda_j} \cap Cl_{\widetilde{M}}|\widetilde{\cal S}|) = U_j \cap Cl_M|{\cal S}|$. 
\eit 
\vskip 1mm 

\itemiii From (i)(b) and (ii)(c) it follows that 
\bit 
\itema $Cl_{\widetilde{M}}|\widetilde{\cal S}_{j, \lambda_j}|$ is strongly displaceable from itself in $\widetilde{M}$ in the case [I]. 
\itemb $Cl_{\widetilde{M}}|\widetilde{\cal S}_{j, \lambda_j}|$ is strongly displaceable from $Cl_M|\widetilde{\cal S}|$ in $\widetilde{M}$ in the case [II]. 
\eit 
\itemiv 
Hence, for any finite subset $\Delta \subset \Lambda_j$ 
the disjoint finite union $Cl_{\widetilde{M}}\, |\widetilde{\cal S}_{j, \Delta}| 
= \bigcup_{ \lambda_j \in \Delta} Cl_{\widetilde{M}}\,|\widetilde{\cal S}_{j, \lambda_j}|$
has the same property as in (iii), that is, 
\bit 
\itema $Cl_{\widetilde{M}}\, |\widetilde{\cal S}_{j, \Delta}|$ is strongly displaceable from itself in $\widetilde{M}$ in the case [I]. 
\itemb $Cl_{\widetilde{M}}\, |\widetilde{\cal S}_{j, \Delta}|$ is strongly displaceable from $Cl_{\widetilde{M}}|\widetilde{\cal S}|$ in $\widetilde{M}$ in the case [II]. 
\eit 
\eit 
\eenum 
\vskip 1mm 

Since $c(\widetilde{\cal H}) = c({\cal H})$ and $c(\widetilde{\cal H}^{(m)}) = c({\cal H}^{(m)})$, the conclusions now follow from Theorem~\ref{thm_even_open_handle}. 
\end{proof}

\noindent {\bf [3] Infinite sums of finitely many compact $2m$-manifolds}

\begin{setting}\label{setting_inf_sum}
Consider a finite family ${\cal N} = \{ (N_i, {\cal T}_i) \}_{i=0}^\ell$, 
where each $N_i$ is a compact $2m$-manifold such that 
$\partial N_i$ (possibly empty) is a disjoint union of two $(2m-1)$-manifolds $\partial_\pm N_i$
and ${\cal T}_i$ is a $C^\infty$ triangulation of $N_i$. 
Assume that $\partial_- N_0 = \emptyset$. 
Let ${\cal C}_i$ denote the set of $m$-simplices of ${\cal T}_i$ for $i=0,1, \cdots, \ell$.
\end{setting}

We consider the class of open $2m$-manifolds which are infinite sums of compact manifolds in ${\cal N}$. 

\begin{setting_5.8+}
Suppose $M$ is an open $2m$-manifold ($m \geq 3$), 
$1 \leq r \leq \infty$, $r \neq 2m+1$, 
$\{ M_k \}_{k \geq 0}$ is an exhausting sequence of $M$ and   
${\cal T}$ is a $C^\infty$ triangulation of $M$ such that each $M_k$ underlies a subcomplex of ${\cal T}$.  
Let ${\cal S}$ denote the set of $m$-simplices of ${\cal T}$ and 
if $N \subset M$ underlies a subcomplex of ${\cal T}$, then ${\cal S}|_N$ denotes the set of $m$-simplices of ${\cal T}|_N$. 

Let $L_k := M_{k-1, k}$ $(k \geq 0)$ and set $\partial_- L_k := \partial M_{k-1}$ and $\partial_+ L_k := \partial M_k$. 
Suppose 
\bit 
\item[$(\alpha)$] 
$L_0 \equiv M_0$ admits a $C^\infty$ diffeomorphism $\phi_{0} : 
\big(L_{0}, \partial L_0, {\cal T}|_{L_{0}}\big) 
 \approx (N_0, \partial N_0, {\cal T}_0)$ which is a simplicial isomorphism and 

\item[$(\beta)$]  each $L_k$ $(k \geq 1)$ is a disjoint union of compact $2m$-manifolds $L_{k, s}$ $(s = 1, \cdots, n_k)$ and 
each $L_{k, s}$ admits some $i \equiv i(k,s) \in \{ 1, \cdots, \ell \}$ and  a $C^\infty$ diffeomorphism \\
\hspace*{20mm} $\phi_{k,s} : 
\big(L_{k, s}, L_{k, s} \cap \partial_- L_k, L_{k, s} \cap \partial_+ L_k, {\cal T}|_{L_{k, s}}\big) 
 \approx (N_i, \partial_- N_i, \partial_+ N_i, {\cal T}_i)$, \\ 
which is a simplicial isomorphism. 
\eit 
\end{setting_5.8+}

In this case we say that $M$ (or $(M, \{ M_k \}_{k \geq 0}, {\cal T})$) is an infinite sum of the model manifolds ${\cal N} = \{ (N_i, {\cal T}_i) \}_{i=0}^\ell$. 

\begin{setting_5.8++}
Suppose ${\cal C}_0$ has a finite cover $\{{\cal C}_{0,j} \}_{j=1}^{p+q}$ and 
each ${\cal C}_i$ $(i=1, \cdots, \ell)$ has a finite cover $\{{\cal C}_{i,j} \}_{j=1}^{p}$, where $p \geq 1$ and $q \geq 0$. 
We assume that \ for each $j =1, \cdots, p$ \\
\hsh  the family $\{{\cal C}_{i,j} \}_{i=0}^\ell$ satisfies the following condition : 
\bit 
\item[$(\ast)$] if $|{\cal C}_{i_0,j}| \cap \partial_+ N_{i_0} \neq \emptyset$ for some $i_0 = 0,1, \cdots, \ell$, then 
$|{\cal C}_{i,j}| \cap \partial_- N_i = \emptyset$ for any $i = 1, \cdots, \ell$. 
\eit 
\end{setting_5.8++}

\begin{prop}\label{prop_inf_sum} 

In Settings~{\rm \ref{setting_inf_sum}, ~\ref{setting_inf_sum}$^+$, ~\ref{setting_inf_sum}$^{++}$}\!\! : 
\benum[\ (1)] 
\item[{\rm [I]}\,] 
\btab[t]{@{}ll}
$cld\,{\rm Diff}^r(M)_0 \leq 3(2p+q) + 10$, \hsf & $clb^dd\,{\rm Diff}^r(M)_0 \leq 2(m+2)(2p+q+2)$, \\[2mm] 
$cld\,{\rm Diff}^r_c(M)_0 \leq 3(p+q) + 5$, & $clb^fd\,{\rm Diff}_c^r(M)_0 \leq 2(m+2)(p+q+1)$, 
\etab 
\vskip 1mm  
if \btab[t]{c@{\ }l}
{\rm (i)} & each $|{\cal C}_{0,j}|$ $(j=1, \cdots, p+q)$ is strongly displaceable from $|{\cal C}_{0,j}| \cup (\partial N_0 \times [0,1))$ in 
$\widetilde{N}_0$ \ and \\[2mm]
{\rm (ii)} & each $|{\cal C}_{i,j}|$ $(i = 1, \cdots, \ell, j=1, \cdots, p)$ is strongly displaceable from $|{\cal C}_{i,j}| \cup (\partial N_i \times [0,1))$ in $\widetilde{N}_i$.  
\etab 
\vskip 2mm 
\item[{\rm [II]}] 
\btab[t]{@{}ll}
$cld\,{\rm Diff}^r(M)_0 \leq 2(2p+q) + 14$, \hsf &  
$clb^dd\,{\rm Diff}^r(M)_0 \leq (2m+3)(2p+q+2) + 6$, \\[2mm] 
$cld\,{\rm Diff}_c^r(M)_0 \leq 2(p+q) +7$, & 
$clb^fd\,{\rm Diff}_c^r(M)_0 \leq (2m+3)(p+q+1) + 3$,  
\etab 
\vskip 1.5mm 
if \btab[t]{c@{\ }l}
{\rm (i)} & each $|{\cal C}_{0,j}|$ $(j=1, \cdots, p+q)$ is strongly displaceable from $|{\cal C}_0| \cup (\partial N_0 \times [0,1))$ in 
$\widetilde{N}_0$ \ and \\[2mm] 
{\rm (ii)} & each $|{\cal C}_{i,j}|$ $(i = 1, \cdots, \ell, j=1, \cdots, p)$ is strongly displaceable from $|{\cal C}_{i}| \cup (\partial N_i \times [0,1))$ in $\widetilde{N}_i$.  
\etab 
\eenum 
\end{prop} 

\begin{cor}\label{cor_inf_sum} In Settings~{\rm \ref{setting_inf_sum}, ~\ref{setting_inf_sum}$^+$}: \\ 
\hsp 
\btab[t]{@{}ll}
$cld\,{\rm Diff}^r(M)_0 \leq 2(2a+b) + 14$, & 
$clb^dd\,{\rm Diff}^r(M)_0 \leq (2m+3)(2a+b+2) + 6$, \\[2mm] 
$cld\,{\rm Diff}_c^r(M)_0 \leq 2(a+b) +7$, & 
$clb^fd\,{\rm Diff}_c^r(M)_0 \leq (2m+3)(a+b+1) + 3$,  
\etab 
\vskip 2mm 
\hsp  
if \ \ $(\natural)$ 
\btab[t]{@{\,}c@{ \ }l}
{\rm (i)} & $a_1, a_2 \geq 0$, $a := a_1 + a_2 \geq 1$, $b \geq 0$, \\[2mm] 
{\rm (ii)} & $\#{\cal C}_0 \leq a+b$, \hsf $\# \{ \sigma \in {\cal C}_0 \mid \sigma \cap \partial N_0 \neq \emptyset \} \leq a_2 + b$, \\[2mm] 
{\rm (iii)} & for each $i=1, \cdots, \ell$ \\[2mm] 
& $\#{\cal C}_i \leq a$, \hsf $\#\{ \sigma \in {\cal C}_i \mid \sigma \cap \partial_- N_i \neq \emptyset \} \leq a_1$,  \hsf 
$\#\{ \sigma \in {\cal C}_i \mid \sigma \cap \partial_- N_i = \emptyset \} \leq a_2$, \\[2mm] 
& if $\sigma \in {\cal C}_i$ and $\sigma \cap \partial_+ N_i \neq \emptyset$, then $\sigma \cap \partial_- N_i = \emptyset$. 
\etab 
\end{cor}

Note that the condition $(\natural)$ is obviously satisfied if $a_1, a_2 \geq \max \{ \#{\cal C}_i \mid i=0,1, \cdots, \ell \}$. 

\begin{proof}[\bf Proof of Proposition~\ref{prop_inf_sum}] We apply Setting~\ref{setting_exhausting_seq} and Proposition~\ref{prop_exhausting seq}. 
\benum[\ (1)] 
\item For Setting~\ref{setting_exhausting_seq} : By Setting~\ref{setting_inf_sum}$^+$\,$(\alpha)$, $(\beta)$ 
\bit 
\itemI ${\cal S}_0 := {\cal S}|_{L_0}$ inherits a finite cover $\{ {\cal S}_{0, j} \}_{j=1}^{p+q}$ 
corresponding with the finite cover $\{ {\cal C}_{0,j} \}_{j=1}^{p+q}$ of ${\cal C}_0$,  
\itemII for each $k \geq 1$ 
\bit 
\itema each ${\cal S}|_{L_{k,s}}$ $(s = 1, \cdots, n_k)$ inherits a finite cover $\{ {\cal S}_{{k, s}, j} \}_{j=1}^p$ 
corresponding with the finite cover $\{ {\cal C}_{i,j} \}_{j=1}^p$ of ${\cal C}_i$

\itemb ${\cal S}_k := {\cal S}|_{L_k} = \bigcup_{s = 1}^{n_k} {\cal S}|_{L_{k,s}}$ has a finite cover $\{ {\cal S}_{{k}, j} \}_{j=1}^p$ defined by 
${\cal S}_{{k}, j} := \bigcup_{s = 1}^{n_k} {\cal S}_{{k, s}, j}$. 
\eit
\eit 
From Setting~\ref{setting_inf_sum}$^{++}$\,$(\ast)$ it follows that $(\natural)$
$\{ |{\cal S}_{k,j}| \}_{k \geq 0}$ is a disjoint family for each $j=1, \cdots, p$. 

\item The assumptions in [I] and [II] imply the corresponding conditions in [I] and [II] in Proposition~\ref{prop_exhausting seq}.  
Therefore, the conclusions follow from Proposition~\ref{prop_exhausting seq}.  
\eenum 
\vskip -8mm 
\end{proof}

\begin{proof}[\bf Proof of Corollay~\ref{cor_inf_sum}] 
We apply Setting~\ref{setting_inf_sum}$^{++}$ and Proposition~\ref{prop_inf_sum} to the following situation : 
\benum[(1)] 
\item We use the following notations : \\
\hsh ${\cal C}_0 = \{ \alpha_{1}, \cdots, \alpha_{g}, \beta_{1}, \cdots, \beta_{h} \}$, \\
\hsp \hsh where \ 
$\alpha_{u} \cap \partial N_0 = \emptyset$ \ $(u=1, \cdots , g)$ \ and \ 
$\beta_{v} \cap \partial N_0 \neq \emptyset$ \ $(v=1, \cdots , h)$, \\
\hsh ${\cal C}_i = \{ \sigma_{i,1}, \cdots, \sigma_{i,x_i}, \tau_{i,1}, \cdots, \tau_{i,y_i}\}$ \ \ $(i=1, \cdots, \ell)$, \\
\hsp \hsh where \ 
$\sigma_{i,u} \cap \partial_- N_i \neq \emptyset$ \ $(u=1, \cdots , x_i)$ \ and \ 
$\tau_{i,v} \cap \partial_- N_i = \emptyset$ \ $(v=1, \cdots , y_i)$. \\
It follows that $g + h \leq a+b$, $h \leq a_2 + b$ and $x_i \leq a_1$, $y_i \leq a_2$ $(i=1, \cdots, \ell)$. 
\vskip 1mm 
\item For Setting~\ref{setting_inf_sum}$^{++}$ : 
Define the finite covers $\{{\cal C}_{0,j} \}_{j=1}^{a+b}$ of ${\cal C}_0$ and $\{{\cal C}_{i,j} \}_{j=1}^{a}$ of ${\cal C}_i$ $(i=1, \cdots, \ell)$ by \\[2mm] 
\hsf ${\cal C}_{0,j} = 
\left\{ \bary[c]{@{ \ }cl}
\{ \alpha_j \} & (j=1, \cdots, g) \\[1mm] 
\emptyset & (j = g+1, \cdots, a+b - h) \\[1mm] 
\{ \beta_t \} & (j=a+b - h + t) 
\eary \right.$ 
\hsp 
${\cal C}_{i,j} = 
\left\{ \bary[c]{@{ \ }cl}
\{ \sigma_{i,j} \} & (j=1, \cdots, x_i) \\[1mm] 
\emptyset & (j = x_i+1, \cdots, a - y_i) \\[1mm] 
\{ \tau_{i,t} \} & (j=a - y_i + t) 
\eary \right.$ \\
\hfill $(i=1, \cdots, \ell)$. \hspace*{3mm} \mbox{} \\
The conditions $g \leq a+b - h$ and $x_i \leq a-y_i$ assure the well-definedness. 
The condition $(\ast)$ follows from the assumptions (i)\,$\sim$\,(iii). 
In fact, if $|{\cal C}_{i_0,j}| \cap \partial_+ N_{i_0} \neq \emptyset$, then $j > a_1 \geq x_i$. 
These covers obviously satisfy the condition [II] in Proposition~\ref{prop_inf_sum}, from which follows the conclusion. 
\eenum 
\vskip -8mm 
\end{proof}
 
\begin{example}\label{exp_inf_conn_sum} 
A typical example of ``an infinite sum of finitely many compact $n$-manifolds''  
is the class of ``an infinite connected sum $M$ of a closed $n$-manifold $N$''. 
We allow any number of branches in each step of the connected sum, so that 
$M$ permits any 0-dimensional compact metrizable space as its topological end. 
Since these branches are realized by iteration of two branches, 
$M$ is represented as an infinite sum of 
model manifolds ${\cal N} = \{ N_0, N_1, N_2, N_3 \}$, where 
$N_0$, $N_1$ are $n$-disks, $N_2$ is an $n$-disk with two open $n$-disks being removed and 
$N_3$ is $N$ with $2$ open $n$-disks being removed, 
Here, $\partial_- N_0 = \emptyset$ and $\partial_- N_i =$ an $(n-1)$-sphere $(i=1,2,3)$. 
\end{example}

\begin{example}\label{exp_inf_conn_sum} 
Another example   
is given by the class of ``the complement $M$ of the intersection of 
a nested sequence $\{ C_k \}_{k=0}^\infty$ of compact $n$-submanifolds in a closed $n$-manifold $N$''. 
This means that \\
\hspace*{18mm} $M = N - C_\infty$ \ \ and \ \ 
$\ds N \equiv C_0 \Supset C_1 \Supset \cdots \Supset C_k \Supset C_{k+1} \Supset \cdots \supset C_\infty := \mbox{\small $\ds \bigcap_{k=0}^\infty$} C_k$. \\
The open $n$-manifold $M$ has the exhausting sequence $M_k := N_{C_k}$ $(k \geq 1)$, whose $k$-th part is given by 
$L_k := M_{k-1,k} = (C_{k-1})_{C_k}$ $(k \geq 1)$. 
If there exist finitely many model embeddings between compact $n$-manifolds $F_i \subset {\rm Int}\,E_i \subset E_i$ $(i=1, \cdots, \ell)$ and 
if each $C_k$ $(k \geq 0)$ is a disjoint union of compact $2m$-manifolds $C_{k, s}$ $(s = 1, \cdots, t_k)$ each of which admits 
a $C^\infty$ diffeomorphism $(C_{k, s}, C_{k, s} \cap C_{k+1}) \approx (E_i, F_i)$ onto some pair $(E_i, F_i)$, 
then $M$ is expressed as an infinite sum of finitely many compact $n$-manifolds $L_0$ and $(E_i)_{F_i}$ $(i=1, \cdots, \ell)$.  
\end{example}

\section{Conjugation-generated norm on diffeomorphism groups}
\subsection{Conjugation-generated norm} \mbox{} 

First we recall basic facts on conjugation-generated norms on groups \cite{BIP}. 
Suppose $G$ is a group. Let $G^\times := G - \{ e \}$. 
For $g \in G$ let $C(g)$ denote the conjugacy class of $g$ in $G$ 
and let $C_g := C(g) \cup C(g^{-1})$. 
Since $C_g$ is a symmetric and conjugation invariant subset of $G$, it follows that $N(g) = N(C_g) = \bigcup_{k\geq 0}(C_g)^k$ 
and we obtain the ext.~conj.-invariant norm $q_{G, C_g}$ on $G$. 
This norm is denoted by $\nu_g$ and called the conjugation-generated norm with respect to $g$. 
Note that $C_h = C_g$ and $\nu_h = \nu_g$ for any $h \in C_g$. 
We also consider the quantity \\[1.5mm] 
\hspace*{20mm} 
$\nu(G) := \min \{ k \in \IZ_{\geq 0} \cup \{ \infty \} \mid \nu_g(f) \leq k \text{ for any $g \in G^\times$ and $f \in G$}\}.$ \\[0.5mm] 
A group $G$ is called \emph{uniformly simple} (\cite{Tsu09}) if  $\nu(G) < \infty$, that is, there is $k \in \IZ_{\geq 0}$ such that for any $f \in G$ and $g \in G^\times$, $f$ can be expressed as a product of at most $k$ conjugates of $g$ or $g^{-1}$. 

A group $G$ is called \emph{bounded} if any conj.-invariant norm on $G$ is bounded 
(or equivalently, any bi-invariant metric on $G$ is bounded).
Every bounded perfect group is uniformly perfect. 

\begin{fact}\label{fact_unif-simple_bounded} 
(1)  If $G$ is uniformly simple, then $G$ is simple and $\nu_g$ is  bounded for any $g \in G^\times$. 
\benum 
\item[(2)] If $\nu_g$ is bounded for some $g \in G^\times$, then $G$ is bounded. More precisely, 
if $g \in G^\times$ and $\nu_g \leq k$ for some $k \in \IZ_{\geq 0}$, then 
$q \leq k q(g)$ for any ext.~conj.-invariant norm $q$ on $G$.  
\eenum 
\end{fact}

Next we clarify relations between the conjugation-generated norm 
and the commutator length supported in balls in diffeomorphism groups \cite{BIP, Tsu09}. 
Recall that ${\cal C}(X)$ denotes the set of connected components of a topological space $X$. 

\begin{defn}\label{defn_comp_end-non-trivial} Suppose $M$ is an $n$-manifold possibly with boundary and $g \in {\rm Diff}^r(M)$. We say that 
\benum
\item $g$ is component-wise non-trivial if $g|_U \neq \id_U$ for any $U \in {\cal C}(M)$, 
\item $g$ is component-wise end-non-trivial if \\
\hsh $(\flat)$ 
\btab[t]{@{\,}c@{\ \,}l}
(i) & $g|_U \neq \id_U$ for any $U \in {\cal C}(M) \cap {\cal K}(M)$ and \\[2mm] 
(ii) & $g|_V \neq \id_V$ for any $(U, K, V)$ with 
\btab[t]{@{\,}l} 
$U \in {\cal C}(M) - {\cal K}(M)$, $K \in {\cal K}(U)$ and \\[2mm]
$V \in {\cal C}(U - K)$ such that $Cl_U V$ is not compact. 
\etab 
\etab 
\eenum 
\end{defn}

\begin{compl}\label{compl_shrap=flat} 
The condition $(\flat)$ is equivalent to the following practical condition $(\sharp)$ :
\bit 
\item[$(\sharp)$] There exists an exhausting sequence $M_i \in {\mathcal S}{\mathcal M}_c(M)$ $(i \geq 1)$ in $M$ for which \\
$g|_L \neq \id_L$ for any $L \in \bigcup_{i \geq 1} {\cal C}(M_{i-1, i})$. 
\eit 
\end{compl} 

Compliment~\ref{compl_shrap=flat} is verified by a routine argument on compact $n$-submanifolds of $M$. 
Note that any exhausting sequence $\{ M_i \}_{i \geq 1}$ in $M$ admits $g \in {\rm Diff}(M,\partial)_0$ such that 
$g|_L \neq \id_L$ for any $L \in \bigcup_{i \geq 1} {\cal C}(M_{i-1, i})$.

\begin{setting}\label{setting_C_g^4}
Suppose $M$ is an $n$-manifold possibly with boundary and 
${\cal G}$ is a subgroup of ${\rm Diff}^r(M)_0$. 
For $A \subset M$, let ${\cal G}_A := {\cal G} \cap {\rm Diff}^r(M, M_A)_0$. 
Note that $g({\cal G}_A)g^{-1} = {\cal G}_{g(A)}$ \ for any $g \in {\cal G}$. 
\end{setting}

\begin{lem}\label{lem_C_g^4}  $($cf.~\cite{BIP}, \cite[Lemma 3.1]{Tsu09} et al.$)$ \ In Setting~{\rm \ref{setting_C_g^4}}; 
Suppose $g \in {\cal G}$, $A \subset M$ and $g(A) \cap A = \emptyset$. 
\benum 
\item[{\rm (1)}] $({\cal G}_A)^c \subset C_g^4$ \ in ${\cal G}$. 
More precisely, for any $a, b \in {\cal G}_A$ the following identity holds :  \\[0.5mm] 
\hspace*{10mm} $[a,b] =g\big(g^{-1}\big)^c g^{bc}\big(g^{-1}\big)^b$ \ in \ ${\cal G}$, \hsh 
where $c := a^{g^{-1}} \in {\cal G}_{g^{-1}(A)}$. Note that $cb = bc$.  
\vskip 0.5mm 
\item[{\rm (2)}] $({\cal G}_B)^c \subset C_g^4$ \ in ${\cal G}$, if $B \subset M$ and there exists $h \in {\cal G}$ with $h(B) \subset A$.  
\eenum 
\end{lem} 

\begin{proof} (2) 
Let $k := h^{-1}gh \in {\cal G}$. Then, $k(B) \cap B = \emptyset$ and by (1) we have $({\cal G}_B)^c \subset C_{k}^4 = C_g^4$.  
\end{proof} 

\begin{lem}\label{lem_nu_leq_4clb} Suppose $M$ is an $n$-manifold possibly with boundary and 
$g \in {\rm Diff}^r(M, \partial)_0$ is component-wise non-trivial. 
\benum 
\item[{\rm [I]}\,] $($\cite{BIP}, \cite[Lemma 3.1]{Tsu09} et al.$)$ 
\bit 
\itemI $\nu_g \leq 4 clb^f$ in ${\rm Diff}^r(M, \partial)_0$. 
\itemII $\nu_g \leq 4 clb^f$ in ${\rm Diff}^r_c(M, \partial)_0$ \ if $g \in {\rm Diff}^r_c(M, \partial)_0$ $($and ${\cal C}(M)$ is a finite set$)$. 
\eit 

\item[{\rm [I\hspace{-0.1mm}I]}] $\nu_g \leq 4 clb^d$ in ${\rm Diff}^r(M, \partial)_0$ 
if $g$ is component-wise end-non-trivial.  
\eenum 
\end{lem}

\begin{proof}[\bf Proof of Lemma~\ref{lem_nu_leq_4clb}] \mbox{} \\
\hspace{1mm} {[I]} 
Let ${\cal G} = {\rm Diff}^r(M, \partial)_0$ in (i) and  ${\cal G} ={\rm Diff}_c^r(M, \partial)_0$ in (ii).  

\benum 
\item[(1)] First we show that $({\cal G}_D)^c = {\rm Diff}^r(M,M_D)_0^c \subset C_g^4$ \ for any $D \in {\cal B}_f^r(M)$. \\
For any $U \in {\cal C}(M)$, since $g|_U \neq \id_U$, 
there exists $E_U \in {\cal B}(U)$ such that $g(E_U) \cap E_U = \emptyset$. \break 
Let $E := \bigcup \{E_U \mid U \in {\cal C}(M) \} \in {\cal B}_d(M)$. 
Then, $g \in {\cal G}$ and $g(E) \cap E = \emptyset$. 
Hence, by Lemma~\ref{lem_C_g^4}\,(2) it suffices to find $h \in {\rm Diff}^r_c(M, \partial)_0 \subset {\cal G}$ with $h(D) \subset E$. 
There exists $U_1, \cdots, U_m \in {\cal C}(M)$ such that $D \subset \bigcup_{i=1}^m  U_i $. 
For each $i = 1, \cdots, m$, since $D \cap U_i \in {\cal B}_f^r(U_i)$ and $U_i$ is connected,  
there exists $h_i \in {\rm Diff}_c^r(U_i, \partial)_0$ with $h_i(D \cap U_i) \subset E_{U_i}$. 
Then, $h$ is defined by $h = h_i$ on $U_i$ and $h = id$ on $M - \bigcup_{i=1}^m  U_i$. 

\item[(2)] Given any $f \in {\cal G}$. If $clb^f(f) = \infty$, then the assertion is trivial. 
Suppose $k := clb^f(f) \in \IZ_{\geq 0}$. 
Then $f = f_1 \cdots f_k$ for some $D_i \in {\cal B}_f^r(M)$ and $f_i \in {\rm Diff}^r(M, M_{D_i})_0^c$ $(i=1, \cdots, k)$. 
Since $\nu_g(f_i) \leq 4$ $(i=1, \cdots, k)$, it follows that $\nu_g(f) \leq 4k$ 
\eenum 
\vskip 1mm 
\noindent \hspace{1mm} {[II]} Take an exhausting sequence $\{ M_i\}_{i \geq 1}$ in $M$ 
as in Compliment~\ref{compl_shrap=flat}\,$(\sharp)$.  
Let $L_i := M_{i-1, i}$ $(i \geq 1)$ \break 
\hsh ${\cal C} := \bigcup_{i \geq 1} {\cal C}(L_i)$ and ${\cal G} := {\rm Diff}^r(M, \partial)_0$.  
We show the following claims in order. 
\benum 
\item There exists a family $E_L \in {\cal B}(L)$ $(L \in {\cal C})$ such that $g(E) \cap E = \emptyset$ for $E := \bigcup_{L\in {\cal C}} E_L \in {\cal B}_d(M)$. 

In fact, inductively, we can find a family of points $p_L \in {\rm Int}\,L$ $(L \in {\cal C})$ such that $g(A) \cap A = \emptyset$ 
for $A := \{ p_L \mid L\in {\cal C}\} \subset M$. 
Since $A$ is closed in $M$, there exists an open neighborhood $U$ of $A$ in $M$ with $g(U) \cap U = \emptyset$. 
For each $L \in {\cal C}$ choose an $n$-ball neighborhood $E_L$ of $p_L$ in ${\rm Int}\,L \cap U$.
Then, $\{ E_L \}_{L\in {\cal C}}$ is discrete since $E_L \in {\cal B}(L)$ $(L \in {\cal C})$, and $g(E) \cap E = \emptyset$ since $E \subset U$. 

\vskip 1mm 
\item Given any $D = \bigcup_{j \in J} D_j \in {\cal B}_d^r(M)$ 
(i.e., $D \subset {\rm Int}\,M$ and $\{ D_j \}_{j \in J}$ is a discrete family of $n$-balls in $M$). 
 Then, there exists $h \in {\cal G}$ such that $h(D) \subset E$. 
 
The required $h$ is obtained as the composition $h = h_2h_1$ of $h_1, h_2 \in {\cal G}$ such that \\
\hsp (i) for each $j\in J$ there exists $i(j) \geq 1$ with $h_1(D_j) \Subset L_{i(j)}$ \ and \ (ii) $h_2(h_1(D)) \subset E$. 

\bit 
\itemI There exists a discrete family $\{ C_j \}_{j \in J}$ of $n$-balls in $M$ such that $D_j \Subset C_j \subset {\rm Int}\,M$ $(j \in J)$. 
For each $j\in J$ there exists $i(j) \geq 1$ with $D_j \cap {\rm Int}\,L_{i(j)} \neq \emptyset$. 
Then, we can push each $D_j$ into ${\rm Int}\,L_{i(j)}$ in $C_j$. 
This procedure yields the desired $h_1$. 

\itemII 
For each $L \in {\cal C}$, since $h_1(D) \cap L \in {\cal B}_f^r(L)$ and $L$ is connected, 
there exists $h_L \in {\rm Diff}^r(M, M_{{\rm Int}\,L})_0$ such that $h_L(h_1(D) \cap L) \subset E_L$. 
Then, $h_2$ is defined by $h_2|_L = h_L$ $(L \in {\cal C})$. 
\eit
\vskip 1mm 
\item $({\cal G}_D)^c = {\rm Diff}^r(M,M_D)_0^c \subset C_g^4$ in ${\cal G}$ for any $D \in {\cal B}_d^r(M)$. 

This follows from (2) and Lemma~\ref{lem_C_g^4}\,(2). 

\item $\nu_g \leq 4 clb^d$ in ${\cal G}$. 

Given any $f \in {\cal G}$. If $clb^d(f) = \infty$, then the assertion is trivial. 
Suppose $k := clb^d(f) \in \IZ_{\geq 0}$. 
Then $f = f_1 \cdots f_k$ for some $D_i \in {\cal B}_d^r(M)$ and $f_i \in {\rm Diff}^r(M, M_{D_i})_0^c$ $(i=1, \cdots, k)$. 
Since $f_i \in C_g^4$ $(i=1, \cdots, k)$, it follows that $f \in C_g^{4k}$ and $\nu_g(f) \leq 4k$. 
\eenum 
\vskip -7mm 
\end{proof}

\subsection{Estimates of conjugation-generated norm on diffeomorphism groups} \mbox{}

The estimates on $clb^f$ and $clb^d$ obtained in Sections 4, 5 lead to the following conclusions for $\nu_g$. 

\subsubsection{\bf Odd-dimensional case} \mbox{}

\begin{thm}\label{thm_nu_cpt_bdry_odd} 
Suppose $M$ is a compact $(2m+1)$-manifold possibly with boundary $(m \geq 0)$, $1 \leq r \leq \infty$, $r \neq 2m+2$. 
If $g \in {\rm Diff}^r(M, \partial)_0$ and $g|_U \neq \id_U$ for any $U \in {\cal C}(M)$, then the following holds. 
\benum
\item[{\rm (1)}] $\nu_gd\,{\rm Diff}^r(M, \partial)_0 \leq 4 clb^fd\, {\rm Diff}^r(M, \partial)_0 \leq 16m+24$.  
\item[{\rm (2)}] If $M$ is closed, then \ \ 
$\nu_gd\,{\rm Diff}^r(M)_0 \leq 4 clb^fd\, {\rm Diff}^r(M)_0 \leq 8c({\cal H})+8$ \\
\hspace*{80mm} for any handle decomposition ${\cal H}$ of $M$. 
\eenum 
\end{thm}

\begin{thm}\label{thm_nu_open_odd} Suppose $M$ is an open $(2m+1)$-manifold $(m \geq 0)$, $1\leq r\leq\infty$, $r\neq 2m+2$ 
and ${\cal H}$ is any handle decomposition of $M$. 
\benum
\item[{\rm (1)}] $\nu_gd\,{\rm Diff}^r(M)_0 \leq 4 clb^dd\, {\rm Diff}^r(M)_0\leq 16c({\cal H}) + 16 \leq 32m+48$ \\
\hsp if $g \in {\rm Diff}^r(M)_0$ is component-wise end-non-trivial. 
\vskip 1mm 
\item[{\rm (2)}] $\nu_gd\,{\rm Diff}^r_c(M)_0 \leq 4 clb^fd\,{\rm Diff}_c^r(M)_0 \leq 8c({\cal H})+8 \leq 16m +24$ \\
\hsp if $g \in {\rm Diff}^r_c(M)_0$ is component-wise non-trivial. 
\eenum 
\end{thm} 

Theorem~\ref{thm_nu_cpt_bdry_odd}\,(1), (2) follow from 
Theorem~\ref{thm_cpt_bdry_odd}, Proposition~\ref{cl_odd_cpt_bd_handle}\,[II] and  Lemma 6.2\,[I](i). 
Theorem~\ref{thm_nu_open_odd}\,(1), (2) follow from Theorem~\ref{thm_open_odd} and Lemma 6.2 [II], [I](ii). 

The estimate in Theorem~\ref{thm_nu_cpt_bdry_odd}\,(2) is compared with that in \cite{Tsu09}, that is, \\
\hspppp $\nu_g\,{\rm Diff}^r(M)_0 \leq 16c({\cal H}) + 12$ \ (\cite[Theorem 1.5]{Tsu09}). 

\subsubsection{\bf Even-dimensional case} \mbox{}

\begin{thm}\label{thm_nu_cpt_bdry_even} 
Suppose $M$ is a compact $2m$-manifold possibly with boundary $(m \geq 3)$, 
$1 \leq r \leq \infty$, $r \neq 2m+1$, 
${\mathcal T}$ is a $C^\infty$ triangulation of $M$, 
${\cal S}$ is the set of $m$-simplices of ${\cal T}$, 
${\cal S} \supset {\cal F} \supset \{ \sigma \in {\cal S} \mid \sigma \not \subset \partial M \}$ and $\{ {\cal F}_j \}_{j=1}^k$ is a finite cover of ${\cal F}$ $($as a set$)$. 
If $g \in {\rm Diff}^r(M, \partial)_0$  is component-wise non-trivial,  
then the following holds. 

\benum 
\item[{\rm [I]}\ ] 
$\nu_g d \, {\rm Diff}^r(M, \partial)_0 \leq 4clb^f\!d\,{\rm Diff}^r(M, \partial)_0 \leq 8(m+2)(k+1)$ \\ 
\hspace*{5mm} if each $|{\cal F}_j|$ $(j=1, \dots, k)$ is strongly displaceable from $|{\cal F}_j| \cup (\partial M \times [0,1))$ in $\widetilde{M}$.  
\item[{\rm [I\hspace{-0.2mm}I]}\,] 
$\nu_g d \, {\rm Diff}^r(M, \partial)_0 \leq  4clb^f\!d\,{\rm Diff}^r(M, \partial)_0 \leq 4(2m+3)(k+1) + 12$ \\ 
\hspace*{5mm} if each $|{\cal F}_j|$ $(j=1, \dots, k)$ is strongly displaceable from $|{\cal S}| \cup (\partial M \times [0,1))$ in $\widetilde{M}$. 

\item[{\rm [I\hspace{-0.2mm}I\hspace{-0.2mm}I]}] 
$\nu_g d \, {\rm Diff}^r(M, \partial)_0 \leq  4clb^f\!d\,{\rm Diff}^r(M, \partial)_0 \leq 4(2m+1)(\ell +1) + 20$ \hsh 
for $\ell := \# \{ \sigma \in {\cal S} \mid \sigma \not \subset \partial M \}$. 
\eenum 
\end{thm} 

Recall that $\widetilde{M} := M \cup_{\partial M} (\partial M \times [0,1))$. 
Theorem~\ref{thm_nu_cpt_bdry_even} follows from Theorem~\ref{thm_cpt_bdry_even}, Corollary~\ref{cor_cpt_bdry_even} and  Lemma 6.2\,[I](i). 
When $M$ is a closed $2m$-manifold, the estimate in Theorem~\ref{thm_nu_cpt_bdry_even}\,[I\hspace{-0.2mm}I\hspace{-0.2mm}I] is compared with that in \cite{Tsuboi3}, that is, \\[0.5mm] 
\hspp \hsh $\nu_g\,{\rm Diff}^r(M)_0 \leq 16(\ell+4)m + 12\ell + 28$ \ \ (\cite[Proof of Corollary 1.3]{Tsuboi3} (p.173)). \\[0.5mm] 
For a closed $2m$-manifold with a handle decomposition, we can use Corollary~\ref{cor_closed} to obtain a fine estimate of $\nu_g$. 

In Section 5.3 we obtained 
some estimates on $clb^d$  
for some important classes of open $2m$-manifolds  
(covering spaces, infinite sums etc.)  
The consequence on $\nu_g$ of these results are listed below. 

\begin{prop}\label{prop_2m_open_no-m-h_nu}
Suppose $M$ is an open $2m$-manifold, $1 \leq r \leq \infty$, $r \neq 2m+1$, 
${\cal H}$ is a handle decomposition of $M$, 
$g \in {\rm Diff}^r(M)_0$ is component-wise end-non-trivial and $h \in {\rm Diff}_c^r(M)_0$ is component-wise non-trivial.
\benum
\item[{\rm (1)}] 
If ${\cal H}$ includes no $m$-handles, then \\
\hspp $\nu_gd\,{\rm Diff}^r(M)_0 \leq 4clb^dd\,{\rm Diff}^r(M)_0 \leq 16c({\cal H})+8 \leq 32m+8$, \\
\hspp $\nu_h d\,{\rm Diff}_c^r(M)_0 \leq 4clb^fd\,{\rm Diff}_c^r(M)_0 \leq 8c({\cal H})+4 \leq 16m+4$.

\item[{\rm (2)}] 
Suppose $m \geq 3$, ${\cal H}$ includes only finitely many $m$-handles, 
${\cal S}$ is the set of all open $m$-cells of $P_{\cal H}$ and $\{ {\cal S}_j \}_{j=1}^k$ $(k \geq 1)$ is a finite cover of ${\cal S}$. 

\bit  
\item[{\rm [I]}\,] 
$\nu_g d\,{\rm Diff}^r(M)_0 \leq  4clb^dd\,{\rm Diff}^r(M)_0 \leq 8(m+2)(k+3) - 28$, \\ 
$\nu_h d\,{\rm Diff}^r_c(M)_0 \leq  4clb^fd\,{\rm Diff}^r_c(M)_0 \leq 8(m+2)(k+1)$. \\
\hsh if $Cl_M|{\cal S}_{j}|$ is strongly displaceable from itself in $M$ for each $j=1, \dots, k$. 
\item[{\rm [I\hspace{-0.1mm}I]}] 
$\nu_gd\,{\rm Diff}^r(M)_0 \leq 4clb^dd\,{\rm Diff}^r(M)_0 \leq 4(2m+3)(k+3)-8$, \\ 
$\nu_h d\,{\rm Diff}_c^r(M)_0 \leq 4clb^fd\,{\rm Diff}^r_c(M)_0 \leq 4(2m+3)(k+1) +12$. \\
\hsh if $Cl_M|{\cal S}_{j}|$ is strongly displaceable from $Cl_M|{\cal S}|$ in $M$ for each $j=1, \dots, k$.
\eit 
\eenum 
\end{prop}

\begin{thm}\label{thm_even_covering_nu}
Suppose $\pi : \widetilde{M} \to M$ is a $C^\infty$ covering space 
over a closed $2m$-manifold $M$ $(m \geq 3)$, $1 \leq r \leq \infty$, $r \neq 2m+1$, 
$g \in {\rm Diff}^r(\widetilde{M})_0$ is component-wise end-non-trivial and $h \in {\rm Diff}_c^r(\widetilde{M})_0$ is component-wise non-trivial.
\benum[(1)] 
\item[]\hspace{-6mm} {\rm [A]}\ 
Suppose ${\cal H}$ is a handle decomposition of $M$, 
${\cal S}$ is the set of open $m$-cells of $P_{\cal H}$ and \\
$\{ {\cal S}_j \}_{j=1}^k$ $(k \geq 1)$ is a finite cover of ${\cal S}$ such that 
\bit 
\item[] each $Cl_M |{\cal S}_j|$ $(j = 1, \cdots, k)$ has an open neighborhood in $M$ which is evenly covered by $\pi$. 
\eit 
\eenum 
\benum 
\item[{\rm [I]}\,] If $Cl_M|{\cal S}_{j}|$ is strongly displaceable from itself in $M$ for each $j=1, \dots, k$, then \\ [1mm] 
\btab[t]{c@{ \ }l} 
{\rm (i)} & $\nu_hd\,{\rm Diff}_c^r(\widetilde{M})_0 \leq 4clb^fd\,{\rm Diff}^r_c(\widetilde{M})_0 \leq 8(m+2)(k+1)$, \\[2mm] 
{\rm (ii)} & 
$\nu_gd\,{\rm Diff}^r(\widetilde{M})_0 \leq 4clb^dd\,{\rm Diff}^r(\widetilde{M})_0 \leq 16(m+2)(k+1)$. 
\etab 
\vskip 2mm 
\item[{\rm [II]}] If $Cl_M|{\cal S}_{j}|$ is strongly displaceable from $Cl_M|{\cal S}|$ in $M$ for each $j=1, \dots, k$, then \\[1mm] 
\btab[t]{c@{ \ }l} 
{\rm (i)} & 
$\nu_hd\,{\rm Diff}_c^r(\widetilde{M})_0 \leq 4clb^fd\,{\rm Diff}^r_c(\widetilde{M})_0 \leq 4(2m+3)(k+1) + 12$, 
 \\[2mm]
{\rm (ii)} & 
$\nu_gd\,{\rm Diff}^r(\widetilde{M})_0 \leq 4clb^dd\,{\rm Diff}^r(\widetilde{M})_0 \leq 8(2m+3)(k+1) + 24$. 
\etab 
\eenum 

\vskip 2mm 
\benum[(1)] 
\item[]\hspace{-6mm} {\rm [B]}\ Suppose $M$ has a $C^\infty$ triangulation with at most $\ell$ $m$-simplices.  
\eenum 
\bit 
\item[]
\bit 
\itemI 
$\nu_hd\,{\rm Diff}_c^r(\widetilde{M})_0 \leq 4clb^fd\,{\rm Diff}_c^r(\widetilde{M})_0 \leq 4(2m+3)(\ell+1) + 12$. 
\vskip 1mm 
\itemII 
$\nu_gd\,{\rm Diff}^r(\widetilde{M})_0 \leq 4clb^dd\,{\rm Diff}^r(\widetilde{M})_0 \leq 8(2m+3)(\ell +1) + 24$. 
\eit 
\eit 
\end{thm} 

\begin{setting}\label{setting_inf_sum_nu}
Suppose $M$ is an open $2m$-manifold ($m \geq 3$), $1 \leq r \leq \infty$, $r \neq 2m+1$, 
$g \in {\rm Diff}^r(M)_0$ is component-wise end-non-trivial and $h \in {\rm Diff}_c^r(M)_0$ is component-wise non-trivial.
Assume that 
\bit 
\itemI ${\cal N} = \{ (N_i, {\cal T}_i) \}_{i=0}^\ell$ is a family of compact $2m$-manifolds as in Setting~\ref{setting_inf_sum} and \\
it satisifies the condition $(\ast)$ in Setting~\ref{setting_inf_sum}$^{++}$ with respect to 
some finite covers $\{{\cal C}_{0,j} \}_{j=1}^{p+q}$ of ${\cal C}_0$ and $\{{\cal C}_{i,j} \}_{j=1}^{p}$ of ${\cal C}_i$ $(i=1, \cdots, \ell)$. 

\itemII $(M, \{ M_k \}_{k \geq 0}, {\cal T})$ is an infinite sum of the model manifolds in ${\cal N}$ as in Setting~\ref{setting_inf_sum}$^+$. 
\eit 
\end{setting}

\begin{prop}\label{prop_inf_sum_nu} 
In Setting~{\rm \ref{setting_inf_sum_nu}}:  
\benum[\ (1)] 
\item[{\rm [I]}\,] $\nu_gd\,{\rm Diff}^r(M)_0 \leq  4clb^dd\,{\rm Diff}^r(M)_0 \leq 8(m+2)(2p+q+2)$, \\ 
$\nu_hd\,{\rm Diff}_c^r(M)_0 \leq  4clb^fd\,{\rm Diff}_c^r(M)_0 \leq 8(m+2)(p+q+1)$ \\
\hsp if $\{{\cal C}_{0,j} \}_{j=1}^{p+q}$ and $\{{\cal C}_{i,j} \}_{j=1}^{p}$ $(i=1, \cdots, \ell)$ 
satisfy the condition in Proposition{\rm ~\ref{prop_inf_sum}\,[I]}. 

\item[{\rm [II]}] $\nu_gd\,{\rm Diff}^r(M)_0 \leq  4clb^dd\,{\rm Diff}^r(M)_0 \leq 4(2m+3)(2p+q+2) + 24$ \\
$\nu_hd\,{\rm Diff}_c^r(M)_0 \leq  4clb^fd\,{\rm Diff}_c^r(M)_0 \leq 4(2m+3)(p+q+1) + 12$ \\ 
\hsp if $\{{\cal C}_{0,j} \}_{j=1}^{p+q}$ and $\{{\cal C}_{i,j} \}_{j=1}^{p}$ $(i=1, \cdots, \ell)$ 
satisfy the condition in Proposition{\rm ~\ref{prop_inf_sum}\,[II]}. 

\item[{\rm [III]}] $\nu_gd\,{\rm Diff}^r(M)_0 \leq  4clb^dd\,{\rm Diff}^r(M)_0 \leq 4(2m+3)(2a+b+2) + 24$ \\
$\nu_hd\,{\rm Diff}_c^r(M)_0 \leq  4clb^fd\,{\rm Diff}_c^r(M)_0 \leq 4(2m+3)(a+b+1) + 12$ \\
\hsp if ${\cal C}_0$ and ${\cal C}_i$ $(i=1, \cdots, \ell)$ satisfy the condition $(\natural)$ in Corollary~{\rm \ref{cor_inf_sum}}.  
\eenum 
\end{prop} 

The estimates in Proposition~\ref{prop_2m_open_no-m-h_nu}, Theorem~\ref{thm_even_covering_nu} and Proposition~\ref{prop_inf_sum_nu} 
are immediately obtained by Lemma~\ref{lem_nu_leq_4clb}\,[II] and 
the results in Section 5.3, i.e., Propositions~\ref{prop_2m_open_no-m-h}, ~\ref{prop_2m_open_finite-m-h}, 
Theorem~\ref{thm_even_covering} $+$ Corollay~\ref{cor_even_covering} 
and Proposition~\ref{prop_inf_sum} $+$ Corollary~\ref{cor_inf_sum}, respectively. 

\subsubsection{\bf Some special cases} \mbox{}

The next lemma is the $\nu$-version of Lemma~\ref{basic_lemma_cl}. 

\begin{lem}\label{basic_lemma_nu} 
Suppose $M$ is an $n$-manifold possibly with boundary, $1 \le r \le \infty, \, r \neq n+1$, 
$F \in {\rm Isot}^r_c(M, \partial)_0$ and $f := F_1 \in {\rm Diff}^r_c(M, \partial)_0$. 
Assume that there exist $V, W \in {\cal K}({\rm Int}\,M)$ and $\phi \in {\rm Diff}^r_c(M, \partial)_0$ such that 
$V \supset W$, $\phi(V) \subset W$, ${\rm supp}\,F \subset {\rm Int}_M V - W$ and 
$\phi|_N \neq \id_N$ for any $N \in {\cal C}(M)$. 
Then, $\nu_\phi(f) \leq 6$ in ${\rm Diff}^r_c(M, \partial)_0$. 
\end{lem}  

\begin{proof} 
We take $U \in {\cal K}({\rm Int}\,M)$ such that $V \subset U$ and $\phi \in {\rm Diff}^r(M, M_U)_0$, and 
repeat the proof of Lemma~\ref{basic_lemma_cl}. 
Then, we work in ${\rm Diff}^r_c(M, \partial)_0$. 
It follows that $f = hg_0$, \ $g_0 = [\phi, H^{-1}] = \phi (H^{-1}\phi^{-1}H)$ \ and \ $clb^f(h) \leq 1$. 
Hence, $\nu_\phi(g_0) \leq 2$ and $\nu_\phi(h) \leq 4\,clb^f(h) \leq 4$ by Lemma~\ref{lem_nu_leq_4clb}\,[I](ii). 
This implies that $\nu_\phi(f) \leq 6$ in ${\rm Diff}^r_c(M, \partial)_0$. 
\end{proof}

\begin{example}\label{example_displaceable_2} 
Suppose $M$ is a closed $n$-manifold, $1 \le r \le \infty, \, r \neq n+2$. Then \\
\hsh $\nu_{\phi_\xi}d\,{\rm Diff}^r(M \times I, \partial)_0 \leq 6$ \ \ 
for any $\xi \in {\rm Diff}^r(I, \partial)_0 - \{ \id_I \}$ \ and \ $\phi_\xi := \id_M \times \xi \in {\rm Diff}^r(M \times I, \partial)_0$. 
\end{example} 

This example follows from the same argument as in Example~\ref{example_displaceable}, using Lemma~\ref{basic_lemma_nu}.
The next proposition follows from Lemma~\ref{lem_nu_leq_4clb}\,(II) and Proposition~\ref{prop_product_end}\,(2). 

\begin{prop}\label{prop_product_end_nu} 
Suppose $W$ is a compact connected $n$-manifold  with boundary, $1 \le r \le \infty, \, r \neq n+1$ and 
$g \in {\rm Diff}^r({\rm Int}\,W)_0$ is component-wise end-non-trivial. 
Let $m:= clb^fd\,{\rm Diff}^r(\partial W \times I; \partial)_0 < \infty$. Then, \\  
\hspp $\nu_gd\,{\rm Diff}^r({\rm Int}\,W)_0 \leq 4 clb^dd\,{\rm Diff}^r({\rm Int}\,W)_0 \leq 4\max\{ clb^f d\,{\rm Diff}^r(W; \partial)_0, m \} + 4m$. \end{prop} 

In this case, $g$ is component-wise end-non-trivial if and only if $g$ satisfies the following condition; 
\bit 
\item[$(\dagger)$] $g|_{L \times (0,1]} \neq \id$ for any $L \in {\cal C}(\partial W)$ and 
any collar neighborhood $L \times [0,1]$ of $L = L \times \{ 0 \}$ in $W$. 
\eit 

\subsection{Boundedness and uniform simplicity of diffeomorphism groups} \mbox{}

In this final section we summarize the conclusions on boundedness and uniform simplicity of diffeomorphism groups, 
which are immediately induced from the estimates on $\nu_g$ in Section 6.2 
(together with the estimates on $clb^f$ and $clb^d$ in the previous sections). 
The estimates on $clb^d$ in the previous sections lead us to some boundedness result on the group ${\rm Diff}^r(M)_0$ even for  a non-compact manifold $M$.  
However, the investigation of (uniform) simplicity 
is restricted to the subgroup ${\rm Diff}^r_c(M; \partial)_0$, 
since the diffeomorphism group ${\rm Diff}^r(M)_0$ of any $n$-manifold $M$ includes 
the proper normal subgroups 
(a) ${\rm Diff}^r(M, \partial)_0$ if $\partial M \neq \emptyset$ and 
(b) ${\rm Diff}^r_c(M)_0$ if $M$ is non-compact.  
The next lemma follows from Lemma~\ref{lem_nu_leq_4clb} together with 
Fact~\ref{fact_unif-simple_bounded}\,(2) and Remark~\ref{rmk_clb=clb^f}\,(3). 

\begin{lem}\label{lem_bdd_nu} 
Suppose $M$ is an $n$-manifold possibly with boundary. 
\benum
\item[{\rm (1)}] ${\rm Diff}^r(M, \partial)_0$ is bounded if $clb^d$ is bounded in ${\rm Diff}^r(M, \partial)_0$. 
\item[{\rm (2)}] \bit 
\itemI If $clb^f$ is bounded in ${\rm Diff}_c^r(M, \partial)_0$ and $\#{\cal C}(M) < \infty$, then ${\rm Diff}^r_c(M, \partial)_0$ is bounded.
\itemII If ${\rm Diff}^r_c(M, \partial)_0$ is bounded and $r \neq n+1$, then $clb^f$ is bounded in ${\rm Diff}_c^r(M, \partial)_0$. 
\eit 
\eenum
\end{lem} 
 
The following is the main result in this section, 
which follows immediately from the results in Section 6.2. 

\begin{cor}\label{cor_bounded} \mbox{} Suppose $1 \leq r \leq \infty$, $r \neq n+1$.
\benum  
\item[{\rm [1]}]  Suppose $M$ is a compact $n$-manifold possibly with boundary.  
When $n \neq 2, 4$,  
\bit  
\item[{\rm (i)\,}] ${\rm Diff}^r(M, \partial)_0$ is bounded \ \ and \ \ 
{\rm (ii)} \,${\rm Diff}^r(M, \partial)_0$ is uniformly simple if $M$ is connected. 
\eit  

\item[{\rm [2]}] Suppose $M$ is an open $n$-manifold. Then,  

\bit 
\item[{\rm (i)}\,] ${\rm Diff}^r(M)_0$ is bounded,  
\item[{\rm (ii)}] ${\rm Diff}^r_c(M)_0$ is {\rm (a)} bounded if $\#{\cal C}(M) < \infty$ and {\rm (b)} uniformly simple if $M$ is connected 
\eit 
in the following cases : 
\bit 
\item[{\rm (1)}] $n = 2m+1$ $(m \geq 0)$
\item[{\rm (2)}] $n = 2m$ $(m \geq 1)$ and $M$ satisfies one of the following conditions: 
\bit 
\item[] \hspace{-7mm} for $m \geq 1$ : 
\bit 
\item[{\rm (a)}] $M$ has a handle decomposition without $m$-handles. 
\eit 
\item[] \hspace{-7mm} for $m \geq 3$ : 

\bit 
\item[{\rm (b)}] $M$ has a handle decomposition ${\cal H}$ with only finitely many $m$-handles 
and for which the closure of each open $m$-cell of $P_{\cal H}$ is strongly displaceable from itself in $M$. 

\item[{\rm (c)}] $M$ is a $C^\infty$ covering space over a closed $2m$-manifold. 
\item[{\rm (d)}] $M$ is an infinite sum of finitely many compact $2m$-manifolds $($as in Setting~{\rm \ref{setting_inf_sum_nu}}$)$.
\eit 
\eit 
\eit 
\vskip 2mm 

\item[{\rm [3]}] Suppose $W$ is a compact $n$-manifold  with boundary. Then, \\ 
${\rm Diff}^r({\rm Int}\,W)_0$ is bounded, if $clb^f$ is bounded in ${\rm Diff}^r(W, \partial)_0$ $($equivalently, ${\rm Diff}^r(W, \partial)_0$ is bounded\,$)$.  
\eenum 
\end{cor}


\begin{thebibliography}{99}
        
\bibitem{A-F1}
K.Abe and K.Fukui,
{\it Erratum and addendum to \reflectbox{"}Commutators of $C^{\infty}$-diffeomorphisms preserving a submanifold"}, J. Math. Soc. Japan, {\bf 65-4} (2013), 1329-1336.  

\bibitem{A-F2}
K.Abe and K.Fukui, Characterization of the uniform perfectness of diffeomprphism groups preserving a submanifold, 
Foliations 2012 (World Scientific), (2013), 1-8.

\bibitem{Ba}
A.Banyaga,
{\it The structure of classical diffeomorphisms}, Kluwer Akademic Publisher, Dordrecht, 1997.

\bibitem{BIP}
D.Burago, S.Ivanov and L.Polterovich, 
{\it Conjugation-invariant norms on groups of geomatric origin}, 
Advanced Studies in Pure Math. 52, (2008), Groups of diffeomorphisms, 221-250.

\bibitem{Cu}
W.D.Curtis, The automorphism group of a compact groups action, Trans. Amer. Math. Soc., {\bf 203} (1975), 45-54. 

\bibitem{Ep}
D.B.A. Epstein, 
{\it The simplicity of certain groups of homeomorphisms}, Compositio Math., {\bf 22} (1970), 165-173.

\bibitem{F2}
K.Fukui, 
{\it Commutator length of leaf preserving diffeomorphisms}, Publ. RIMS Kyoto Univ., {\bf 48-3}, (2012), 615-622.

\bibitem{H1}
M.Herman, 
{\it Simplicit\'e du groupe des diff\'eomorphismes de classe $C^{\infty}$, isotopes $\grave{a}$ l'identit\'e, du tore de dimension $n$}, 
C.~R.~Acad. Sci. Paris {\bf 273} (1971), 232-234. 

\bibitem{LMR}
J.Lech, I.Michalik and T.Rybicki,
{\it On the boundedness of equivariant homeomorphism groups}, Opuscula Math., {\bf 38-3} (2018), 395-408. 

\bibitem{L}
W.Ling, 
{\it Normal subgroups of the group of automorphisms of an open manifold that has boundary}, 
Preprint, (1977).

\bibitem{Ma}
J.Mather,
{\it Commutators of diffeomorphisms $I$, $II$ {\rm and} $III$}, 
Comm. Math. Helv., {\bf 49}(1974),512-528, {\bf 50}(1975),33-40 and {\bf 60}(1985),122-124.

\bibitem{Mc}
D. McDuff,
{\it The latice of normal subgroups of the group of diffeomorphisms or homeomorphisms of an open manifold},
J. London Math. Soc., {\bf 18} (1978), 353-364.

\bibitem{R2}
T.Rybicki,
{\it Boundedness of certain automorphism groups of an oepn manifold}, 
Geom. Dedicata, {\bf 151}(2011), 175-186.

\bibitem{S}
P. A. Schweitzer,
{\it Normal subgroups of diffeomorphism and homeomorphism groups of ${\bf R}^n$ and other open manifolds},
Ergodic Theory and Dynamical Systems, {\bf 31}(2011), 1835-1847.

\bibitem{Th}
W.Thurston, 
{\it Foliations and groups of diffeomorphisms}, 
Bull. Amer. Math. Soc., {\bf 80} (1974), 304-307.

\bibitem{Tsuboi1} 
T.Tsuboi, 
{\it On the group of foliation preserving diffeomorphisms}, 
Foliations 2005, ed. by P.Walczak et al., World scientific, Singapore, (2006), 411-430. 

\bibitem{Tsuboi2} 
T.Tsuboi, {\it On the uniform perfectness of diffeomorphism groups}, Advanced Studies in Pure Math. 52, (2008), Groups of diffeomorphisms, 505-524.

\bibitem{Tsu09} T. Tsuboi, {\it On the uniform simplicity of diffeomorphism groups}, In Differential Geometry, World scientific Publ., Hackensack, N.J., , (2009), 43-55.

\bibitem{Tsuboi3} 
T.Tsuboi, {\it On the uniform perfectness of the groups of diffeomorphism groups of even-dimensional manifolds}, 
Comm. Math. Helv., {\bf 87}(2012), 141-185.
\end{thebibliography}
\end{document}